\documentclass[reqno,11pt]{amsart}

\usepackage[cal=cm, bb=ams, frak=esstix, scr=mma]{mathalfa}
\usepackage{amssymb, amsmath, latexsym, amssymb, amsfonts, amsbsy, amsthm, amsfonts, bookmark, cancel, caption, color, enumitem, graphicx, hyperref,  mathtools, mathrsfs, multicol, scalerel, stackengine, tcolorbox, tikz-cd, titlesec} 
\usepackage[backend=biber,style=numeric,sortcites,natbib=true, url=false, doi=false]{biblatex}
\renewbibmacro{in:}{} 
\hypersetup{pdflinkmargin=0.5pt}  

\renewcommand{\phi}{\varphi} 

\stackMath   
\titleformat{\section}{\LARGE}{\thesection.}{0.33em}{\centering}
\titleformat{\subsection}{\Large\bfseries}{\thesubsection.}{0.33em}{}
\titleformat{\subsubsection}{\large\bfseries}{\thesubsubsection.}{0.33em}{}
\allowdisplaybreaks
\renewcommand{\comment}[1]{}

\newtheorem{theorem}{Theorem}[section]

\newtheorem{lemma}{Lemma}[section]
\newtheorem{remark}[lemma]{Remark}
\newtheorem{corollary}[lemma]{Corollary} 

\newtheorem{proposition}[lemma]{Proposition}
\numberwithin{equation}{section}  

\let\widehat\undefined
\newcommand\widehat[1]{%
\savestack{\tmpbox}{\stretchto{%
  \scaleto{%
    \scalerel*[\widthof{\ensuremath{#1}}]{\kern-.6pt\bigwedge\kern-.6pt}%
    {\rule[-\textheight/2]{1ex}{\textheight}} 
  }{\textheight}%
}{0.5ex}}%
\stackon[2pt]{#1}{\tmpbox}%
}
\newcommand\widecheck[1]{ \savestack{\tmpbox}{\stretchto{\scaleto{
    \scalerel*[\widthof{\ensuremath{#1}}]{\kern-.6pt\bigvee\kern-.6pt}
    {\rule[-\textheight/2]{1ex}{\textheight}} 
  }{\textheight} }{0.5ex}} \stackon[2pt]{#1}{\tmpbox} }
\let\xxrightharpoonup\xrightharpoonup    
\let\xrightharpoonup\undefined  
\newcommand{\xrightharpoonup}[1]{\ifthenelse{\equal{#1}{\star}}
{\xxrightharpoonup{\raisebox{-0.25ex}[0pt][-1.5ex]{\(\scriptstyle \,\,#1\,\,\)}}}
{\xxrightharpoonup{\raisebox{0ex}[0pt][-1.5ex]{\(\scriptstyle \,\,#1\,\,\)}}}}

\title[hydrostatic approximation]{On the hydrostatic approximation of Navier-Stokes-Maxwell system with 2D electronic fields}
\author{Faiq Raees}
\address[F. Raees]{Courant Institute of Mathematical Sciences, New York University, 251 Mercer Street, New York, United States of America.\newline \indent Department of Mathematics, New York University Abu Dhabi, Saadiyat Island, P.O. Box 129188, Abu Dhabi, United Arab Emirates.}
\email{fr872@nyu.edu}
\author{Weiren Zhao}
\address[W. Zhao]{Department of Mathematics, New York University Abu Dhabi, Saadiyat Island, P.O. Box 129188, Abu Dhabi, United Arab Emirates.}
\email{zjzjzwr@126.com, wz19@nyu.edu}
\begin{document}

\begin{abstract} 
    In this paper, we prove the local well-posedness of a scaled anisotropic Navier-Stokes-Maxwell system in a two-dimensional striped domain with a transverse magnetic field around $ (0,0,1)$ in Gevrey-2 class. We also justify the limit from the scaled anisotropic equations to the associated hydrostatic system and obtain the precise convergence rate. Then, we prove the global well-posedness for the system and show that small perturbations near $(0,0,1)$ decay exponentially in time. Finally, we show the optimality of the Gevrey-2 regularity by proving the solution to linearized hydrostatic system around shear flows $(V(y),0,0)=(y(1-y),0,0)$ with some initial data $(\zeta, \zeta ^1)$ grows exponentially. More precisely, for some large parameter $  \lvert k \rvert>M\gg 1 $ corresponding to the frequency in $x$, there exists a solution $ h_k(t,x,y)$ of the system   
    \begin{equation*} 
        \begin{cases}  \partial_{tt}h_k+\partial_th_k-\partial_{yy}h_k+V(y) \partial_x h_k =0,\\
            h_k(0,x,y)=\zeta,\quad \partial_th_k(0,x,y)= \zeta ^1,\\
            h_k(t, x,0)=h_k(t, x, 1)=0,  \end{cases}  
    \end{equation*} 
    such that for any $s\in [0,\frac{1}{2})$ and $t\in [T_k,T_0)$ with $T_{k}\approx |k|^{s-\frac{1}{2}}\to 0$ as $|k|\to \infty$ and some $T_0$ small and independent of $k$, it satisfies 
    \begin{align*}
        \lVert  h_k(t) \rVert_{L^2 }\geq C \, e^{\sqrt{|k|}t}( \lVert  \zeta \rVert_{L^2} + \lVert  \zeta ^1 \rVert_{L^2}),
    \end{align*}
    for some $C > 0$ independent of $k$. 
    \mbox{}\vspace*{-2\baselineskip}
\end{abstract} 
 \maketitle  

\section{Introduction.}
In this paper, we examine a mathematical model that describes the behavior of electrically conducting fluids and hot plasma within thin boundary layers, which are relevant in various physical scenarios such as the hydrodynamics of plasma in neutron stars and white dwarfs in astronomy \cite{Aarach2023}, self-cooled liquid metal blankets in nuclear fusion reactors \cite{Buhler_2007}, and electromagnetic stirring in metallurgy \cite{moffatt1991electromagnetic}. Specifically, we focus on the Navier-Stokes-Maxwell system with negligible viscosity and magnetic diffusion in a narrow strip. We begin by rescaling the equations to form an anisotropic Navier-Stokes-Maxwell system and then rigorously establish the hydrostatic limit of this system as the vertical dimension of the strip approaches zero.

The mathematical analysis of boundary layers in electromagnetic fluids began with the work of Hartmann \cite{hartmann1937theory} who studied the flow of viscous electrically-conducting fluid with a background magnetic field. This started the development of the theory of MHD systems with results validating Hartmann's result under the assumption that the characteristic speed of the plasma are non-relativistic (for example \cite{MR3657241}, \cite{MR3975147}, \cite{MR4753440}, \cite{wang2022hydrostatic}). This hypothesis simplifies several terms in the Maxwell's equation by preventing the displacement current in the Ampere's law which introduces another source of magnetic field.  In examples like neutron stars, such an hypothesis might be a point of contention, as a strong time-dependent electric field arises when the plasma density decreases below a critical value \cite{kumar2020frb}. The study of the boundary layers of plasma where the magnetic field is affected by the displacement current was undertaken in the recent work \cite{MR4753440}. Inspired by their paper, we will instead study the boundary layers of plasma where the electric field is affected by a displaced magnetic field. We mention that the well-posedness results of classical Navier-Stokes-Maxwell system can be found in \cite{MR4105512}, \cite{MR3164536}, \cite{masmoudi2010global}.  
\subsection{Navier-Stokes-Maxwell with 2D restriction.}
We consider the following Navier-Stokes-Maxwell (NSM) system: 
\begin{equation}\label{eq:The NSM system}
    \begin{cases} \partial_t \mathbf{U} + \mathbf{U} \cdot \nabla \mathbf{U} - \nu \Delta U + \nabla p = J \times  \mathbf{B} & \text{conservation of linear momentum,} \\
        \mathrm{div\,} \mathbf{U} = 0 & \text{incompressibility,} \\ 
        \mathrm{div\,} \mathbf{B} = \mathrm{div\,} \mathbf{E} = 0 & \text{Gauss' law of electric and magnetic fields,} \\ 
        \partial_t \mathbf{B} + \mathrm{curl\,} \mathbf{E} = 0 &\text{Faraday's law,} \\ 
        \frac{1}{c^2} \partial_t \mathbf{E} + \mu _0 J = \mathrm{curl\,} \mathbf{B} &\text{Ampere's law,} \\ 
        J = \mu(\mathbf{E} + \mathbf{U} \times  \mathbf{B}) & \text{Ohm's law,} 
    \end{cases}   
\end{equation} 
where \( \nabla = ( \partial_X, \partial_Y, \partial_Z), \Delta = \partial_X^2 + \partial_Y ^2+\partial_Z^2\) and \(\mathbf{U}, \mathbf{B}, \mathbf{E}\) stand for the velocity of a viscous incompressible fluid, the incompressible magnetic field and the electronic field respectively. Here we now restrict to the Cartesian 2.5-dimensional solutions in thin striped domains \( (t, X, Y) \in [0,T] \times  \mathbb{R} \times  [0,\varepsilon]\) where the form of the solution is given by 
\[  \mathbf{U}(t,X,Y)  = \begin{pmatrix} U_1 \\ U_2 \\0 \\ \end{pmatrix}, \quad  \mathbf{E}(t,X,Y)  = \begin{pmatrix} E_1 \\ E_2 \\0 \\ \end{pmatrix}, \quad  \mathbf{B}(t,X,Y)  = \begin{pmatrix} 0 \\ 0 \\B \\ \end{pmatrix}.
\] 
Then it is easy to check that 
\[ J = \mu  _0  \begin{pmatrix} E_1 + U_2 B \\ E_2 - U_1 B \\0 \\ \end{pmatrix},   \, J \times  \mathbf{B} =   \mu _0\begin{pmatrix} E_2 B - U_1 B^2  \\  - E_1 B - U_2 B^2  \\0 \\ \end{pmatrix}, \, \mathrm{curl\,}\mathbf{E} = \partial_X E_2 - \partial_Y E_1,  \, \mathrm{curl\,} \mathbf{B} = \begin{pmatrix} \partial_Y B \\ - \partial_X B  \\0 \\ \end{pmatrix}. 
\] 
Finally, to get rid of the current density \( J\), we apply the curl operator to the evolution equation for \( \mathbf{E}\) along with the identities, 
\[ \mathrm{div\,} \mathbf{E} = \mathrm{div\,} \mathbf{U} = 0 = \partial_z  B = 0 , 
\] 
\[   \mathrm{curl\,} \mathbf{E} = - \partial_t B ,\quad \mathrm{curl\,} \mathrm{curl\,} \mathbf{B}  = - \Delta B,\quad \frac{1}{\mu } \mathrm{curl\,} J  =   \mathrm{curl\,} \mathbf{E}   + \mathrm{curl\,} ( \mathbf{U} \times  \mathbf{B}) =  - \partial_t B  - \mathbf{U} \cdot \nabla B 
\] 
to obtain 
\[  -  \frac{1}{\mu _0 \mu c^2}  \partial_{tt} B    - \partial_t B - \mathbf{U} \cdot \nabla B  = - \Delta B. 
\] 
In the light of the above calculations, we now consider the following formulation of the 2D-Navier-Stokes-Maxwell system \eqref{eq:The NSM system} for \((t, X, Y) \in (0,T) \times  \mathbb{R} \times  (0,\varepsilon) : \)
\begin{equation}\label{eq:The 2D-NSM system}
    \begin{cases}  \partial_t U_1 + \mathbf{U} \cdot \nabla U_1   - \nu \Delta U_1 + \partial_X P = \mu(E_2 B - U_1 B^2) \\ 
    \partial_t U_2 + \mathbf{U} \cdot \nabla U_2 - \nu \Delta U_2 + \partial_Y P  = - \mu(E_1 B + U_2 B^2) \\
    \partial_X U_1 + \partial _Y U_2 = \partial_X E_1 + \partial _Y E_2 = 0 \\
    \partial_t B + \partial_X E_2  - \partial_Y E_1 = 0 \\
    \frac{1}{c^2 \mu _0 \mu} \partial_{tt} B + \partial_t B - \frac{1}{\mu _ 0 \mu } \Delta B + \mathbf{U} \cdot \nabla B = 0, 
    \end{cases}  
\end{equation} 
We will consider the system under the following boundary conditions:  
\[  U_1(t, X , 0) = U_2(t, X , 0)  = \partial_Y E_1(t, X , 0) = E_2 (t, X, 0) = 0  , 
\] 
\[  U_1 (t, X, \varepsilon) = U_2(t, X , \varepsilon) = \partial_Y E_1(t, X, \varepsilon) = E_2(t, X, \varepsilon) = 0.  
\] 
One can physically intuit the above setup as a flow of viscous-electrically conducting fluid between an two infinite current-carrying wires where the wires are situated at \( y =0\) and \(y = + \varepsilon\). Then the boundary condition on \( U_1\) and \( U_2\) are simply the non-slip condition for fluids; the boundary condition on \(E_2\) is simply due to the current flowing in the horizontal direction so there cannot be a vertical component of the electric field inside the wire. The boundary condition on \( E_1 \) is the statement that the magnetic field at the plates is not time-varying (as per the Faraday's law and the boundary condition on \( E_2\)).  

\subsection{The re-scaled equations and the hydrostatic system.}
The hydrostatic approximation naturally emerges in the analysis of fluid flow when the vertical extent of the region being studied is significantly smaller than its horizontal dimensions. In this paper, we fix \(\nu = \varepsilon ^2\)  and consider the asymptotic limit as \(\varepsilon  \to 0^{+}\). We first introduce the change of variables 
\[  x = X \in \mathbb{R}, \quad y = \frac{Y}{\varepsilon} \in [0,1],  
\] 
and the re-scaled variables: 
\[  u^\varepsilon (t, x, y) = U_1(t,X, Y) , \quad v^\varepsilon(t,x , y) = \frac{U_2(t, X, Y)}{\varepsilon}, \quad p^\varepsilon(t, x, y) = P(t,X, Y )
\] 
\[ e^\varepsilon(t, x, y) = E_1(t, X, Y), \quad f^\varepsilon(t,x,y) =\frac{E_2(t,X,Y)}{\varepsilon}, \quad b^\varepsilon(t,x,y) = \frac{B(t,X,Y)}{\varepsilon}.  
\] 
where the scaling on \( v^\varepsilon\) is forced to retain a divergence-free structure on \( (u^\varepsilon , v^\varepsilon):\)
\[  \partial_x u^\varepsilon + \partial_y v^\varepsilon = \partial_X U_1 + \varepsilon \partial_Y v^\varepsilon   = \partial_ XU_1 + \partial_Y U_2  = 0, 
\] 
and similarly for \( (e^\varepsilon, f^\varepsilon)\). 
Then we find that \((u^\varepsilon, v^\varepsilon, p^\varepsilon, e^\varepsilon, f^\varepsilon, b^\varepsilon)\) verifies the following system in \((t,x,y) \in [0,T] \times  \mathbb{R} \times  [0,1] : \) 
\begin{equation}\label{eq:System for (u^e,b^e)}
    \begin{cases} \partial_t u^\varepsilon + u^\varepsilon \partial_x u^\varepsilon + v^\varepsilon \partial_y u^\varepsilon - \varepsilon ^2 \partial_{xx} u^\varepsilon - \partial_{yy} u^\varepsilon + \partial_x p^\varepsilon  = \alpha(f^\varepsilon b^\varepsilon - u^\varepsilon (b^\varepsilon)^2) \\ 
    \varepsilon ^2( \partial_t v^\varepsilon  + u^\varepsilon \partial_x v^\varepsilon + v^\varepsilon \partial_y v^\varepsilon - \varepsilon ^2 \partial_{xx} v^\varepsilon - \partial_{yy} v^\varepsilon   ) + \partial_{y} p^\varepsilon    = - \alpha(e^\varepsilon b^\varepsilon + v^\varepsilon (b^\varepsilon)^2) \\ 
    \partial_x u^\varepsilon + \partial_y v^\varepsilon  = \partial_x e^\varepsilon + \partial_y f^\varepsilon = 0 \\ 
    \partial_t b^\varepsilon + \partial_x f^\varepsilon - \partial_y e^\varepsilon = 0 \\ 
    \beta \partial_{tt} b^\varepsilon + \partial_t b^\varepsilon - \gamma (\varepsilon ^2 \partial_{xx} b^\varepsilon + \partial_{yy} b^\varepsilon  )  + u^\varepsilon \partial_x b^\varepsilon + v^\varepsilon \partial_y b^\varepsilon = 0 
    \end{cases}  
\end{equation} 
where 
\[  \alpha:= \mu \varepsilon ^2  , \quad  m:= \frac{c}{\varepsilon}, \quad \beta := \frac{1}{\mu ^2 \mu _0 \alpha} , \quad \gamma := \frac{1}{\mu _0 \alpha }.
\] 
We impose the following boundary conditions:
\[ ( u^\varepsilon , v^\varepsilon , \partial_y e^\varepsilon , f^\varepsilon ) |_{y = 0,1} = 0, \quad b^\varepsilon|_{t  = 0 , y = 0,1} = 1. 
\]  
The nature of the boundary condition on \(b^\varepsilon\) will be elucidated in the next subsection. 

Formally sending \(\varepsilon \to 0^{+}\), we obtain the {hydrostatic} Navier-Stokes-Maxwell system: 
\begin{equation}\label{eq:System for (u^p,b^p)}
   \begin{cases}  \partial_t u^{P} + u^P \partial_x u^P + v^P \partial_y u^P  -\partial_{yy} u^P + \partial_x p^P = \alpha (f^P b^P - u^P (b^P)^2)  \\ 
    \partial_y p^P = - \alpha (e^P b^P + v^P (b^P)^2) \\
    \partial_x u^P + \partial_y v^P = \partial_x e^P + \partial_y f^P = 0 \\ 
    \partial_t b^P + \partial_x f^P - \partial_y e^P = 0 \\
    \beta \partial_{tt} b^P + \partial_t b^P - \gamma \partial_{yy} b^P + u^P \partial_x b^P  + v^P \partial_y b^P  = 0.       
\end{cases}  
\end{equation}
complemented with the boundary condition 
\[ ( u^P, v^P , \partial_y e^P , f^P   ) _{y = 0,1} = 0,  \quad b^P|_{t  = 0 , y = 0,1} = 1. 
\] 
\subsection{Main results.}  
One of the most fundamental system for the boundary layer theory is the classical Prandtl theory which concerns the behavior of solutions to the Navier-Stokes equations at vanishing viscosity in presence of a boundary. Akin to the classical Prandtl system, our system \eqref{eq:System for (u^p,b^p)} loses a derivative due to \(v^p \partial_y u^p\) and \( v^p \partial _y b^p\). The local well-posedness of the classical Prandtl system was obtained by Oleinik in \cite{oleinik1966mathematical} where she showed that if the initial data is Sobolev and monotonic in \(y\) then the solutions are locally Sobolev. One can even obtain a global result under a favourable pressure gradient condition \((\partial_x p \leq 0 )\) \cite{MR2020656}. Removing the structural assumptions of monotonicty, we have local well-posedness of the Prandtl equations provided the initial data is analytic in the horizontal direction \cite{MR3284569}, \cite{kukavica2013local}, \cite{MR2049030}. The assumption of analyticity was later replaced by a milder Gevrey-\(\frac{7}{4}\) regularity in the horizontal variable \cite{gerard2015well}. This was then improved to Gevrey-\(2\) regularity in horizontal with Sobolev-regularity in vertical direction \cite{Dietert_2019, li2020well, chen2018well}.   


For the hydrostatic Navier-Stokes (Euler) equations, it is known that, without any structural assumption on the initial data, real-analyticity of the initial data is both necessary and sufficient for local well-posedness of the system \cite{MR2563627}, \cite{MR2737223}. The global well-posedness and the global-in-time limit was shown by Paicu, Zhang and the Zhang for a small analytic initial data \cite{MR4125518}. Even more, Renardy showed that the linearisation of the hydrostatic Navier-Stokes equations is ill-posed for certain parallel shear flows and the solutions grow like \( e^{  \lvert k \rvert t}\). This meant that a structural assumption was necessary to break out of analyticity. This was provided by assuming convexity on the initial data: Under this assumption, G\'erard-Varet, Masmoudi and Vicol proved local well-posedness and the hydrostatic approximation holds under Gevrey-\(\frac{9}{8}\) regularity \cite{MR1808843}. These results were then improved to Gevrey-\(\frac{3}{2}\) in \cite{MR4803680}. We also refer to \cite{MR4621052} for the recent result of the Triple-deck system. 

The magnetic field has a stabilising effect: We refer to the following papers for the discussion of the MHD boundary layers with non-degenerate background magnetic field \cite{MR3657241}, \cite{MR3975147}. The local-posedness of MHD boundary layer in Sobolev spaces is established without assuming monotonicty on the velocity field in \cite{MR3882222}. Without the non-degeneracy assumption, it was proved that the system is ill-posed in Gevrey-\(2\) class \cite{MR3864769}. Aarach investigated the rescaled anisotropic MHD system and proved its global well-posedness and the approximation to the hydrostatic MHD system for small analytic-in-\(x\) data \cite{MR4362378}. This result was improved to merely requiring Sobolev regularity under the assumption that the tangential magnetic field is non-degenerate in \cite{wang2022hydrostatic}. 

We also want to mention the hyperbolic Navier-Stokes system which has a similar structure. For this system, the global well-posedness and the global-in-time limit was proved for small analytic data in \cite{MR4650833} which was then improved to small Gevrey-\(2\) data in \cite{MR4429384}.  

Similar to the problems above, the main difficulty in providing a rigorous justification of the limit from \eqref{eq:System for (u^e,b^e)} to \eqref{eq:System for (u^p,b^p)} lies in the loss of horizontal derivatives for \((u^P, b^P)\) in trying to obtain an energy estimate for the limiting system \eqref{eq:System for (u^p,b^p)}. Indeed, following the classical Prandtl theory (\cite{MR3327535},  \cite{MR3385340}, \cite{MR4271962}), we first determine \(v^P  =  -  \partial_x \int_{0}^y u^P \mathrm{\,d}y'\) owing to the divergence-free condition for \( (u,v)\) and the no-slip boundary condition on \(v\). This gives us a loss in the horizontal derivative for \(u^P\). Similarly, in the equation for \(b^P\), we have a loss in the horizontal derivative due to \( u^P \partial_x b^P\) and \( v^P \partial_y b^P\). 

Returning to our NSM-system, we find that when the magnetic field is a small perturbation near a strong non-degenerate magnetic field, for simplicity a constant homogenous one, then the term \(v^P (b^P)^2\) in \eqref{eq:System for (u^p,b^p)} will provide us a strong linear damping to compensate for a horizontal-derivative loss in \(v\). This means that transport-terms of the form \( v^P \partial_y b^P\) are not a major issue due to the damping and diffusion in the vertical direction. However, the same cannot be said for terms like \( u^P \partial_x b^P \) due to the absence of horizontal diffusion. This poses a strong difficulty in closing energy estimates in classical function spaces. 

That said, the equation for the magnetic field in \eqref{eq:System for (u^p,b^p)} is a wave-type equation so we expect that the loss of one horizontal derivative can be recovered if the initial data is taken in Gevrey-2 class. For the above reasons, we will study the well-posedness of the system \eqref{eq:System for (u^p,b^p)} with the initial magnetic field being a small perturbation near \(1\) in the Gevrey-2 class.

To formalise the perturbation, we fix some \(h^\varepsilon := b^\varepsilon - 1\) which satisfies \( h^\varepsilon(0,x,0) = h^\varepsilon(0,x,1) = 0\). Using \( \partial_t b^\varepsilon(t,x,0) = \partial_t b^\varepsilon(t,x,1) = 0\), we have \( h^\varepsilon (t,x,0) = h^\varepsilon(t,x,1) = 0\), for all \( t \geq 0\). 
Thus, our scaled hydrostatic Navier-Stokes-Maxwell system can be written as 
\begin{equation}\label{eq:System for (u^e,h^e)}
    \begin{cases}  \partial_t u^\varepsilon + u^\varepsilon \partial_x u^\varepsilon + v^\varepsilon \partial_y u^\varepsilon - \varepsilon ^2 \partial_{xx} u^\varepsilon - \partial_{yy} u^\varepsilon + \partial_x p^\varepsilon  = \alpha(f^\varepsilon b^\varepsilon - u^\varepsilon (b^\varepsilon)^2) \\ 
        \varepsilon ^2( \partial_t v^\varepsilon  + u^\varepsilon \partial_x v^\varepsilon + v^\varepsilon \partial_y v^\varepsilon - \varepsilon ^2 \partial_{xx} v^\varepsilon - \partial_{yy} v^\varepsilon   ) + \partial_{y} p^\varepsilon    = - \alpha(e^\varepsilon b^\varepsilon + v^\varepsilon (b^\varepsilon)^2) \\ 
        \partial_x u^\varepsilon + \partial_y v^\varepsilon  = \partial_x e^\varepsilon + \partial_y f^\varepsilon = 0 \\ 
        \partial_t b^\varepsilon + \partial_x f^\varepsilon - \partial_y e^\varepsilon = 0 \\ 
        \beta \partial_{tt} h^\varepsilon + \partial_t h^\varepsilon - \gamma (\varepsilon ^2 \partial_{xx} h^\varepsilon + \partial_{yy} h^\varepsilon  )  + u^\varepsilon \partial_x h^\varepsilon + v^\varepsilon \partial_y h^\varepsilon = 0 , \\
        (u^\varepsilon, v^\varepsilon, e^\varepsilon, f^\varepsilon)|_{t = 0} = (u_{\mathrm{in}}, v_{\mathrm{in}} , f_{\mathrm{in}} ,e_{\mathrm{in}}) ,  \quad h^\varepsilon|_{t = 0, y = 0,1} = 0.\\
        (u^\varepsilon, v^\varepsilon,\partial_y e^\varepsilon, f^\varepsilon )|_{y = 0,1} = 0 . 
          \end{cases}   
\end{equation}  

\noindent Similarly, our hydrostatic Navier-Stokes-Maxwell system can be written as: 
\begin{equation}\label{eq:System for (u^p,h^p)}
    \begin{cases} \partial_t u^P + (u^P , v^P) \cdot \nabla u^P - \partial_{yy} u^P  + \partial_x p^P = \alpha (f^p  + f^p h^P - u^P  + 2u^P h^P - u^P (h^P)^2), \\
    \partial_y p^P  = - \alpha ( e^P + e^P h^P  + v^P  + 2 v^P h^P + v^P (h^P)^2) , \\
\nabla \cdot (u^P, v^P) = \nabla \cdot ( e^P , f^P) = 0 , \\
\partial_t h^P  + \partial_x f^P - \partial_y e^P = 0 , \\
\beta \partial_{tt} h^P  + \partial_t h^P  - \gamma \partial_{yy} h^P + (u^P, v^P) \cdot \nabla h^P = 0 , \\ 
(u^P, v^P , e^P , f^P)|_{t = 0} =(u_{\mathrm{in}}, v_{\mathrm{in}}, e_{\mathrm{in}}, f_{\mathrm{in}}),  \quad h^P|_{t = 0,y = 0,1} = 0 , \\ 
(u^P, v^P , \partial_y e^P , f^P )|_{y = 0,1} = 0  \end{cases}  
\end{equation} 
where \( h^P  := b^P - 1\) satisfying \(h^P (0,x,0) = h^P(0,x,1) = 0 \). 

\subsubsection{Notation and conventions.}
For any \(s \geq 0 \) and any function \(\phi\), we define \( \mathrm{A}_s \phi\) by  
\begin{equation}\label{eq:Definition of As}
    \mathrm{A}_s \phi := e^{(\delta _0 - \lambda(t)) \langle  \partial_x \rangle^{\frac{1}{2}}}  \langle  \partial_x \rangle^s \phi  = \mathcal{F}_x ^{-1} \left( e^{(\delta _0 - \lambda(t)) \langle  \xi \rangle^{\frac{1}{2}}} \langle  \xi \rangle^s \mathcal{F}_x \phi(\xi)\right)
\end{equation} 
where \(\langle \xi \rangle : = (1 +  \lvert \xi \rvert^2)^{\frac{1}{2}}, \delta _0\) is a fixed positive number and \(\lambda(t)\) is an increasing function. 
For \(p \in [1,\infty]\), we define the Gevrey-\(\frac{1}{2}\) norm with Sobolev correction \( \mathcal{G}^{ \delta ; s }_p\) by 
\begin{align*}
      \lVert  \phi \rVert_{\mathcal{G}^{   \delta ; s}_p} :&= \lVert  e^{\delta \langle  \partial_x \rangle^{\frac{1}{2}}} \langle \partial_x  \rangle^s  \phi \rVert_{L^p_y([0,1] ; L^2_x(\mathbb{R}))}   := \lVert  e^{\delta \langle  \xi \rangle^{\frac{1}{2}}} \langle  \xi \rangle^s \hat{\phi}  \rVert_{L^p_y([0,1] ; L^2_\xi(\mathbb{R}))} \\
      &:= \begin{cases} \left(\int_{0}^1 \left( \int_{\mathbb{R}} e^{2 \delta  \langle  \partial_x \rangle^{\frac{1}{2}}} \langle  \xi \rangle^{2 s}  \lvert  \mathcal{F}_x \phi(\xi, y) \rvert^2 \mathrm{\,d} \xi  \right)^{\frac{p}{2}} \mathrm{\,d} y\right)^{\frac{1}{p}} & p \in [1,\infty) , \\
     \sup\limits_{y \in [0,1]}\,  \left(\int_{\mathbb{R}} e^{2 \delta  \langle  \partial_x \rangle^{\frac{1}{2}}} \langle  \xi \rangle^{2 s}  \lvert  \mathcal{F}_x \phi(\xi, y) \rvert^2 \mathrm{\,d} \xi  \right)^{\frac{1}{2}} & p = \infty  .  \end{cases}   
\end{align*} 
Since for most of the paper, we are taking \( \delta := \delta _0 - \lambda(t)\) it is usually omitted unless specified. In particular: 
\[  \lVert  \phi  \rVert_{\mathcal{G}^{s}_p} := \lVert  e^{(\delta _0 - \lambda(t)) \langle  \partial_x \rangle^{\frac{1}{2}}} \langle  \partial_x \rangle^s \phi  \rVert_{L^p_y([0,1] ; L^2_x(\mathbb{R}))} . 
\] 
Moreover, we will employ the following shorthand to simplify the presentation of many quantities: for a finite sequence of functions \( (a_i)_{i \leq N}\) and a partial derivative \( \partial_X\), we denote 
\[  \lVert  (a_i)_{i\in N} \rVert_{\mathcal{G}^{s}_p} := \sum_{i \leq N } \lVert  a_i \rVert_{H}, \quad \lVert  \partial_X (a_i)_{i \in N} \rVert_{\mathcal{G}^{s}_p} := \sum_{i \leq N} \lVert \partial_X a_i \rVert_{\mathcal{G}^{s}_p}.   
\] 
As an example: 
\[  \lVert  \partial_y (a,b,c) \rVert_{\mathcal{G}^{s}_p} = \lVert  \partial_y a \rVert_{\mathcal{G}^{s}_p} + \lVert  \partial_y b \rVert_{\mathcal{G}^{s}_p} + \lVert  \partial_y c \rVert_{\mathcal{G}^{s}_p}.  
\] 
Finally, we use \( C_{\mathrm{in}} \geq   1 \) to denote constants that depend on \( u_{\mathrm{in}}, v_{\mathrm{in}}, e_{\mathrm{in}},f_{\mathrm{in}}, h_{\mathrm{in}}\) and \(0 <    c_0 \ll  1 \) for constants that depend on \( (e_{\mathrm{in}}, f_{\mathrm{in}}, h_{\mathrm{in}})\). We use \( C\) to represent a universal constant that may depend on \(\alpha, \beta, \gamma\) and change line to line; as a shorthand for \( f \leq C g\), we will sometimes use \( f  \lesssim  g\).  

\subsubsection{Statements of the theorems.}
Our first result is concerned with the uniform local well-posedness of the anisotropic system \eqref{eq:System for (u^e,h^e)} with initial data taken in Gevrey-2 class. 
\begin{theorem}[Local well-posedness]\label{thm:Local well-posedness}
    Let \(s \geq 10\) and fix a \(\delta _0 > 0\).  Suppose that \( (u_{\mathrm{in}}, h_{\mathrm{in}}, e_{\mathrm{in}}, f_{\mathrm{in}})\) satisfy  
    \begin{equation}\label{eq:Control on initial for Local well-posedness}
        \lVert  \partial_y (u_{\mathrm{in}}, e_{\mathrm{in}},f_{\mathrm{in}},h_{\mathrm{in}}) \rVert_{\mathcal{G}^{\delta _0  ; s + \frac{1}{2}}_2}^2 \leq C_{\mathrm{in}} , \quad \lVert  h_{\mathrm{in}}\rVert_{\mathcal{G}^{\delta _0  ; s + \frac{1}{2}}_2}^2 \leq c_0
    \end{equation}  
    with any \( C_{\mathrm{in}} > 0 \) and some \(c_0\) small enough independent of \( C_{\mathrm{in}}\). Then there is a \(T_0 > 0\), {such that for any \(\varepsilon > 0 \)}, the system \eqref{eq:System for (u^e,h^e)} has a unique solution in Gevrey-2 class for \(t \in [0,T_0]\). 

    Moreover, there is \(C > 0 \) such that for any \( t \in [0,T_0], \)  
    \begin{equation}\label{eq:Estimate on u for LW}
        \lVert  u^\varepsilon \rVert_{\mathcal{G}^{ \delta _0 / 2 ; s}_2}^2  + \lVert  \partial_y u^\varepsilon \rVert_{\mathcal{G}^{\delta _0 / 2 ; s - 2}_2}^2   +  \int_{0}^{T_0} \lVert  u^\varepsilon \rVert^2_{\mathcal{G}^{\delta _0 / 2 ; s + \frac{1}{4}}_2} \mathrm{\,d} t  + \int_{0}^{T_0} \lVert  \partial_t u^\varepsilon \rVert_{\mathcal{G}^{s-2}_2}^2 \mathrm{\,d} t  \lesssim {C_{\mathrm{in}}} ,   
    \end{equation} 
    and 
    \begin{equation}\label{eq:Estimate on h for LW}
        \lVert  \partial_y h^\varepsilon \rVert_{\mathcal{G}^{\delta _0 / 2  ; s  - \frac{3}{4}}_2}^2  + \lVert  h^\varepsilon \rVert_{\mathcal{G}^{\delta _0 / 2 ; s - \frac{1}{4} }_2} ^2+ \int_{0}^{T_0}  \lVert  \partial_y h^\varepsilon \rVert_{\mathcal{G}^{\delta _0 / 2  ; s  - \frac{1}{2}}_2}^2  \mathrm{\,d} t + \int_{0}^{T_0}  \lVert  h^\varepsilon \rVert_{\mathcal{G}^{\delta _0 / 2 ; s }_2} ^2  \mathrm{\,d} t +  \int_{0}^{T_0} \lVert  \partial_t h^\varepsilon \rVert_{\mathcal{G}^{s-\frac{3}{4}}_2}^2 \mathrm{\,d} t  \lesssim   c_0. 
    \end{equation}  
\end{theorem}  
\begin{remark}\label{rmk:Estimates extend to system for u^p}
    Since the estimates obtained in the proof for Theorem \ref{thm:Local well-posedness} are independent of \(\varepsilon \in (0,1)\) so via the same strategy we can prove that the limiting system \eqref{eq:System for (u^p,h^p)} admits a solution on \( [0,T_0]\) with the same estimates \eqref{eq:Estimate on u for LW} and \eqref{eq:Estimate on h for LW} provided the initial condition satisfy \eqref{eq:Control on initial for Local well-posedness}. 
\end{remark}  
As claimed, the solutions for the scaled system \eqref{eq:System for (u^e,h^e)} converge to that of \eqref{eq:System for (u^p,h^p)}. More precisely, we will prove the following convergence rate: 
\begin{theorem}[Convergence rate]\label{thm:Convergence rate}
    The solution obtained from Theorem \ref{thm:Local well-posedness} converges to the solution of \eqref{eq:System for (u^p,b^p)} with the same initial data. Moreover there exists a constant \(C_{\mathrm{in},c_0, \delta _0}\) such that for any \(t \in [0,T_0] :\)
    \[ \lVert  u^\varepsilon - u^P \rVert_{\mathcal{G}^{\delta _0 / 2 ; s - 4}_2}^2 + \lVert  \partial_y (h^\varepsilon - h^P) \rVert_{\mathcal{G}^{\delta _0 / 2 ; s - \frac{3}{4} - 4}_2}^2  + \lVert  h^\varepsilon - h^P \rVert_{\mathcal{G}^{\delta _0 / 2 ; s - \frac{1}{4} - 4}_2}^2 \leq C_{\mathrm{in},c_0, \delta _0} \varepsilon ^4. 
    \] 
\end{theorem}   
Finally, we show that given smallness on the initial data, the solutions are global in time. 
\begin{theorem}[Global well-posedness]\label{thm:Global well-posedness}
    For all \(\delta _0 >0  \) and \(s \geq 10\) there exists a \(\kappa _0 > 0\) such that if \((u_{\mathrm{in}}, h_{\mathrm{in}}, e_{\mathrm{in}},f_{\mathrm{in}})\) satisfy 
    \[  \lVert  \partial_y (u_{\mathrm{in}}, e_{\mathrm{in}},f_{\mathrm{in}},h_{\mathrm{in}}) \rVert_{\mathcal{G}^{\delta _0 ; s + \frac{1}{2}}_2}^2 \leq \kappa ^2  
    \] 
    for any \(\kappa \leq \kappa _0\) then the system \eqref{eq:System for (u^e,b^e)} has a global unique solution. Moreover, it holds that 
    \[  \lVert  u^\varepsilon \rVert_{\mathcal{G}^{\delta _0 / 2 ; s}_2}^2 + \lVert  \partial_y h^\varepsilon \rVert_{\mathcal{G}^{\delta _0 / 2 ; s}_2}^2 \leq C \kappa ^2 e^{- \frac{\theta}{2} t } 
    \] 
    where \(\theta > 0 \) is a universal constant.
\end{theorem}  
\begin{remark}\label{rmk:Comparision with Ping Zhang work}
    This theorem is surprising when contrasted with the results of \cite{MR4753440}: In that paper, the authors considered the system \eqref{eq:The NSM system} with a 2D-magnetic field (as opposed to our one-directional magnetic field) near \((1,0)\) and initial data posed in Gevrey-2 class. In their setup, however, only a local result seems tractable. Indeed, when \((u,v,b_1, b_2, e)\) is near \( (0,0,1,0,0)\) then the linearised equations for their limiting system are 
    \[ \begin{cases} \partial_t u - \partial_y ^2 u  + \partial_x p  = 0 ,  & (t,x,y) \in [0,T] \times  \mathbb{R} \times  [0,1] , \\ 
    \partial_y p  = - v + e , \\ 
\partial_{tt} \phi  + \partial_t \phi - \partial_{yy} \phi + v = 0 , \\ 
\partial_t \phi + e = 0  = \mathrm{div\,}(u,v) = \mathrm{div\,}(b_1, b_2) ,  \\ 
(u,\phi)|_{y = 0,1} = 0 ,  \end{cases}   
    \] 
    where the systemic constants \(\alpha,\beta,\gamma\) are omitted for clarity and \( \nabla^\perp \phi  = (b_1  - 1, b_2)\). Then the standard energy estimate for this system is 
    \[  \frac{1}{2}\frac{\mathrm{d}}{\mathrm{d}t}  \left(\lVert  u \rVert_{L^2}^2 + \lVert  \partial_t \phi \rVert_{L^2}^2 + \lVert  \partial_y \phi \rVert_{L^2}^2\right)  + \lVert  \partial_y u  \rVert_{L^2}^2 + \lVert  \partial_t \phi + v \rVert_{L^2}^2 = 0 . 
    \] 
    A decay estimate for the corresponding solution seems impossible. However, compare this to the linearised limiting system that arises when \((u,v,b,e,f)\) is near \((0,0,1,0,0)\): 
    \[ \begin{cases} \partial_t u  - \partial_{yy} u + \partial_x p = f - u , & (t,x,y) \in [0,T] \times  \mathbb{R} \times  [0,1] , \\ 
    \partial_y p = - e - v , \\
        \partial_t b = \mathrm{curl\,}(e,f) , \\ 
        \partial_{tt} b + \partial_t b - \partial_{yy} b  = 0 , \\ \mathrm{div\,}(u,v) = \mathrm{div\,}(e,f) = 0, \\
(u,v,\partial_y e, f, b - 1)|_{y =0,1} = 0.  \end{cases}   
    \] 
    Then the decay of \(u\) is easily obtained from the first two equation as 
    \[  \frac{1}{2}\frac{\mathrm{d}}{\mathrm{d}t} \lVert  u \rVert_{L^2}^2   + \lVert  \partial_y u \rVert_{L^2}^2  + \lVert  u \rVert_{L^2}^2 + \lVert  v \rVert_{L^2}^2 = 0 ,
    \] 
    whereas for \(b - 1\) we have an exponential decay in time due to the damped-wave structure as can be seen by performing separation of variables and noting the time-function decays like \(e^{-\theta t}\) for some \(\theta > 0; \) see also \cite{kafnemer2022lp}. 
\end{remark} 

\begin{remark}\label{rmk:Comparision with Sobolev for MHD}
    We would also like to compare our results with \cite{wang2022hydrostatic} where the authors were able to prove the convergence of the MHD system in a thin strip in Sobolev classes. The problem in our case arises due to the magnetic field \(b^P\) in \eqref{eq:System for (u^p,b^p)} appearing in a wave-like structure as opposed to a transport-structure in \cite{wang2022hydrostatic}. 
\end{remark}

In the classical Prandtl theory, the theory of well-posedness is generally obtained under the assumptions of analyticity or monotonicity of the initial data and as such, when such assumptions are no longer present, some instabilities can develop resulting in the failure of the Prandtl model, at least globally in space-time if not locally. For example, it was shown in \cite{MR1476316} that smooth solutions of the Prandtl equation do not always exist globally in time. Similarly, asymptotic derivation of the Prandtl model was shown to be invalid in \(W^{1,\infty}\) due to some counterexamples presented by Grenier in \cite{MR1761409}. Moreover, Gerard and Dormy, in \cite{gerard2010ill}, established a linear ill-posedness result for the Prandtl problem in the Sobolev space setting; more precisely, they showed the existence of solutions that grow like \( e^{\sqrt{k} t}\) for high-frequencies in \( k \) in \(x\) for the Prandtl equation linearised around the a non-monotonic shear flow. For the theory of hydrostatic Euler and Navier-Stokes, there has been ill-posedness results; most notably of Renardy for a class of parallel shear flows \cite{MR2563627} and \cite{MR3509003}. For the theory of triple-deck model, \cite{gerard2010ill} shows that the result of Iyer and Vicol \cite{MR4275335} of local well-posedness for analytic data is essentially optimal.  

The survey of above results suggests ill-posedness to occur for Gevrey-\(2+\). Unsurprisingly, the main issue is from the equation for the magnetic field as the energy estimates for the velocity can be closed in Sobolev spaces due to damping of vertical velocities and horizontal diffusion. However, the equation for the magnetic field poses the main issue: particularly the transport term \( u \partial_x h\). This is due to a loss of horizontal derivative in \(h\) which is not recoupable via the gain in regularity by the \(\mathcal{CK}\)-terms (as they only provide a gain of \(\frac{3}{4}\) derivatives). 
As such, in the latter half of the paper, we will show linear ill-posedness for the damped wave equation for the magnetic field for a Poiseuille flow velocity by constructing solutions that exhibit growth like \( e^{\sqrt{k}t}\) for high-frequency \(k\) in \(x\).   

More precisely, we will consider the linearised limit system in \( (t,x,y) \in [0,T] \times \mathbb{T} \times  [0,1] : \)
\begin{equation}\label{eq:ill-posedness:System for h}
    \begin{cases}  
       \partial_{tt} h + \partial_t h - \partial_{yy} h +y(1 - y) \partial_x h  = 0 ,   \\
       h|_{y = 0,1}  = \partial_t h |_{y = 0,1}= 0 , \\ h|_{t = 0 } =\zeta, \quad \partial_t h|_{ t =0} = \zeta ^1
    \end{cases}  
\end{equation}   
and prove that the linearized equations are ill-posed in the sense of Hadamard: 
\begin{theorem}[Ill-posedness]\label{thm:Results:ill-posedness for the Linear system for h}
   For any \(s < \frac{1}{2}\) and \( C_0 > 0 \), there exists \(M\) such that for \( \lvert k \rvert > M\), there exists initial data \(\zeta , \zeta ^1\) which is of the form 
   \begin{align*}
       &\zeta(x,y)=a_k(y)\cos kx+b_k(y)\sin kx,\\
       &\zeta ^1(x,y)=c_k \zeta(x,y).
   \end{align*}
   with \( \lVert  (a_k, b_k) \rVert_{ L^2_y} \lesssim 1\) and $|c_k|\approx \sqrt{|k|}$ such that the solution \( h_k\) for the system \eqref{eq:ill-posedness:System for h} which exhibits exponential growth: 
   \begin{align*}
       \lVert  h_k(t) \rVert_{L^2_x H^1_y}\gtrsim e^{\sqrt{|k|}t}-Ce^{-C_0^{-1}\sqrt{|k|}}
   \end{align*}
   and in particular, for all \(t \in [T_k, T_0)\) where  \(T _k\approx_{C_0}  \lvert k \rvert^{s - \frac{1}{2}}   \ll  1 \)  and \(T_0 > 0\) is independent of \(k,\)
   \begin{equation}\label{eq:Results:ExponentialGrowthForh}
        \lVert  h_k(t) \rVert_{L^2_x H^1_y} \gtrsim C_0^{\frac{1}{2} - s} e^{C_0^{-1}\lvert k \rvert^s}. 
   \end{equation} 
\end{theorem}

\begin{remark}\label{rmk:Results:ill-posedness in full space}
    Theorem \ref{thm:Local well-posedness} also holds in the periodic setting, namely, $(x,y)\in \mathbb{T}\times [0,1]$. Moreover, Theorem \ref{thm:Results:ill-posedness for the Linear system for h} can be proved in the full space setting with few modifications provided we replace \( L^2_x\) with some weighted \(L^2\)-space such that \( e^{ikx}\) becomes integrable.
\end{remark}

\begin{remark}\label{rmk:Results:WKB Method}
    By taking the Fourier transform in $x\to k$ and Laplace transform in $t\to c$, we can reduce the ill-posedness problem of the system \eqref{eq:ill-posedness:System for h} to studying the eigenvalue problem for the complex harmonic oscillator in a finite interval with Dirichlet boundary condition. Dividing \(k\) from the eigenvalue problem and translating the problem by \(\frac{1}{2}\), namely $z=y-\frac{1}{2}$, we obtain a system of the form 
    \[  - \tfrac{1}{k} \phi''  + iz^2 \phi = \tfrac{\lambda}{k} \phi, \quad \phi(\pm \tfrac{1}{2}) = 0,  
    \] 
    where $k$ is the wave number, $z=y-\frac{1}{2}\in [-\frac{1}{2},\frac{1}{2}]$, and $\lambda=c^2+c-i\frac{k}{4}$. 
    As \( \frac{1}{k}\) is small, this problem is classically studied via the complex WKB method which was developed in the papers \cite{MR929202} and \cite{MR1382227} for a general potential \(V(z)\). The eigenvalue problem for the potential \(V(z) = iz^2 \) was undertaken by Shkalikov in \cite{shkalikov2002limit}. Here, we provide a construction using the classical result for the first-eigenvalue for the complex harmonic oscillator and use energy methods to complete our ill-posedness result. 
\end{remark}
\section{Local Well-posedness in Gevrey-2 Class}
In this section, we prove Theorem \ref{thm:Local well-posedness}. For the solution \((u^\varepsilon, v^\varepsilon, p, e^\varepsilon, f^\varepsilon, h^\varepsilon)\) of \eqref{eq:System for (u^e,b^e)}, we split it into \((u^L, v^L, e^L,f^L,h^L)\) and \( (u,v,p,e,f,h):= (u^\varepsilon - u^L , v^\varepsilon - v^L , p^\varepsilon, e^\varepsilon - e^L , f^\varepsilon - f^L , h^\varepsilon - h^L ) \) where the first part satisfies the linear system in \( (t,x,y) \in [0,T] \times  \mathbb{R} \times  [0,1] : \)
\begin{equation}\label{eq:Linearised 2D NSM system with initial data}
    \begin{cases} \partial_t u^L - \varepsilon ^2 \partial_{xx} u^L  - \partial_{yy} u^L  = 0  \\  
    \partial_x u^L + \partial_y v^L = \partial_x e^L + \partial_{y} f^L = 0  \\ \partial_t h^L  + \partial_x f^L - \partial_y e^L  = 0 \\
    \beta \partial_{tt} h^L + \partial_t h^L  - \gamma(\varepsilon ^2 \partial_{xx} h^L + \partial_{yy} h^L   )  = 0  \\ 
    (u^L , v^L , e^L , f^L , h^L)|_{t = 0} =  (u_{\mathrm{in}}, v_{\mathrm{in}}, e_{\mathrm{in}}, f_{\mathrm{in}}, h_{\mathrm{in}}) \\
    { (u^L , v^L)|_{y = 0,1} = (0,0)}
      \end{cases}  
\end{equation}  
and \((u,v,p,e,f,h)\) verifies the following quasilinear system with the forcing coming from the linear part 
\begin{equation}\label{eq:System for (u,h)}
    \begin{cases} 
        \partial_t u  - \varepsilon ^2 \partial_{xx} u  - \partial_{yy} u + \partial_x p + u^\varepsilon \partial_x u^\varepsilon + v^\varepsilon \partial_y u^\varepsilon \\
        \qquad\qquad\qquad\qquad\qquad\qquad\qquad\qquad = \alpha(f^\varepsilon + f^\varepsilon h^\varepsilon  - u^\varepsilon - u^\varepsilon (h^\varepsilon)^2 - 2 u^\varepsilon h^\varepsilon  )  \\
        \varepsilon ^2 ( \partial_t v + u^\varepsilon \partial_x v^\varepsilon  + v^\varepsilon \partial_y v^\varepsilon - \varepsilon ^2 \partial_{xx} v   - \partial_{yy} v     )  + \partial_y p\\
        \qquad\qquad\qquad\qquad\qquad\qquad\qquad\qquad = - \alpha (e^\varepsilon  + e^\varepsilon h^\varepsilon  + v^\varepsilon + v^\varepsilon (h^\varepsilon)^2 + 2 v^\varepsilon h^\varepsilon) \\
        \partial_x u + \partial_y v   = \partial_x e + \partial_y f  = 0 \\ 
        \partial_t h  +  \partial_x f - \partial_y e  = 0  \\ 
        \beta \partial_{tt} h + \partial_t h - \gamma(\varepsilon ^2 \partial_{xx}  h + \partial_{yy} h   )  + u^\varepsilon \partial_x h^\varepsilon + v^\varepsilon \partial_y h^\varepsilon = 0  \\ 
        (u,v,f,e,h)|_{t = 0} = (u,v,f,e,h)|_{y = 0, 1} = (0,0,0,0,0)
     \end{cases}  
\end{equation} 
Note here that the initial data has been incorporated in the linear system. Moreover, linear system is \textbf{not} the linearised system of \eqref{eq:System for (u^e,b^e)} as the \(- \alpha u^\varepsilon\) contain an additional linear term.

\subsection{Linear system}
We first treat the linear system. Note that for any \(\sigma' \geq 0: \)
\begin{equation}\label{eq:Differentiating Aₛ}
  \partial_t[ \mathrm{A}_{\sigma'} h] = \mathrm{A}_{\sigma'} (\partial_t h) -  \dot{\lambda} \langle  \partial_x \rangle^{\frac{1}{2}}  [ \mathrm{A}_{\sigma'} h]. 
\end{equation} 

Then as \(\partial_{yy} u^L \) vanishes on the boundary, applying the operator \( \mathrm{A}_{s+1}\) \eqref{eq:Definition of As}, for \(\lambda (0) = 0\) and \(\dot{\lambda}\) to be determined later, on \eqref{eq:Linearised 2D NSM system with initial data} using \eqref{eq:Differentiating Aₛ} we get, 
\begin{equation}\label{eq:LW:Equation for u^L}
    \partial_t \mathrm{A}_{s + 1} u^L + \dot{\lambda} \langle  \partial_x \rangle^{s + \frac{5}{2}} u^L  - \varepsilon ^2 \partial_{xx} \mathrm{A}_{s + 1} u^L - \partial_{yy} \mathrm{A}_{s + 1}u^L = 0. 
\end{equation} 
Then testing \eqref{eq:LW:Equation for u^L} with \(  \partial_{yy}   \mathrm{A}_{ s+ 1} u^L\), we get  
\begin{equation}\label{eq:Gev-2 energy estimate for u^L}
    \frac{1}{2} \frac{\mathrm{d}}{\mathrm{d}t} \lVert  \partial_y u^L \rVert_{\mathcal{G}_2^{ s + 1}}^2 + \dot{\lambda}(t) \lVert  \partial_y u^L \rVert_{\mathcal{G}_2^{ s + \frac{5}{4}}}^2 + \lVert  (\varepsilon \partial_x , \partial_y) \partial_y u^L \rVert_{\mathcal{G}_2^{ s + 1}} = 0, 
\end{equation} 
and testing \eqref{eq:LW:Equation for u^L} with \(\partial_t  \mathrm{A}_{ s + 1} h^L \), we get
\begin{equation}\label{eq:Gev-2 energy estimate for h^L}
    \frac{1}{2} \frac{\mathrm{d}}{\mathrm{d}t} \lVert  (\beta \partial_t, \gamma \varepsilon \partial_x,\gamma \partial_y) h^L \rVert_{\mathcal{G}^{s  + 1}_2}^2 + \dot{\lambda} (t) \lVert  \partial_t h^L  \rVert_{\mathcal{G}_{2}^{ s + \frac{5}{4}}}^2  + \lVert  \partial_t h^L \rVert_{\mathcal{G}_2^{ s + 1}}^2 = 0. 
\end{equation} 

We then have the following lemma. 
\begin{lemma}\label{lem:Smallness on u^L,h^L}
    The system \eqref{eq:Linearised 2D NSM system with initial data} has a global solution \((u^L , v^L , e^L , f^L , h^L)\). Moreover, for any \(s > 0 \) and non-decreasing function \( 0 \leq \lambda(t) \leq \delta _0\), it holds that 
    \[  \lVert  (1,\partial_y)v^L \rVert_{\mathcal{G}_2^{ s}} ^2 + \lVert  (1, \partial_y ) u^L \rVert_{\mathcal{G}_2^{ s + 1}}^2  + \lVert  \nabla (e^L , f^L) \rVert_{\mathcal{G}_{2}^{ s + 1}}^2  + \lVert  \partial_t h^L \rVert_{\mathcal{G}^{s + 1}_2}^2 \leq C C_{\mathrm{in}}
    \] 
    and 
    \[   \lVert  (1,\partial_y) h^L \rVert_{\mathcal{G}_2^{ s + 1}}^2 \leq C c_0 .
    \] 
\end{lemma}  
\begin{proof} 
    Using the divergence-free condition for \( (u^L, v^L)\) and Gagliardo-Nirenberg: 
    \[ \lVert  v^L \rVert_{\mathcal{G}^{s}_2} \lesssim  \lVert  \partial_y v^L \rVert_{\mathcal{G}^{s}_2}  = \lVert  \partial_x u^L \rVert_{\mathcal{G}^{s}_2}  \lesssim  \lVert  u^L \rVert_{\mathcal{G}^{s + 1}_2}.  
    \] 
    As \(\lambda \) is an increasing function so \(\dot{\lambda} \geq 0 \) and so \eqref{eq:Gev-2 energy estimate for u^L} implies  
    \[  \frac{\mathrm{d}}{\mathrm{d}t} \lVert  \partial_y u^L \rVert_{\mathcal{G}^{s + 1}_2}^2 \leq 0 \implies \lVert  \partial_y u^L \rVert_{\mathcal{G}^{s + 1}_2} \leq \lVert  \partial_y u_{\mathrm{in}} \rVert_{\mathcal{G}^{s + 1}_2} 
    \] 
    and thus 
    \[  \lVert (1, \partial_y)  v^L \rVert_{\mathcal{G}^{s}_2}   + \lVert  (1, \partial_y ) u^L \rVert_{\mathcal{G}^{s}_2} ^2  \leq \lVert  \partial_y u_{\mathrm{in}} \rVert_{\mathcal{G}^{s + 1}_2}^2 \leq C C_{\mathrm{in}}. 
    \] 
    Similarly, we have \(\partial_t h ^L= \partial_y e^L - \partial_x f^L\) and \(\partial_x e^L + \partial_y f^L = 0\) so Proposition \ref{prop:Elliptic Estimates} and \eqref{eq:Gev-2 energy estimate for h^L} implies 
    \[  \lVert  \nabla (e^L , f^L)  \rVert_{\mathcal{G}^{s}_2}^2  \lesssim  \lVert \partial_t h^L \rVert_{\mathcal{G}^{s + 1}_2}^2 \lesssim  \lVert \partial_t h_{\mathrm{in}} \rVert_{\mathcal{G}^{s+1}_2}^2  , \quad \lVert  (1, \partial_y h^L) \rVert_{\mathcal{G}^{s + 1}_2} \lesssim  \lVert  \partial_y h_{\mathrm{in}} \rVert_{\mathcal{G}^{s + 1}_2}^2 \lesssim  c_0
    \] 
    To conclude, we recall \( \partial_t h_{\mathrm{in}} = \partial_y e_{\mathrm{in}} - \partial_x f_{\mathrm{in}}\) so 
    \[  \lVert  \partial_t h_{\mathrm{in}} \rVert_{\mathcal{G}^{s+ 1}_2}^2 \leq   \lVert  \nabla (e_{\mathrm{in}}, f_{\mathrm{in}}) \rVert_{\mathcal{G}^{s+1}_2}^2\lesssim  C_{\mathrm{in}}. \qedhere
    \]   
\end{proof}    
 
 \subsection{Uniform estimate for the nonlinear system.}
 
Recall that \( (u,v,p,e,f,h)\) solves the following system \eqref{eq:System for (u,h)} in \( [0,T] \times \mathbb{R} \times  [0,1] : \) 
\[      \begin{cases} \partial_t u -\varepsilon ^2 \partial_{xx} u - \partial_{yy} u + \partial_x p + u^\varepsilon  \partial_x u^\varepsilon  + v^\varepsilon  \partial_y u^\varepsilon  \\
    \qquad\qquad\qquad\qquad\qquad\qquad\qquad\qquad  = \alpha\left( f^\varepsilon  + f^\varepsilon h^\varepsilon  - u^\varepsilon  - u^\varepsilon  (h^\varepsilon )^2  - 2 u^\varepsilon  h^\varepsilon \right) ,\\ 
\varepsilon ^2 \left( \partial_t v  + u^\varepsilon  \partial_ x v^\varepsilon  + v^\varepsilon  \partial_y v^\varepsilon  - \varepsilon ^2 \partial_{xx}v^\varepsilon  - \partial_{yy} v \right) + \partial_y p \\
\qquad\qquad\qquad\qquad\qquad\qquad\qquad\qquad= -\alpha \left(e^\varepsilon  + e^\varepsilon  h^\varepsilon  + v^\varepsilon  + v^\varepsilon  (h^\varepsilon )^2 + 2 v^\varepsilon  h^\varepsilon \right), \\ 
\partial_x u + \partial_y v = \partial_x e + \partial_y f = 0, \\
\partial_t h + \partial_x f  - \partial_y e = 0, \\ 
\beta \partial_{tt} h + \partial_t h - \gamma( \varepsilon ^2 \partial_{xx} h + \partial_{yy} h) + u^\varepsilon   \partial_x h^\varepsilon  + v^\varepsilon   \partial_y h^\varepsilon  = 0,  \\
(u , v , e , f , h )|_{t = 0}   = (u,v,\partial_y e, f, h)|_{y = 0,1} =  (0,0,0,0,0), \end{cases}  
\] 
where 
\[  (u  ,v   , e  , f  , h   , p  ) := ( u^\varepsilon -  u^L, v  ^\varepsilon -  v^L , e ^\varepsilon -  e^L , f ^\varepsilon -  f^L , h ^\varepsilon -  h^L , p  ).  
\] 

For \(s \geq 10 \) and \(\sigma = s - \frac{3}{4}\), we introduce the energy functionals  
\begin{align*}
    \mathcal{E}_u& := \lVert  (u,\varepsilon    v)\rVert_{\mathcal{G}_2^{ s}}^2 , \\
    \mathcal{D}_u & :=  \lVert  ( \sqrt{ 2} \varepsilon \partial_x u , \varepsilon ^2 \partial_x v , \partial_y u) \rVert_{\mathcal{G}^{s}_2}^2+   \alpha \lVert  (u , v )\rVert_{\mathcal{G}^{s}_{2}}^2  ,\\
     \mathcal{CK}_u &:= \dot{\lambda}  \lVert(  u , \varepsilon    v) \rVert_{\mathcal{G}^{s  + \frac{1}{4}}_{2}}^2 , \\
    \shortintertext{and} 
    \mathcal{E}_h &: = \frac{\beta}{2} \lVert  \partial_t h  \rVert_{\mathcal{G}^{\sigma}_{2}}^2   + \frac{\gamma}{2} \lVert  (\varepsilon \partial_x, \partial_y) h  \rVert_{\mathcal{G}^{\sigma}_{2}}^2   + \frac{\beta}{2} \dot{\lambda}^2  \lVert  h \rVert_{\mathcal{G}^{\sigma + \frac{1}{2}}_{2}}^2 ,  \\ \
    \mathcal{D}_h &: = \lVert  \partial_t h \rVert_{\mathcal{G}^{\sigma}_{2}}^2, \\ 
    \mathcal{CK}_h &:= \frac{\beta}{2} \dot{\lambda}  \lVert  \partial_t h \rVert_{\mathcal{G}^{\sigma + \frac{1}{4}}_{2}}^2 + \gamma \dot{\lambda} \rVert(\varepsilon \partial_x, \partial_y )h \rVert_{\mathcal{G}^{\sigma + \frac{1}{4}}_{2}}^2   + \frac{\beta}{2} \left(\dot{\lambda}  \lVert  \partial_t [ \mathrm{A}_{ \sigma  + \frac{1}{4}}h] \rVert_{L^2}^2  + \dot{\lambda}^3  \lVert  h \rVert_{\mathcal{G}^{s}_{2}}^2\right).  
\end{align*}  
We also introduce the low regularity energy:  
\begin{align*}
    \mathrm{E}_u & := \lVert  \partial_y u  , \sqrt{2} \varepsilon  \partial_x u  ,\varepsilon ^2  \partial_x v  \rVert_{\mathcal{G}^{ s - 2}_{2}}^2, \\ 
    \mathrm{D}_u &:= \lVert  \partial_t (u, \varepsilon   v) \rVert_{\mathcal{G}^{ s - 2}_{2}}^2, \\ 
    \mathrm{CK}_u &:=  \dot{\lambda}(t) \lVert ( \partial_y u  , \sqrt{2} \varepsilon  \partial_x u  ,\varepsilon ^2  \partial_x v ) \rVert_{\mathcal{G}^{ \sigma - 1}_{2}}^2,
\end{align*}  
With this setup, we will prove the theorem on local-well-posedness: 

\begin{proof}[Proof of Theorem \ref{thm:Local well-posedness}]
We first prove a uniform estimates of \((u, v, h, e, f)\). The existence then follows directly using a standard compactness argument. To prove the uniform estimate we aim to make a bootstrap argument. To that end, we first establish some a-priori estimates. 
    

\subsubsection{Control on \texorpdfstring{\(\mathcal{E}_u.\)}{Eu.}}   
    Applying \(\mathrm{A}_s u\) on equation of \(u\) and \( \mathrm{A}_sv \) on the equation for \(v\) using \eqref{eq:Differentiating Aₛ} and summing them up, we get
    \begin{align*}
        &\partial_t \mathrm{A}_{s} u + \dot{\lambda} \langle  \partial_x \rangle^{\frac{1}{2}} \mathrm{A}_{s} u + \mathrm{A}_{s} ( u^\varepsilon  \partial_x u^\varepsilon  ) + \mathrm{A}_{s}(v^\varepsilon  \partial_y u^\varepsilon ) - \varepsilon ^2 \partial_{xx} \mathrm{A}_{s} u - \partial_{yy} \mathrm{A}_{s} u + \partial_x \mathrm{A}_{s} p   \\
        &\qquad\qquad\qquad\qquad= \alpha \mathrm{A}_{s} ( f^\varepsilon  + f^\varepsilon  h^\varepsilon  - u^\varepsilon  - u^\varepsilon (h^\varepsilon )^2  - 2 u^\varepsilon  h^\varepsilon   ), \\
        & \varepsilon ^2 ( \partial_t \mathrm{A}_{s} v + \dot{\lambda} \langle  \partial_x \rangle^{\frac{1}{2}} \mathrm{A}_{s} v + \mathrm{A}_{s}(v^\varepsilon  \partial_x v^\varepsilon ) + \mathrm{A}_{s} (v^\varepsilon  \partial_y v^\varepsilon )  - \varepsilon ^2 \partial_{xx} \mathrm{A}_{s} v - \partial_{yy} \mathrm{A}_{s} v + \partial_y \mathrm{A}_{s} p)\\
        &\qquad\qquad\qquad\qquad= - \alpha \mathrm{A}_{s} ( e^\varepsilon   + e^\varepsilon  h^\varepsilon  + v^\varepsilon   + v^\varepsilon  (h^\varepsilon )^2  + 2 v^\varepsilon  h^\varepsilon ). 
    \end{align*} 
    By taking the \(L^2\)-inner product of the above equations with \( (\mathrm{A}_{s}u , \mathrm{A}_{s} v)\) and using the integration by parts, we have (where the pressure disappears due to divergence-free condition of \((u,v)\))  
    \begin{align*}
        \frac{1}{2}\frac{\mathrm{d}}{\mathrm{d}t} \mathcal{E}_{u, v} + \mathcal{CK}_{u, v}  + \mathcal{D}_{u, v}   
         &=  - \langle  \mathrm{A}_{s}(u^\varepsilon  \partial_x u  ) , \mathrm{A}_{s} u  \rangle - \langle  \mathrm{A}_{s} (u^\varepsilon  \partial_x u^L) , \mathrm{A}_{s} u \rangle \\
         &\phantom{=} - \langle  \mathrm{A}_{s} (v^\varepsilon  \partial_y u ) , \mathrm{A}_{s}u \rangle  - \langle  \mathrm{A}_{s} ( v^\varepsilon  \partial_y u^L) , \mathrm{A}_{s} u \rangle \\
         & \phantom{=}- \varepsilon ^2 \langle  \mathrm{A}_{s} ( u^\varepsilon  \partial_x v), \mathrm{A}_{s} v \rangle - \varepsilon ^2 \langle  \mathrm{A}_{s} ( u^\varepsilon  \partial_x v^L) , \mathrm{A}_{s} v \rangle \\
         &\phantom{=} - \varepsilon ^2 \langle  \mathrm{A}_{s} ( v^\varepsilon  \partial_y v) , \mathrm{A}_{s}v \rangle  - \varepsilon ^2 \langle  \mathrm{A}_{s} ( v^\varepsilon  \partial_y v^L), \mathrm{A}_{s} v\rangle  \\
         & \phantom{=}+ \alpha \langle  \mathrm{A}_{s} (f^\varepsilon  + f^\varepsilon   h^\varepsilon   - u^L   - u^\varepsilon  (h^\varepsilon )^2  - 2 u^\varepsilon  h^\varepsilon ) , \mathrm{A}_{s} u \rangle \\
         & \phantom{=}- \alpha \langle  \mathrm{A}_{s} ( e^\varepsilon  + e^\varepsilon h^\varepsilon  + v^L + v^\varepsilon  (h^\varepsilon )^2 + 2 v^\varepsilon  h^\varepsilon  ) , \mathrm{A}_{s}v \rangle \\
        &=: \mathrm{I}_1 + \mathrm{I}_1^L + \mathrm{I}_2  + \mathrm{I}_2^L + \mathrm{I}_3 + \mathrm{I}_3^L + \mathrm{I}_4 + \mathrm{I}_4^L + \mathrm{I}_5 + \mathrm{I}_6
    \end{align*} 
Then from Proposition \ref{prop:Estimate on〈Aₛ(g∇ϕ),Aₛϕ〉}, we have 
\begin{align*}
      \lvert \mathrm{I}_1 \rvert &\lesssim  \lVert  \partial_y u^\varepsilon  \rVert_{\mathcal{G}^{2}_2} \lVert  u \rVert_{\mathcal{G}^{s+ \frac{1}{4}}_2}^2 + \lVert  u^\varepsilon  \rVert_{\mathcal{G}^{s}_2} \lVert  \partial_y u \rVert_{\mathcal{G}^{2}_2} \lVert  u \rVert_{\mathcal{G}^{s}_2}  \\
     &\lesssim ( \lVert  \partial_y u \rVert_{\mathcal{G}^{2}_2} + C_{\mathrm{in}} ^{\frac{1}{2}})\lVert  u \rVert_{\mathcal{G}^{s + \frac{1}{4}}_2}^2 +   ( \lVert  u \rVert_{\mathcal{G}^{s}_2} + C_{\mathrm{in}}^{\frac{1}{2}}) \lVert  \partial_y u \rVert_{\mathcal{G}^{2}_2} \lVert  u \rVert_{\mathcal{G}^{s}_2} . 
\end{align*} 

Letting \( \dot{\lambda} \geq 32 C C_{\mathrm{in}}^{\frac{1}{2}}\), we obtain 
\begin{align*}
     \lvert \mathrm{I}_1 \rvert  &\leq  C\left( \sqrt{\mathrm{E}_u} \lVert  u \rVert_{\mathcal{G}^{s + \frac{1}{4}}_2}^2 + C_{\mathrm{in}}^{\frac{1}{2}} \lVert  u \rVert_{\mathcal{G}^{s + \frac{1}{4}}_2}^2  + \sqrt{\mathrm{E}_u}\mathcal{E}_u+ C^{\frac{1}{2}}_{\mathrm{in}} \sqrt{ \mathrm{E}_u \mathcal{E}_u}\right) \\
     &\leq \frac{\dot{\lambda}}{32} \lVert u \rVert_{\mathcal{G}^{s + \frac{1}{4}}_2}^2  +C \dot{\lambda} \lVert u \rVert_{\mathcal{G}^{s + \frac{1}{4}}_2}^2 \sqrt{\mathcal{E}_u} +C \sqrt{\mathrm{E}_u}\mathcal{E}_u +C C_{\mathrm{in}}^{\frac{1}{2}} \sqrt{\mathrm{E}_u \mathcal{E}_u} \\
     &\leq \frac{1}{32} \mathcal{CK}_u  + C\mathcal{CK}_u \sqrt{\mathcal{E}_u}  + C\sqrt{\mathrm{E}_u}\mathcal{E}_u +C C^{\frac{1}{2}}_{\mathrm{in}} \sqrt{ \mathrm{E}_u \mathcal{E}_u}. 
\end{align*} 

Repeating the same estimate for \(\mathrm{I}_3\) and summing up, we obtain 
\begin{equation}\label{eq:Control on I1 and I3}
    \lvert \mathrm{I}_1 \rvert +  \lvert \mathrm{I}_3 \rvert - \frac{1}{16}\mathcal{CK} ^u  \lesssim    \mathcal{CK} \sqrt{\mathcal{E}_u}  + \sqrt{ \mathrm{E}_u  }\mathcal{E}_u + C^{\frac{1}{2}}_{\mathrm{in}} \sqrt{ \mathrm{E}_u \mathcal{E}_u}. 
\end{equation} 

Using Proposition \ref{prop:Gevrey-Holder}, we also have 
\begin{align*}
     \lvert \mathrm{I}_1^L \rvert + \lvert \mathrm{I}_3^L \rvert &\lesssim  \left( \lVert  u^\varepsilon  \rVert_{\mathcal{G}^{2}_\infty} \lVert  \partial_x u^L  \rVert_{\mathcal{G}^{s}_2} + \lVert  u^\varepsilon  \rVert_{\mathcal{G}^{s}_2} \lVert  \partial_x u^L \rVert_{\mathcal{G}^{2}_\infty}\right) \lVert  u \rVert_{\mathcal{G}^{s}_2} \\
     &\qquad\qquad  +  ( \lVert  u^\varepsilon  \rVert_{\mathcal{G}^{2}_\infty} \lVert  \partial_x v^L   \rVert_{\mathcal{G}^{s}_2} + \lVert  u^\varepsilon  \rVert_{\mathcal{G}^{s}_2} \lVert  \partial_x v^L \rVert_{\mathcal{G}^{2}_\infty}) \lVert  \varepsilon v \rVert_{\mathcal{G}^{s}_2}  \\
     &\lesssim  \lVert  (u , \varepsilon v) \rVert_{\mathcal{G}^{s}_2} \sqrt{C_{\mathrm{in}}} \left( \lVert  u^\varepsilon  \rVert_{\mathcal{G}^{2}_\infty} + \lVert  u^\varepsilon  \rVert_{\mathcal{G}^{s}_2}\right) \\ 
     &\lesssim  \lVert  (u , \varepsilon v ) \rVert_{\mathcal{G}^{s}_2} \sqrt{C_{\mathrm{in}}} \left( ( \lVert  \partial_y u \rVert_{\mathcal{G}^{2}_2} + \sqrt{C_{\mathrm{in}}}) + ( \lVert  u \rVert_{\mathcal{G}^{s}_2} + \sqrt{C_{\mathrm{in}}})\right) \\
     &\lesssim  \sqrt{C_{\mathrm{in}}} ( \mathcal{E}_u + \mathrm{E}_u + C_{\mathrm{in}})
\end{align*} 
Similarly, using Proposition \ref{prop:Estimate on〈Aₛ(g∇ϕ),Aₛϕ〉}, we get 
\begin{align*}
     \lvert \mathrm{I}_2 \rvert &\lesssim  \lVert  \partial_y v^\varepsilon  \rVert_{\mathcal{G}^{2}_2} \lVert  \partial_y u \rVert_{\mathcal{G}^{s}_2} \lVert u  \rVert_{\mathcal{G}^{s}_2} + \lVert  v^\varepsilon  \rVert_{\mathcal{G}^{s}_2} \lVert  \partial_y u \rVert_{\mathcal{G}^{2}_2} \lVert  \partial_y u  \rVert_{\mathcal{G}^{s}_2}  + \lVert  \partial_{yy} v^\varepsilon  \rVert_{\mathcal{G}^{2}_2} \lVert  u \rVert_{\mathcal{G}^{s}_2}^2 \\
     &\lesssim  \lVert  u^\varepsilon  \rVert_{\mathcal{G}^{3}_2} \lVert \partial_y u\rVert_{\mathcal{G}^{s}_2} \lVert  u \rVert_{\mathcal{G}^{s}_2} + \lVert  v^\varepsilon  \rVert_{\mathcal{G}^{s}_2} \lVert  \partial_y u \rVert_{\mathcal{G}^{2}_2} \lVert  \partial_y u \rVert_{\mathcal{G}^{s}_2} + \lVert  \partial_y u^\varepsilon  \rVert_{\mathcal{G}^{3}_2} \lVert  u \rVert_{\mathcal{G}^{s}_2}^2 \\
     &\lesssim  (\sqrt{\mathrm{E}_u} + \sqrt{C_{\mathrm{in}}} )\sqrt{\mathcal{D}_u \mathcal{E}_u  }  + (\sqrt{\mathcal{D}_u} + \sqrt{C_{\mathrm{in}}}) \sqrt{\mathrm{E}_u \mathcal{E}_u  }  + \sqrt{\mathrm{E}_u} \mathcal{E}_u \\
     &\lesssim \sqrt{C_{\mathrm{in}}}  (\sqrt{\mathcal{D}_u \mathcal{E}_u  } +  \sqrt{\mathrm{E}_u \mathcal{E}_u  } ) + \sqrt{\mathrm{E}_u \mathcal{D}_u \mathcal{E}_u}  + \sqrt{\mathrm{E}_u} \mathcal{E}_u  , 
\end{align*}

and using Proposition \ref{prop:Gevrey-Holder}, we get 
\begin{align*}
    \lvert \mathrm{I}_2^L \rvert&\lesssim  \lVert  u \rVert_{\mathcal{G}^{s}_2} \lVert  v^\varepsilon \rVert_{\mathcal{G}^{2}_\infty} \lVert  \partial_y u^L \rVert_{\mathcal{G}^{s}_2} + \lVert  u \rVert_{\mathcal{G}^{s}_\infty} \lVert  v^\varepsilon \rVert_{\mathcal{G}^{s}_2} \lVert  \partial_y u^L \rVert_{\mathcal{G}^{2}_2} \\
    &\lesssim  \lVert  u \rVert_{\mathcal{G}^{s}_2} \lVert  u^\varepsilon \rVert_{\mathcal{G}^{3}_2} \lVert  \partial_y u^L \rVert_{\mathcal{G}^{s}_2} + \lVert  u \rVert_{\mathcal{G}^{s}_2}^{\frac{1}{2}}\lVert  \partial_{y} u \rVert_{\mathcal{G}^{s}_2}^{\frac{1}{2}} \lVert  v^\varepsilon \rVert_{\mathcal{G}^{s}_2}  \lVert  \partial_y u^L \rVert_{\mathcal{G}^{2}_2} \\
    &\lesssim  \sqrt{C_{\mathrm{in}} \mathcal{E}_u } ( \sqrt{\mathrm{E}_u} + \sqrt{C_{\mathrm{in}}} )  + \sqrt{C_{\mathrm{in}}} \mathcal{E}_u^{\frac{1}{4}} \mathcal{D}_u^{\frac{1}{4}}(\mathcal{D}_u^{\frac{1}{2}} + \sqrt{C_{\mathrm{in}}}) 
\end{align*} 

Thus for \(\mathrm{I}_2\), we have  
\begin{equation}\label{eq:Control on I2}
     \lvert \mathrm{I}_2 \rvert +  \lvert \mathrm{I}_2^L \rvert - \frac{1}{32} \mathcal{D}_u \lesssim   \sqrt{C_{\mathrm{in}}} (\mathcal{E}_u + \mathrm{E}_u + C_{\mathrm{in}}) ^2    . 
\end{equation} 
Using Proposition \ref{prop:Estimate on〈Aₛ(g∇ϕ),Aₛϕ〉}, we get 
\begin{align*}
     \lvert \mathrm{I}_4 \rvert &\lesssim \lVert  \partial_y v^\varepsilon  \rVert_{\mathcal{G}^{2}_2} \lVert  \varepsilon \partial_y v \rVert_{\mathcal{G}^{s}_2} \lVert \varepsilon v  \rVert_{\mathcal{G}^{s}_2} + \lVert  v^\varepsilon  \rVert_{\mathcal{G}^{s}_2} \lVert  \varepsilon\partial_y v \rVert_{\mathcal{G}^{2}_2} \lVert  \varepsilon\partial_y v  \rVert_{\mathcal{G}^{s}_2}  + \lVert  \partial_{yy} v^\varepsilon  \rVert_{\mathcal{G}^{2}_2} \lVert \varepsilon v \rVert_{\mathcal{G}^{s}_2}^2  \\ 
     &\lesssim  \lVert  u^\varepsilon  \rVert_{\mathcal{G}^{3}_2} \lVert  \varepsilon \partial_x u \rVert_{\mathcal{G}^{s}_2} \lVert  \varepsilon v \rVert_{\mathcal{G}^{s}_2} + \lVert  v^\varepsilon  \rVert_{\mathcal{G}^{s}_2} \lVert  \varepsilon \partial_x u \rVert_{\mathcal{G}^{2}_2} \lVert  \varepsilon \partial_x u \rVert_{\mathcal{G}^{s}_2} + \lVert  \partial_y u^\varepsilon  \rVert_{\mathcal{G}^{3}_2} \lVert  \varepsilon v \rVert_{\mathcal{G}^{s}_2}^2 \\
     &\lesssim  (\sqrt{\mathrm{E}_u}  + \sqrt{C_{\mathrm{in}}})\sqrt{\mathcal{D}_u \mathcal{E}_u}  + ( \sqrt{\mathcal{D}_u} + \sqrt{C_{\mathrm{in}}}) \sqrt{\mathrm{E}_u \mathcal{D}_u} + (\sqrt{\mathrm{E}_u} + \sqrt{C_{\mathrm{in}}}) \mathcal{E}_u\\
     &\lesssim  C^{\frac{1}{2}}_{\mathrm{in}} \left( \sqrt{\mathcal{D}_u \mathcal{E}_u} + \sqrt{\mathrm{E}_u \mathcal{D}_u} + \mathcal{E}_u\right)  + \sqrt{\mathrm{E}_u \mathcal{D}_u \mathcal{E}_u} +\mathcal{D}_u \sqrt{  \mathrm{E}_u } +  \sqrt{\mathrm{E}_u} \mathcal{E}_u
      \\ \lvert \mathrm{I}_4^L \rvert & \lesssim  \lVert  \varepsilon v \rVert_{\mathcal{G}^{s}_2} \left( \lVert  \varepsilon v^\varepsilon  \rVert_{\mathcal{G}^{s}_2} \lVert  \partial_y v^L \rVert_{\mathcal{G}^{2}_\infty} + \lVert  \varepsilon v^\varepsilon  \rVert_{\mathcal{G}^{2}_\infty} \lVert  \partial_y v^L \rVert_{\mathcal{G}^{s}_2} \right) \\
      & \lesssim  C^{\frac{1}{2}}_{\mathrm{in}} \lVert \varepsilon v \rVert_{\mathcal{G}^{s}_2}( \lVert  \varepsilon v^\varepsilon  \rVert_{\mathcal{G}^{s}_2} + \lVert  \varepsilon u^\varepsilon  \rVert_{\mathcal{G}^{3}_2})\\
      & \lesssim  \sqrt{C_{\mathrm{in}} \mathcal{E}_u}(\sqrt{\mathcal{E}_u} + \sqrt{C_{\mathrm{in}}} )\lesssim  \sqrt{C_{\mathrm{in}}}\mathcal{E}_u + C_{\mathrm{in}} \sqrt{\mathcal{E}_u}. 
\end{align*} 

Thus for \(\mathrm{I}_4\), we have 
\begin{equation}\label{eq:Control on I4}
     \lvert \mathrm{I}_4 \rvert +  \lvert \mathrm{I}_4^L \rvert - \frac{1}{32} \mathcal{D}_u \lesssim  \sqrt{C_{\mathrm{in}}}( \mathcal{E}_u + \mathrm{E}_u + C_{\mathrm{in}})^2 + \mathcal{D}_u \sqrt{\mathrm{E}_u}. 
\end{equation}  

\noindent For the forcing, a direct calculation gives us 
\begin{align*}
    \mathrm{I}_5 &=  \alpha \langle  \mathrm{A}_{s} f^\varepsilon  , \mathrm{A}_{s} u \rangle + \alpha \langle  \mathrm{A}_{s} ( f^\varepsilon  h^\varepsilon ) , \mathrm{A}_{s} u \rangle   - \alpha \langle  \mathrm{A}_{s} u^L , \mathrm{A}_{s} u \rangle   \\
    &\quad- \alpha \langle  \mathrm{A}_{s} (u^\varepsilon  (h^\varepsilon )^2) , \mathrm{A}_{s} u \rangle - 2 \alpha \langle  \mathrm{A}_{s} (u^\varepsilon  h^\varepsilon ) , \mathrm{A}_{s} u  \rangle \\
    &= :\mathrm{I}_{5,1} + \mathrm{I}_{5,2} + \mathrm{I}_{5, L} + \mathrm{I}_{5,3} + \mathrm{I}_{5,4} ,  \\ 
    \mathrm{I}_6 &= - \alpha \langle  \mathrm{A}_{s} e^{\varepsilon } , \mathrm{A}_{s} v \rangle - \alpha \langle  \mathrm{A}_{s} (e^\varepsilon  h^\varepsilon  ) , \mathrm{A}_{s} v\rangle - \alpha \langle  \mathrm{A}_{s} v^L , \mathrm{A}_{s} v \rangle  \\
    &\qquad - \alpha \langle  \mathrm{A}_{s} ( v^\varepsilon  ( h^\varepsilon  )^2) , \mathrm{A}_{s} v  \rangle - 2\alpha \langle \mathrm{A}_{s}( v^\varepsilon  h^\varepsilon  ) , \mathrm{A}_{s} v  \rangle  \\
    &= : \mathrm{I}_{6,1} + \mathrm{I}_{6,2} + \mathrm{I}_{6,L} + \mathrm{I}_{6,3} + \mathrm{I}_{6,4}. 
\end{align*} 
Now let \( \phi(x,y) = \int_{0}^y  u(x,y')\mathrm{\,d} y'     \) be the potential function. Then \(\partial_y \phi  = u\) and \( \partial_x \phi  = -v\) so using the divergence-free condition of \((e,f)\), we have 
\[  \mathrm{I}_{5,1} + \mathrm{I}_{6,1}  =  \alpha \langle \mathrm{A}_{s} f^\varepsilon  ,  - \partial _y \mathrm{A}_{s} \phi   \rangle  - \alpha \langle  \mathrm{A}_{s} e^\varepsilon  , \partial_x \mathrm{A}_{s} \phi  \rangle  = \alpha \langle  \mathrm{A}_{s} ( \partial_ y f^\varepsilon   + \partial_x e^\varepsilon  ) , \mathrm{A}_{s} \phi   \rangle  = 0. 
\] 
For the remaining terms, the strategy is almost similar so we will elaborate the calculation for one of the terms: Let us look at \( \mathrm{I}_{5,2}:\) Using Proposition \ref{prop:Gevrey-Holder}, we can write 
\[   \lvert \mathrm{I}_{5,2} \rvert \lesssim \lVert u \rVert_{\mathcal{G}^{s}_2} \left(\lVert  f^\varepsilon \rVert_{\mathcal{G}^{s}_2} \lVert  h^\varepsilon \rVert_{\mathcal{G}^{2}_\infty} + \lVert  f^\varepsilon \rVert_{\mathcal{G}^{2}_\infty} \lVert  h^\varepsilon \rVert_{\mathcal{G}^{s}_2}\right),
\] 
Now recall that \( h^\varepsilon  = h + h^L\) where \(\lVert h^L \rVert \lesssim c_0 \) in light of Lemma \ref{lem:Smallness on u^L,h^L}. Similarly, we can also control the \(f^\varepsilon\) term using Proposition \ref{prop:Elliptic Estimates} where we recall \( \partial_t h^\varepsilon = \partial_y e^\varepsilon - \partial_x f^\varepsilon\). As such, using Poincare inequality in \(y\), we get
\[    \lvert \mathrm{I}_{5,2} \rvert \lesssim \lVert  u \rVert_{\mathcal{G}^{s}_2} \left(( \lVert  f \rVert_{\mathcal{G}^{s}_2} +\sqrt{C_{\mathrm{in}}})(\lVert  \partial_y h \rVert_{\mathcal{G}^{2}_2} + \sqrt{c_0})  + (\lVert  \partial_ y f \rVert_{\mathcal{G}^{2}_2} + \sqrt{C_{\mathrm{in}}})(\lVert  h \rVert_{\mathcal{G}^{s}_2} +\sqrt{c_0})  \right) . 
\] 
Since these estimates end up involving the \(\mathcal{G}^{s}_2\) norm of \(h,\) we need this term in our controls which ends up corresponding to \(\sigma + \frac{3}{4} = s\) as per Lemma \ref{lem:|∂ₜh|ₛ expansion}. This explains the relation between \(\sigma\) and \(s\). As such, we can control \(\lVert h \rVert_{\mathcal{G}^{s}_2} \lesssim \sqrt{\frac{\mathcal{CK}_h}{\dot{\lambda}}} \) where \(\dot{\lambda}\) can absorb the large constant \(C\). Thus we have 
\begin{align*}
    \lvert \mathrm{I}_{5,2} \rvert &\lesssim   \sqrt{\mathcal{D}_u} \left( (\sqrt{\mathcal{E}_h} + \sqrt{C_{\mathrm{in}}})(\sqrt{\mathcal{E}_h} + \sqrt{c_0}) + (\sqrt{\mathcal{E}_h} + \sqrt{C_{\mathrm{in}}})(\frac{1}{\sqrt{\dot{\lambda}}}\sqrt{\mathcal{CK}_h} + \sqrt{c_0}) \right)\\
    &\lesssim  C_{\mathrm{in}} \sqrt{\mathcal{D}_u} + \sqrt{\mathcal{D}_u} \mathcal{E}_h + \sqrt{\mathcal{D}_u \mathcal{E}_h \mathcal{CK}_h} + \frac{1}{32 C} \sqrt{\mathcal{D}_u \mathcal{CK}_h} . 
\end{align*} 
Here we controlled \(\lVert u \rVert_{\mathcal{G}^{s}_2}\) by the damping instead of the energy term. This is mainly to have similarity with the estimate of \(\mathrm{I}_{6,2}\) as it contains \(\lVert v \rVert_{\mathcal{G}^{s}_2}\) which can only be controlled by damping: 
\begin{align*}
    \lvert \mathrm{I}_{6,2} \rvert  &\lesssim  \alpha \lVert  v \rVert_{\mathcal{G}^{s}_2} \left( ( \lVert  e \rVert_{\mathcal{G}^{s}_2} + C_{\mathrm{in}})(\lVert  \partial_y h \rVert_{\mathcal{G}^{2}_2} + C^h_{\mathrm{in}}) + (\lVert  \partial_y e \rVert_{\mathcal{G}^{2}_2} + C_{\mathrm{in}}) ( \lVert  h \rVert_{\mathcal{G}^{s}_2} + C^h_{\mathrm{in}})  \right) \\
    &\lesssim   C_{\mathrm{in}} \sqrt{\mathcal{D}_u} + \sqrt{\mathcal{D}_u} \mathcal{E}_h + \sqrt{\mathcal{D}_u \mathcal{E}_h \mathcal{CK}_h} + \frac{1}{32 C} \sqrt{\mathcal{D}_u \mathcal{CK}_h}
\end{align*} 
For the terms \( \mathrm{I}_{5,L}\) and \(\mathrm{I}_{6,L}\) we have a simple estimate using Cauchy-Schwarz inequality: 
\[  \lvert \mathrm{I}_{5,L} \rvert  +  \lvert \mathrm{I}_{6,L} \rvert \lesssim  \lVert  u \rVert_{\mathcal{G}^{s}_2} \lVert  u^L \rVert_{\mathcal{G}^{s}_2} + \lVert  v \rVert_{\mathcal{G}^{s}_2} \lVert v^L \rVert_{\mathcal{G}^{s}_2} \lesssim  \sqrt{C_{\mathrm{in}}} \lVert  (u,v) \rVert_{\mathcal{G}^{s}_2} \lesssim  \sqrt{C_{\mathrm{in}} \mathcal{D}_u}. 
\] 
The controls on \(\mathrm{I}_{5,j}\) and \(\mathrm{I}_{6,j}\) for \(j \geq 2\) follows similarly to \(\mathrm{I}_{5,2} :\) the main difference is that in \(\mathrm{I}_{5,3}\) and \(\mathrm{I}_{6,3}\) we have \((h^\varepsilon)^2\) mandating the repeated use of Proposition \ref{prop:Gevrey-Holder}. To illustrate, we present the full proof for \(\mathrm{I}_{5,3} : \) 

\begin{align*}
    \lvert \mathrm{I}_{5,3} \rvert &\lesssim  \alpha \lVert  u \rVert_{\mathcal{G}^{s}_2} \left( \lVert  u^\varepsilon  \rVert_{\mathcal{G}^{s}_2} \lVert  h^\varepsilon  \rVert_{\mathcal{G}^{2}_\infty}^2  + \lVert  u^\varepsilon  \rVert_{\mathcal{G}^{2}_\infty} \lVert  h^\varepsilon  \rVert_{\mathcal{G}^{2}_\infty} \lVert  h^\varepsilon  \rVert_{\mathcal{G}^{s}_2} \right) \\
    &\lesssim  \alpha \lVert  u \rVert_{\mathcal{G}^{s}_2} \left( ( \lVert  u \rVert_{\mathcal{G}^{s}_2} + \sqrt{C_{\mathrm{in}}})( \lVert  \partial_y h \rVert_{\mathcal{G}^{2}_2} + \sqrt{c_0})^2\right) + \\
    & \qquad\qquad\qquad\qquad+ \alpha \lVert  u \rVert_{\mathcal{G}^{s}_2}
    \left((\lVert \partial_y  u \rVert_{\mathcal{G}^{2}_2} + \sqrt{C_{\mathrm{in}}})( \lVert  \partial_y h \rVert_{\mathcal{G}^{2}_2} + \sqrt{c_0})(\lVert  h \rVert_{\mathcal{G}^{s}_2} + \sqrt{c_0})\right),\\ 
    &\lesssim \sqrt{\mathcal{D}_u} ( \sqrt{\mathcal{D}_u} \mathcal{E}_h  + \sqrt{\mathcal{D}_u c_0 \mathcal{E}_h} + c_0 \sqrt{\mathcal{D}_u} + C_{\mathrm{in}}^{\frac{3}{2}} + \sqrt{C_{\mathrm{in}}}\mathcal{E}_h ) \\ 
    &\qquad \qquad + (\sqrt{\mathcal{D}_u\mathrm{E}_u }+ \sqrt{\mathcal{D}_u C_{\mathrm{in}}} ) ( \sqrt{\mathcal{E}_h \mathcal{CK}_h}   +\frac{1}{\sqrt{\dot{\lambda}}} \sqrt{c_0 \mathcal{CK}_h} +  c_0   ) \\ 
    &\lesssim \mathcal{D}_u \mathcal{E}_h  + c_0 \mathcal{D}_u + C^{\frac{3}{2}}_{\mathrm{in}} \sqrt{\mathcal{D}_u} + \sqrt{C_{\mathrm{in}} \mathcal{D}_u} \mathcal{E}_h  \\
    &\qquad+  \sqrt{\mathcal{D}_u \mathrm{E}_u \mathcal{E}_h \mathcal{CK}_h}   + \sqrt{\mathcal{D}_u \mathrm{E}_u c_0 \mathcal{CK}_h} + c_0 \sqrt{\mathcal{D}_u \mathrm{E}_u}   + \sqrt{\mathcal{D}_u  \mathcal{E}_h \mathcal{CK}_h}  + \sqrt{\mathcal{D}_u  \mathcal{CK}_h} + c_0 \sqrt{\mathcal{D}_u C_{\mathrm{in}}}  \\
    &\lesssim  \mathcal{D}_u \mathcal{E}_h  + c_0 \mathcal{D}_u + C^{\frac{3}{2}}_{\mathrm{in}} \sqrt{\mathcal{D}_u} + \sqrt{C_{\mathrm{in}} \mathcal{D}_u} \mathcal{E}_h  + \mathrm{E}_u \mathcal{CK}_h  + c_0 \sqrt{\mathcal{D}_u \mathrm{E}_u}  +\frac{1}{32C} \sqrt{\mathcal{D}_u   \mathcal{CK}_h} + C_{\mathrm{in}} . 
\end{align*} 
The same strategy lets us control \(\mathrm{I}_{6,3}, \mathrm{I}_{5,4}\) and \(\mathrm{I}_{6,4} . \)  We merely write down the final estimate:
\begin{align*}
    \lvert \mathrm{I}_{6,3} \rvert &\lesssim  \alpha \lVert  v \rVert_{\mathcal{G}^{s}_2} \left(( \lVert  v \rVert_{\mathcal{G}^{s}_2} + \sqrt{C_{\mathrm{in}}})(\lVert \partial_y  h \rVert_{\mathcal{G}^{2}_2} + \sqrt{c_0})^2 \right) \\
    & \qquad\qquad +  \alpha \lVert  v \rVert_{\mathcal{G}^{s}_2} \left((\lVert  \partial_y v \rVert_{\mathcal{G}^{2}_2} +\sqrt{ C_{\mathrm{in}}} )(\lVert  \partial_y h \rVert_{\mathcal{G}^{2}_2} + \sqrt{c_0})(\lVert  h \rVert_{\mathcal{G}^{s}_2} + \sqrt{c_0})\right) \\
    &\lesssim  \mathcal{D}_u \mathcal{E}_h  + c_0 \mathcal{D}_u + C^{\frac{3}{2}}_{\mathrm{in}} \sqrt{\mathcal{D}_u} + \sqrt{C_{\mathrm{in}} \mathcal{D}_u} \mathcal{E}_h  + \mathrm{E}_u \mathcal{CK}_h  +c_0\sqrt{\mathcal{D}_u \mathrm{E}_u}  +\frac{1}{32C} \sqrt{\mathcal{D}_u   \mathcal{CK}_h} + C_{\mathrm{in}} ,  \\
    \lvert \mathrm{I}_{5,4} \rvert &\lesssim \alpha \lVert  u \rVert_{\mathcal{G}^{s}_2} \left( \lVert  u^\varepsilon  \rVert_{\mathcal{G}^{s}_2} \lVert h^\varepsilon  \rVert_{\mathcal{G}^{2}_\infty} + \lVert  u^\varepsilon  \rVert_{\mathcal{G}^{2}_\infty} \lVert  h^\varepsilon  \rVert_{\mathcal{G}^{s}_2}\right) \\
    &\lesssim  \mathcal{D}_u \sqrt{\mathcal{E}_h} + \mathcal{D}_u \sqrt{c_0} + \sqrt{\mathcal{D}_u C_{\mathrm{in}} \mathcal{E}_h} + \sqrt{\mathcal{D}_u C_{\mathrm{in}}} + \sqrt{\mathcal{D}_u \mathrm{E}_u \mathcal{CK}_h} + \sqrt{\mathcal{D}_u \mathrm{E}_u c_0} +\frac{1}{32C} \sqrt{\mathcal{D}_u \mathcal{CK}_h } \\
    \lvert \mathrm{I}_{6,4} \rvert &\lesssim \alpha \lVert v \rVert_{\mathcal{G}^{s}_2} \left( (\lVert v \rVert_{\mathcal{G}^{s}_2} + \sqrt{C_{\mathrm{in}}})( \lVert  \partial_y h \rVert_{\mathcal{G}^{2}_2} + c_0) + ( \lVert  \partial_y v \rVert_{\mathcal{G}^{2}_2} + \sqrt{C_{\mathrm{in}}})(\lVert  h \rVert_{\mathcal{G}^{s}_2} + \sqrt{c_0})\right) \\
    &\lesssim  \mathcal{D}_u \sqrt{\mathcal{E}_h} + \mathcal{D}_u \sqrt{c_0} + \sqrt{\mathcal{D}_u C_{\mathrm{in}} \mathcal{E}_h} + \sqrt{\mathcal{D}_u C_{\mathrm{in}}} + \sqrt{\mathcal{D}_u \mathrm{E}_u \mathcal{CK}_h} + \sqrt{\mathcal{D}_u \mathrm{E}_u c_0} + \frac{1}{32C}\sqrt{\mathcal{D}_u \mathcal{CK}_h}. 
\end{align*}

Combining the above inequalities, we obtain:
\begin{align*}
      \mathrm{I}_5 + \mathrm{I}_6 &=   \mathrm{I}_{5,2} + \mathrm{I}_{5,3} + \mathrm{I}_{5,4} + \mathrm{I}_{6,2} + \mathrm{I}_{6,3} + \mathrm{I}_{6,4} + \mathrm{I}_{5,L} + \mathrm{I}_{6,L} \\
      &\lesssim  C^{\frac{3}{2}}_{\mathrm{in}} \sqrt{\mathcal{D}_u}  + \sqrt{\mathcal{D}_u} \mathcal{E}_h + \sqrt{\mathcal{D}_u \mathcal{E}_h \mathcal{CK}_h} +\frac{1}{32 C} \sqrt{  \mathcal{D}_u {\mathcal{CK}_h}} + \mathcal{D}_u \mathcal{E}_h  + c_0 \mathcal{D}_u + \sqrt{C_{\mathrm{in}} \mathcal{D}_u} \mathcal{E}_h \\
      &\qquad + \mathrm{E}_u \mathcal{CK}_h  + c_0\sqrt{\mathcal{D}_u \mathrm{E}_u}  + C_{\mathrm{in}} + \mathcal{D}_u \sqrt{\mathcal{E}_h} + \sqrt{C_{\mathrm{in}} \mathcal{D}_u \mathcal{E}_h} + \sqrt{\mathcal{D}_u \mathrm{E}_u \mathcal{CK}_h}  + \sqrt{c_0\mathcal{D}_u \mathrm{E}_u }  . 
\end{align*}  
Using Young's inequality and \( c_0 \ll 1\), we get 
\begin{align}\label{eq:Control on I5 and I6}
   \begin{split}
    \lvert \mathrm{I}_5 \rvert &+  \lvert \mathrm{I}_6 \rvert - \frac{1}{16} \mathcal{D}_u - \frac{1}{16}{\mathcal{CK}_h} \\   &\lesssim  C^3_{\mathrm{in}} +C_{\mathrm{in}} (\mathcal{E}_h)^2 + \sqrt{\mathcal{D}_u \mathcal{CK}_h} \left(\sqrt{\mathrm{E}_u}+ \sqrt{\mathcal{E}_h} \right) + \mathcal{D}_u \mathcal{E}_h + \mathrm{E}_u \mathcal{CK}_h    + \mathrm{E}_u + \mathcal{D}_u \sqrt{\mathcal{E}_h} +C_{\mathrm{in}} \mathcal{E}_h  \\
    &\lesssim  C^3_{\mathrm{in}} + C_{\mathrm{in}}(\mathcal{E}_h)^2 + \mathrm{E}_u + (\mathcal{D}_u + \mathcal{CK}_h)(\mathcal{E}_h + \sqrt{\mathcal{E}_h} + \sqrt{\mathrm{E}_u} +  \mathrm{E}_u ) 
   \end{split}
\end{align}  
 Combining \eqref{eq:Control on I1 and I3}, \eqref{eq:Control on I2}, \eqref{eq:Control on I4}, and \eqref{eq:Control on I5 and I6}, we get 
 \begin{align*}
     &\frac{1}{2}\frac{\mathrm{d}}{\mathrm{d}t} \mathcal{E}_u  +   \mathcal{CK}_u  +  \mathcal{D}_u +  {\mathcal{CK}_h} - \frac{1}{16}\mathcal{D}_u - \frac{1}{16} \mathcal{CK}_h \\
     & \quad\lesssim (\mathcal{D}_u  + \mathcal{CK}_h)(\sqrt{\mathcal{E}_u} + \mathcal{E}_h + \sqrt{\mathcal{E}_h} + \sqrt{\mathrm{E}_u}  + \mathrm{E}_u   ) + ( \mathcal{E}_u  + \mathrm{E}_u  + C_{\mathrm{in}} )^3  . 
 \end{align*} 
 Concisely, we have 
\begin{equation}\label{eq:Bootstrap on Eu}
  \frac{1}{2}  \frac{\mathrm{d}}{\mathrm{d}t} \mathcal{E}_u  + \frac{1}{2}\mathcal{D}_u +   \mathcal{CK}_u  + \frac{1}{2}{\mathcal{CK}_h}  \lesssim  ( \mathcal{D}_u + \mathcal{CK}_h)(\sqrt{\mathcal{E}_h} + \sqrt{\mathrm{E}_u} +  \mathcal{E}_h + \mathrm{E}_u )  + (\mathcal{E}_u + \mathrm{E}_u + C_{\mathrm{in}})^3. 
\end{equation} 
\subsubsection{Control on \texorpdfstring{\(\mathcal{E}_h\)}{Eh}} 

So applying the operator \( \mathrm{A}_\sigma\) on the equation for \(h\) where \(\sigma := s - \frac{3}{4}\) and using the above for \(\sigma ' := \sigma\), we have 
\begin{align*}
    \beta \partial_{t} [ \mathrm{A}_{\sigma} ( \partial_t h )]& + \beta \dot{\lambda} \langle  \partial_x \rangle^{\frac{1}{2}} [ \mathrm{A}_{\sigma} (\partial_t h )]  + \mathrm{A}_{\sigma} (\partial_t h ) - \gamma ( \varepsilon ^2 \partial_{xx} \mathrm{A}_{\sigma} h + \partial_{yy} \mathrm{A}_{\sigma} h) \\
    & \qquad\qquad+ \mathrm{A}_{\sigma} ( u^\varepsilon  \partial_x h^\varepsilon ) + \mathrm{A}_{\sigma} ( v^\varepsilon  \partial_y h^\varepsilon ) = 0 
\end{align*} 

Inner-producting with \(\mathrm{A}_{\sigma} (\partial_t h)\), we obtain 
\begin{align*}
   & \frac{\beta}{2} \frac{\mathrm{d}}{\mathrm{d}t} \lVert  \partial_t h \rVert_{\mathcal{G}^{\sigma}_2}^2 + \beta \dot{\lambda} \lVert  \partial_t h \rVert_{\mathcal{G}^{\sigma + \frac{1}{4}}_2}^2 + \lVert  \partial_t h \rVert_{\mathcal{G}^{\sigma}_2}^2\\
     & = \gamma \langle  \mathrm{A}_{\sigma} ( \varepsilon ^2 \partial_{xx} h  + \partial_{yy} h ) , \mathrm{A}_{\sigma} (\partial_t h) \rangle  + \langle  \mathrm{A}_{\sigma} ( u^\varepsilon  \partial_x h^\varepsilon   + v^\varepsilon  \partial_y h^\varepsilon ), \mathrm{A}_{\sigma} (\partial_t h ) \rangle \\
    & = : \Pi  _1 + \Pi _2
\end{align*} 
Using integration by parts, we get 
\begin{align*}
    \Pi _1 &= \gamma \langle  \mathrm{A}_{\sigma} ( \varepsilon ^2 \partial_{xx} h  + \partial_{yy} h) , \partial_t [ \mathrm{A}_{\sigma} h]  \rangle + \gamma \langle  \mathrm{A}_{\sigma} ( \varepsilon ^2 \partial_{xx} h + \partial_{yy} h) , \dot{\lambda} \langle  \partial_x \rangle^{\frac{1}{2}} \mathrm{A}_{\sigma} h \rangle \\ 
    &= - \frac{\gamma}{2} \frac{\mathrm{d}}{\mathrm{d}t} \lVert  (\varepsilon \partial_x, \partial_y) h \rVert_{\mathcal{G}^{\sigma}_2}^2 - \gamma \dot{\lambda} \lVert  (\varepsilon \partial_x , \partial _y) h \rVert_{\mathcal{G}^{\sigma + \frac{1}{4}}_2}^2. 
\end{align*} 
We now deal with the term \(\dot{\lambda} \lVert  \partial_t h \rVert_{\mathcal{G}^{\sigma + \frac{1}{4}}_2}^2\) using the method introduced in \mbox{\cite{MR4753440}}: 

\begin{lemma}\label{lem:|∂ₜh|ₛ expansion}
    For \( \phi \in \mathcal{G}^{s}_2\) and \( \partial_t [ \mathrm{A}_{\sigma + \frac{1}{4}} \phi] \in L^2\), we have \( \partial_t \phi \in \mathcal{G}^{\sigma + \frac{1}{4}}_2\) satisfying 
    \[  \lVert  \partial_t \phi \rVert_{\mathcal{G}^{\sigma + \frac{1}{4}}_2}^2 = \lVert  \partial_t [ \mathrm{A}_{\sigma + \frac{1}{4}} \phi] \rVert_{L^2}^2 + \frac{\mathrm{d}}{\mathrm{d}t} \left( \dot{\lambda}^2 \lVert  \phi \rVert_{\mathcal{G}^{\sigma + \frac{1}{2}}_2}^2\right) - 2 \ddot{\lambda} \dot{\lambda} \lVert  \phi \rVert_{\mathcal{G}^{\sigma + \frac{1}{2}}_2}^2 + \dot{\lambda}^3 \lVert  \phi \rVert_{\mathcal{G}^{s}_2}^2 ,  
    \] 
    where \(\sigma := s - \frac{3}{4}\). 
\end{lemma}
\begin{proof} Using \eqref{eq:Differentiating Aₛ}, a routine calculation shows that 
    \begin{align*}
         \lVert  \partial_t \phi\rVert_{\mathcal{G}^{\sigma + \frac{1}{4}}_2}^2 &= \lVert  \mathrm{A}_{\sigma + \frac{1}{4}} ( \partial_t \phi) \rVert_{L^2}^2  = \lVert \partial_t [ \mathrm{A}_{\sigma + \frac{1}{4}} \phi ] +  \dot{\lambda} \mathrm{A}_{\sigma + \frac{3}{4}} \phi \rVert_{L^2}^2 \\
          &=   \lVert  \partial_t[ \mathrm{A}_{\sigma + \frac{1}{4}} \phi] \rVert_{L^2}^2 + 2   \langle \partial_t[ \mathrm{A}_{\sigma + \frac{1}{4}} \phi] , \dot{\lambda} \mathrm{A}_{\sigma + \frac{3}{4}} \phi  \rangle +   \dot{\lambda}^2 \lVert  \mathrm{A}_{\sigma + \frac{3}{4} } \phi \rVert_{L^2}^2 \\
          &=    \lVert  \partial_t[ \mathrm{A}_{\sigma + \frac{1}{4}} \phi] \rVert_{L^2}^2   + \dot{\lambda} \frac{\mathrm{d}}{\mathrm{d}t} \lVert  \phi \rVert_{\mathcal{G}^{\sigma + \frac{1}{2}}_2}^2 +  \dot{\lambda} ^2 \lVert  \phi \rVert_{\mathcal{G}^{\sigma + \frac{3}{4}}_2}^2 \\
          &=   \lVert  \partial_t[ \mathrm{A}_{\sigma + \frac{1}{4}} \phi] \rVert_{L^2}^2 + \frac{1}{\dot{\lambda}}\frac{\mathrm{d}}{\mathrm{d}t} \left(\dot{\lambda} ^2 \lVert \phi \rVert_{\mathcal{G}^{\sigma + \frac{1}{2}}_2}^2\right) - 2\ddot{\lambda}   \lVert  \phi \rVert_{\mathcal{G}^{\sigma + \frac{1}{2}}_2}^2 +  \dot{\lambda}^2 \lVert  \phi \rVert_{\mathcal{G}^{s}_2}^2. \qedhere 
      \end{align*} 
\end{proof}  
Applying Lemma \ref{lem:|∂ₜh|ₛ expansion} by taking \( \phi = h\), we get 
\begin{align*}
  \frac{\beta}{2} \dot{\lambda} \lVert  \partial_t h \rVert_{\mathcal{G}^{\sigma + \frac{1}{4}}_2}^2 &= \frac{\beta}{2}\dot{\lambda} \lVert  \partial_t[ \mathrm{A}_{\sigma + \frac{1}{4}} h] \rVert_{L^2}^2 +  \frac{\beta}{2}\frac{\mathrm{d}}{\mathrm{d}t} \left(\dot{\lambda} ^2 \lVert h \rVert_{\mathcal{G}^{\sigma + \frac{1}{2}}_2}^2\right) - \beta\ddot{\lambda} \dot{\lambda} \lVert  h \rVert_{\mathcal{G}^{\sigma + \frac{1}{2}}_2}^2 +  \frac{\beta}{2}\dot{\lambda}^3 \lVert  h \rVert_{\mathcal{G}^{s}_2}^2. 
\end{align*} 

Combining the above, we have 
\begin{align*}
    &\frac{1}{2}\frac{\mathrm{d}}{\mathrm{d}t} \left( \lVert  (\sqrt{\beta}\partial _t,\sqrt{\gamma} \varepsilon \partial_x ,\sqrt{\gamma}\partial_y ) h \rVert_{\mathcal{G}^{\sigma}_2}^2 + \beta \dot{\lambda}^2 \lVert h \rVert_{\mathcal{G}^{\sigma + \frac{1}{2}}_2}^2 \right)  + \frac{\beta}{2} \dot{\lambda} \lVert \partial_t  h \rVert_{\mathcal{G}^{\sigma + \frac{1}{4}}_2}^2   \\
    &\qquad\qquad+ \frac{\beta}{2} \left(  \dot{\lambda}\lVert  \partial_t [ \mathrm{A}_{\sigma + \frac{1}{4}} h] \rVert_{L^2}^2 + \dot{\lambda}^3\lVert  h \rVert_{\mathcal{G}^{s}_2}^2\right)   + \lVert \partial_t h \rVert_{\mathcal{G}^{\sigma}_2}^2  + \gamma\dot{\lambda} \lVert  (\varepsilon \partial_x, \partial_y) h \rVert_{\mathcal{G}^{\sigma + \frac{1}{4}}_2}^2  \\
    &= \Pi _2 +{ \ddot{\lambda} \dot{\lambda} } \lVert  h \rVert_{\mathcal{G}^{\sigma + \frac{1}{2}}_2}^2  
\end{align*} 
Letting \(\ddot{\lambda}  = 0\) we can drop the second term. For the former, letting \( \beta \dot{\lambda} \geq 1\) for \(t \leq T_0\) and using Corollary \ref{cor:Gevrey-Holder with δ}:
\begin{align*}
     \lvert \Pi _2 \rvert &\lesssim \gamma\lVert  \partial_t h \rVert_{\mathcal{G}^{\sigma + \frac{1}{4}}_2} \left( \lVert u^\varepsilon   \rVert_{\mathcal{G}^{\sigma - \frac{1}{4}}_2} \lVert \partial_x h^\varepsilon  \rVert_{\mathcal{G}^{2}_\infty} + \lVert  u^\varepsilon  \rVert_{\mathcal{G}^{2}_\infty} \lVert  \partial_x h^\varepsilon  \rVert_{\mathcal{G}^{\sigma - \frac{1}{4}}_2} \right) \\
     & \qquad\qquad + \gamma \lVert  \partial_t h \rVert_{\mathcal{G}^{\sigma + \frac{1}{4}}_2} \left( \lVert  v^\varepsilon  \rVert_{\mathcal{G}^{\sigma  - \frac{1}{4}}_\infty} \lVert  \partial_y h^\varepsilon  \rVert_{\mathcal{G}^{2}_2} + \lVert  v^\varepsilon  \rVert_{\mathcal{G}^{2}_\infty} \lVert  \partial_ y h^\varepsilon  \rVert_{\mathcal{G}^{\sigma - \frac{1}{4}}_2}\right)  \\
     &\lesssim  \lVert  \partial_t h \rVert_{\mathcal{G}^{\sigma + \frac{1}{4}}_2} ( \lVert  u \rVert_{\mathcal{G}^{s}_2} + \sqrt{C_{\mathrm{in}}} ) \left(\lVert  \partial_y h \rVert_{\mathcal{G}^{\sigma - \frac{1}{4}}_2} +  \lVert  h \rVert_{\mathcal{G}^{s}_2}  + \sqrt{C^h_{\mathrm{in}}} \right) \\
     &\lesssim \frac{1}{\sqrt{\dot{\lambda}}}  \sqrt{\mathcal{CK}_h} ( \sqrt{\mathcal{E}_u} + \sqrt{C_{\mathrm{in}}} ) ( \sqrt{\mathcal{CK}_h} + \sqrt{C^h_{\mathrm{in}}} )\\
     &\lesssim  \mathcal{CK}_h \sqrt{\mathcal{E}_u} + \frac{\sqrt{C_{\mathrm{in}}}}{\sqrt{\dot{\lambda}}}\mathcal{CK}_h   + \sqrt{\mathcal{CK}_h   \mathcal{E}_u} +\frac{\sqrt{C_{\mathrm{in}}}}{\sqrt{\dot{\lambda}}} \sqrt{\mathcal{CK}_h}. 
\end{align*} 
Moreover, using \(\dot{\lambda} \geq 32 C C_{\mathrm{in}}\), we get  
\begin{equation}\label{eq:Bootstrap on Ehₙ}
    \frac{\mathrm{d}}{\mathrm{d}t} \mathcal{E}_h  + \frac{1}{2} \mathcal{CK}_h  + \mathcal{D}_h - \frac{1}{16} \mathcal{CK}_h \lesssim  \mathcal{CK}_h \sqrt{\mathcal{E}_u} + \mathcal{E}_u + C_{\mathrm{in}}.
\end{equation} 
\subsubsection{Control on \texorpdfstring{\(\mathrm{E}_u\)}{Eu}} 

A direct calculation shows that 
\begin{align*}
    &\lVert  \partial_t  (u, \varepsilon v) \rVert_{\mathcal{G}^{s - 2}_2}^2  + \langle  \mathrm{A}_{s - 2}(u^\varepsilon  \partial_x u^\varepsilon + v^\varepsilon  \partial_y u^\varepsilon ) , \mathrm{A}_{s - 2} (\partial_t u) \rangle \\
    &\qquad+ \varepsilon ^2 \langle  \mathrm{A}_{s - 2} ( u^\varepsilon  \partial_x v^\varepsilon  + v^\varepsilon  \partial_y v^\varepsilon _ n) , \mathrm{A}_{s - 2} ( \partial_t v) \rangle    - \langle  \mathrm{A}_{s - 2} ( \varepsilon ^2 \partial_{xx} u^\varepsilon  + \partial_{yy} u^\varepsilon ) , \mathrm{A}_{s - 2} ( \partial_t u) \rangle \\
    &\qquad  - \varepsilon ^2 \langle  \mathrm{A}_{s - 2} ( \varepsilon ^2 \partial_{xx} v^\varepsilon   + \partial_{yy} v^\varepsilon  ), \mathrm{A}_{s -2} ( \partial_t v) \rangle   \\
    &= \alpha \langle  \mathrm{A}_{s - 2} ( f^\varepsilon  + f^\varepsilon  h^\varepsilon   - u^\varepsilon  - u^\varepsilon  (h^\varepsilon )^2 - 2 u^\varepsilon  h^\varepsilon ) , \mathrm{A}_{s - 2} ( \partial_t u) \rangle \\ 
    & \qquad\qquad - \alpha \langle  \mathrm{A}_{s - 2} ( e^\varepsilon  + e^\varepsilon  h^\varepsilon  + v^\varepsilon  + v^\varepsilon  (h^\varepsilon )^2 - 2 v^\varepsilon  h^\varepsilon ) , \mathrm{A}_{s - 2} ( \partial_t v) \rangle. 
\end{align*} 
Using \eqref{eq:Differentiating Aₛ}, we have 
\begin{align*}
      &-\langle  \mathrm{A}_{s - 2} ( \varepsilon ^2 \partial_{xx} u + \partial_{yy} u) , \mathrm{A}_{s - 2} ( \partial_t u) \rangle \\
      &\qquad\qquad= -\langle  \mathrm{A}_{s - 2} ( \varepsilon ^2 \partial_{xx} u + \partial_{yy} u) , \partial_t [ \mathrm{A}_{s - 2} u]  \rangle - \dot{\lambda} \langle  \mathrm{A}_{\sigma - 1 } ( \varepsilon ^2 \partial_{xx} u + \partial_{yy} u) , \mathrm{A}_{\sigma - 1  } u \rangle \\
      &\qquad\qquad=  \frac{1}{2}\frac{\mathrm{d}}{\mathrm{d}t} \lVert (\varepsilon \partial_x, \partial_y) u \rVert_{\mathcal{G}^{s - 2}_2}^2  + \dot{\lambda} \lVert  ( \varepsilon \partial_x, \partial_y) u \rVert_{\mathcal{G}^{\sigma - 1}_2}^2, \\
      &-\varepsilon ^2\langle  \mathrm{A}_{s - 2} ( \varepsilon ^2 \partial_{xx} v + \partial_{yy} v) , \mathrm{A}_{s - 2} ( \partial_t v) \rangle \\ 
      & \qquad\qquad=  \frac{1}{2}\frac{\mathrm{d}}{\mathrm{d}t} \lVert  (\varepsilon \partial_x, \partial_y) \varepsilon v \rVert_{\mathcal{G}^{s}_2}^2 + \dot{\lambda} \lVert  ( \varepsilon \partial_x, \partial_y) \varepsilon v \rVert_{\mathcal{G}^{\sigma  - 1}_2}^2  \\
      &\qquad\qquad= \frac{1}{2} \frac{\mathrm{d}}{\mathrm{d}t}  \lVert  (\varepsilon ^2 \partial_x v , \varepsilon \partial_x u) \rVert_{\mathcal{G}^{s}_2}^2 + \dot{\lambda} \lVert  (\varepsilon ^2 \partial_x v , \varepsilon \partial_x u) \rVert_{\mathcal{G}^{\sigma - 1}_2}^2. 
\end{align*} 
We end up having 
\begin{align*}
   & \lVert  \partial_t(u, \varepsilon v) \rVert_{\mathcal{G}^{s - 2}_2}^2  + \frac{1}{2}\frac{\mathrm{d}}{\mathrm{d}t}   \lVert ( 2 \varepsilon \partial_x u , \partial_y u , \varepsilon ^2 \partial_x v)\rVert_{\mathcal{G}^{s - 2}_2}^2  + \dot{\lambda}  \lVert  (2 \varepsilon \partial_x u , \partial _y u , \varepsilon ^2 \partial_x v) \rVert_{\mathcal{G}^{\sigma - 1}_2}^2  \\
    &\qquad\qquad= \alpha \langle  \mathrm{A}_{s - 2} f^\varepsilon  , \mathrm{A}_{s - 2}( \partial_t u ) \rangle   +\alpha \langle  \mathrm{A}_{s - 2} (f^\varepsilon  h^\varepsilon ), \mathrm{A}_{s  -2} ( \partial_t u) \rangle - { \alpha\langle \mathrm{A}_{s - 2} u^\varepsilon  , \mathrm{A}_{s - 2} ( \partial_t u) \rangle}  \\
    &\qquad\qquad \quad- \alpha \langle  \mathrm{A}_{s - 2} ( u^\varepsilon  ( h^\varepsilon )^2) , \mathrm{A}_{ s- 2} ( \partial_t u)  \rangle    - 2\alpha \langle  \mathrm{A}_{s - 2}( u^\varepsilon  h^\varepsilon  ) , \mathrm{A}_{ s - 2} ( \partial _t u )  \rangle \\
    &\qquad\qquad\quad - \alpha \langle  \mathrm{A}_{ s -2} ( e^\varepsilon _ n), \mathrm{A}_{ s- 2} ( \partial_t v) \rangle - \alpha \langle  \mathrm{A}_{ s - 2} ( e^\varepsilon  h^\varepsilon  ), \mathrm{A}_{ s- 2} ( \partial_t v) \rangle -  { \alpha \langle  \mathrm{A}_{s - 2} v^\varepsilon  , \mathrm{A}_{ s- 2} (\partial_t v) \rangle}  \\
    &\qquad\qquad\quad - \alpha \langle  \mathrm{A}_{ s - 2} ( v^\varepsilon  (h^\varepsilon )^2) ,\mathrm{A}_{ s- 2} ( \partial_t v)  \rangle -2 \alpha \langle  \mathrm{A}_{ s - 2} ( v^\varepsilon  h^\varepsilon ) , \mathrm{A}_{ s- 2} ( \partial_t v) \rangle \\
   &\qquad\qquad\quad - \langle  \mathrm{A}_{ s- 2} u^\varepsilon  \partial_x u^\varepsilon    , \mathrm{A}_{ s- 2} ( \partial _t u) \rangle - \langle  \mathrm{A}_{ s- 2} ( v^\varepsilon  \partial_y u^\varepsilon  ) , \mathrm{A}_{ s- 2} ( \partial_t u) \rangle \\
   & \qquad\qquad\quad - \varepsilon ^2\langle  \mathrm{A}_{ s- 2} ( u^\varepsilon  \partial_x v^\varepsilon   + v^\varepsilon  \partial_y v^\varepsilon  ) , \mathrm{A}_{ s-2} ( \partial_t v)  \rangle  \\  
   &=: \mathrm{J}_{1} + \mathrm{J}_2 + \mathrm{J}_3 + \mathrm{J} _4 + \mathrm{J}_5 + \mathrm{K}_1 + \mathrm{K}_2 + \mathrm{K} _3 + \mathrm{K}_4 + \mathrm{K}_5 + \mathcal{N} 
\end{align*} 
The terms \( \mathrm{K}_i\) contains \(\partial_t v\) which loses a derivative through \( v = - \partial_x\int_{0}^y  u\). {However, the situation is salvageable as the terms it is inner-producted with have a \(\mathcal{G}^{s-1}_2\) control like \(v^\varepsilon\)}  (via divergence-free condition) and \(e^\varepsilon\) (via Proposition \ref{prop:Elliptic Estimates}). The same estimate holds for \(\mathrm{J}_i\) as there is no loss of derivative through \( \partial_t u\). As such, we have using Corollary \ref{cor:Gevrey-Holder with δ} and \(s - 1 < \sigma : \)
\begin{align*}
    \lvert \mathrm{J}_1 \rvert +  \lvert \mathrm{K}_1 \rvert &\lesssim  \alpha \lVert  \partial_t ( u , v) \rVert_{\mathcal{G}^{s-3}_2} \lVert  (e^\varepsilon  , f^\varepsilon ) \rVert_{\mathcal{G}^{ s- 1}_2} \lesssim   \lVert  \partial_t u \rVert_{\mathcal{G}^{s-2}_2} \left(\lVert  \partial_t h \rVert_{\mathcal{G}^{\sigma}_2} + \sqrt{c_0}\right) \\ 
    &\lesssim \sqrt{\mathrm{D}_u} ( \sqrt{\mathcal{E}_h} + \sqrt{c_0}) \\
    \lvert \mathrm{J}_3 \rvert +  \lvert \mathrm{K}_3 \rvert &\lesssim  \alpha \lVert  (u^\varepsilon  , v^\varepsilon ) \rVert_{\mathcal{G}^{s - 1}_2} \lVert  \partial_t (u  , v) \rVert_{\mathcal{G}^{s - 3}_2} \lesssim   \lVert  u^\varepsilon \rVert_{\mathcal{G}^{s}_2}  \lVert  \partial_t u \rVert_{\mathcal{G}^{s -2}_2}  \\  
    &\lesssim \sqrt{\mathrm{D}_u \mathcal{E}_h} + \sqrt{\mathrm{D}_u C_{\mathrm{in}}}. 
\end{align*} 
The estimate on the remaining quantities follow similarly as above but requires an additional application of Proposition \ref{prop:Gevrey-Holder}. We elaborate the first estimate fully and then merely present the final estimate for the remaining quantities: 
\begin{align*} 
    \lvert \mathrm{J}_2 \rvert +  \lvert \mathrm{K}_2 \rvert &\lesssim  \alpha \lVert  \partial_t u \rVert_{\mathcal{G}^{s - 3}_2} \left( \lVert  f^\varepsilon  \rVert_{\mathcal{G}^{s - 1}_2} \lVert  h^\varepsilon  \rVert_{\mathcal{G}^{2}_2} + \lVert  f^\varepsilon  \rVert_{\mathcal{G}^{2}_\infty} \lVert  h^\varepsilon  \rVert_{\mathcal{G}^{s - 1}_2}\right) \\
    & \qquad\qquad + \alpha \lVert  \partial_t v  \rVert_{\mathcal{G}^{s - 3}_2} \left( \lVert  e^\varepsilon  \rVert_{\mathcal{G}^{s -1}_2} \lVert h^\varepsilon  \rVert_{\mathcal{G}^{2}_\infty} + \lVert  e^\varepsilon  \rVert_{\mathcal{G}^{2}_\infty} \lVert  h^\varepsilon  \rVert_{\mathcal{G}^{s - 1}_2}\right) \\
    &\lesssim   \lVert  \partial_t u \rVert_{\mathcal{G}^{ s- 2}_2} \left( \lVert  (f ^\varepsilon, e^\varepsilon) \rVert_{\mathcal{G}^{s - 1}_2} \lVert h^\varepsilon \rVert_{\mathcal{G}^{2}_\infty} + \lVert  (f^\varepsilon , e^\varepsilon) \rVert_{\mathcal{G}^{2}_\infty} \lVert  h^\varepsilon \rVert_{\mathcal{G}^{s - 1}_2}   \right) \\
    &\lesssim    \lVert  \partial_t u \rVert_{\mathcal{G}^{s - 2}_2}   \lVert  \partial_t h^\varepsilon \rVert_{\mathcal{G}^{\sigma}_2} \lVert h^\varepsilon \rVert_{\mathcal{G}^{ s-1}_2}  \lesssim  \sqrt{\mathrm{D}_u} ( \sqrt{\mathcal{E}_h} + \sqrt{C_{\mathrm{in}}}) ( \sqrt{\mathcal{E}_h} + \sqrt{C^h_{\mathrm{in}}})\\
    &\lesssim  \sqrt{\mathrm{D}_u}( \mathcal{E}_h + \sqrt{C_{\mathrm{in}}\mathcal{E}_h} + \sqrt{C_{\mathrm{in}}})\\
      \lvert \mathrm{J}_4 \rvert +  \lvert \mathrm{K}_4 \rvert & \lesssim   \alpha \lVert   \partial_t (u, v) \rVert_{\mathcal{G}^{s- 3}_2} \left( \lVert (u^\varepsilon ,v^\varepsilon ) \rVert_{\mathcal{G}^{s - 1}_2} \lVert  h^\varepsilon  \rVert_{\mathcal{G}^{2}_\infty}^2  + \lVert  (u^\varepsilon ,v^\varepsilon )  \rVert_{\mathcal{G}^{2}_\infty} \lVert  h^\varepsilon  \rVert_{\mathcal{G}^{2}_\infty} \lVert  h^\varepsilon   \rVert_{\mathcal{G}^{s - 1}_2}\right) \\
      &\lesssim  \sqrt{\mathrm{D}_u}(\mathcal{E}_h + C^h_{\mathrm{in}})\left(  \sqrt{\mathcal{E}_u} +  \sqrt{\mathrm{E}_u} + \sqrt{C_{\mathrm{in}}}\right) ,\\ 
       \lvert \mathrm{J}_5 \rvert +  \lvert \mathrm{K}_5 \rvert &\lesssim  \alpha \lVert  \partial_t (u , v) \rVert_{\mathcal{G}^{s - 3}_2} \left( \lVert  (u^\varepsilon  , v^\varepsilon )\rVert_{\mathcal{G}^{s - 1}_2} \lVert  h^\varepsilon  \rVert_{\mathcal{G}^{2}_\infty} + \lVert (u^\varepsilon  , v^\varepsilon )\rVert_{\mathcal{G}^{2}_\infty} \lVert  h^\varepsilon  \rVert_{\mathcal{G}^{s-1}_2}\right) \\
       &\lesssim  \sqrt{C_{\mathrm{in}}\mathrm{D}_u} ( \sqrt{\mathcal{E}_u \mathcal{E}_h} + \sqrt{\mathcal{E}_u} + \sqrt{ \mathcal{E}_h} + \sqrt{C^h_{\mathrm{in}}}), 
\end{align*} 
For the nonlinear terms, there is no need to distribute the derivatives as there is no derivative loss for \( \varepsilon \partial_t v\). So using Proposition \ref{prop:Gevrey-Holder}:
\begin{align*} 
         \lvert \mathcal{N} \rvert &\lesssim  \lVert  \partial_t u \rVert_{\mathcal{G}^{s - 2}_2} \left( \lVert  u^\varepsilon  \rVert_{\mathcal{G}^{s-2}_2} \lVert  \partial_x u^\varepsilon  \rVert_{\mathcal{G}^{2}_\infty} + \lVert  u^\varepsilon  \rVert_{\mathcal{G}^{2}_\infty} \lVert  \partial_x u^\varepsilon \rVert_{\mathcal{G}^{s-2}_2} + \lVert  v^\varepsilon  \rVert_{\mathcal{G}^{s -2}_\infty} \lVert  \partial_y u^\varepsilon  \rVert_{\mathcal{G}^{2}_2} \right)  \\
         &\qquad\qquad  + \lVert  \partial_t u \rVert_{\mathcal{G}^{s - 2}_2}\lVert  v^\varepsilon  \rVert_{\mathcal{G}^{2}_\infty} \lVert  \partial_y u^\varepsilon  \rVert_{\mathcal{G}^{s-2}_2}\\
         & \qquad\qquad  + \lVert \varepsilon \partial_t  v  \rVert_{\mathcal{G}^{s-2}_2} \left( \lVert  (u^\varepsilon  , v^\varepsilon ) \rVert_{\mathcal{G}^{s - 2}_2} \lVert \varepsilon \nabla v^\varepsilon  \rVert_{\mathcal{G}^{2}_{\infty}}  + \lVert  (u^\varepsilon  ,v^\varepsilon ) \rVert_{\mathcal{G}^{2}_\infty} \lVert \varepsilon  \nabla v^\varepsilon  \rVert_{\mathcal{G}^{s -2}_2}\right)   \\
         &\lesssim  \lVert  \partial_t(u , \varepsilon v) \rVert_{\mathcal{G}^{s-2}_2}  \left( \lVert  u^\varepsilon  \rVert_{\mathcal{G}^{s-1}_2} \lVert  \partial_y u^\varepsilon  \rVert_{\mathcal{G}^{3}_2}+ \lVert  u^\varepsilon  \rVert_{\mathcal{G}^{s}_2}\lVert  \partial_y u^\varepsilon  \rVert_{\mathcal{G}^{s-2}_2}  + \lVert \partial_y u^\varepsilon  \rVert_{\mathcal{G}^{3}_2} \lVert  \varepsilon \partial_x u^\varepsilon  \rVert_{\mathcal{G}^{s - 2}_2}  \right)  \\ 
         &\lesssim \sqrt{\mathcal{D}_u}(\sqrt{\mathrm{E}_u} + \sqrt{C_{\mathrm{in}}}) (\sqrt{\mathcal{E}_u} +\sqrt{C_{\mathrm{in}}}   )  \lesssim  \sqrt{C_{\mathrm{in}} \mathrm{D}_u} ( \mathcal{E}_u + \mathrm{E}_u + C_{\mathrm{in}} ). 
\end{align*}   
Combining the inequalities using Young's inequality to collect \(\mathrm{D}_u\) on left-hand side, we conclude 
\begin{equation}\label{eq:Bootstrap on Elow}
    \frac{1}{2} \mathrm{D}_u + \frac{1}{2}\frac{\mathrm{d}}{\mathrm{d}t}\mathrm{E}_u + \mathrm{CK}_u  - \frac{1}{32}\mathrm{D}_u\lesssim  ( \mathcal{E}_u  + \mathrm{E}_u  + \mathcal{E}_h + C_{\mathrm{in}} )^2
\end{equation}  
\subsubsection{Conclusion}
Combining \eqref{eq:Bootstrap on Eu}, \eqref{eq:Bootstrap on Ehₙ} and \eqref{eq:Bootstrap on Elow}, we obtain 
\begin{align*}
    \frac{\mathrm{d}}{\mathrm{d}t} (\mathcal{E}_u + {\mathcal{E}_h} + \mathrm{E}_u) &+ \frac{1}{2}\mathcal{D}_u + \frac{1}{2}\mathcal{CK}_u + \frac{1}{2}{\mathcal{CK}_h} + \frac{1}{2}  {\mathcal{D}_h}   + \mathrm{D}_u + \mathrm{CK}_u - \frac{1}{16}({\mathcal{CK}_h} + \mathcal{D}_u + \mathrm{D}_u) \\
    &\lesssim   ( \mathcal{D}_u + \mathcal{CK}_h)(\sqrt{\mathcal{E}_h} + \sqrt{\mathrm{E}_u} +  \mathcal{E}_h + \mathrm{E}_u )  + (\mathcal{E}_u + \mathrm{E}_u + C_{\mathrm{in}})^3 . 
\end{align*} 
In particular, we have 
\begin{align}\label{eq:Full Bootstrap for LW}
    \begin{split}
    \frac{\mathrm{d}}{\mathrm{d}t}(\mathcal{E}_u +\mathrm{E}_u +{\mathcal{E}_h})  &  + \frac{1}{2} \mathcal{D}_u + \frac{1}{2}{\mathcal{CK}_h}    \\
    &\leq C  (\mathcal{D}_u + \mathcal{CK}_h)(\sqrt{\mathcal{E}_h + \mathrm{E}_u } + \mathcal{E}_h + \mathrm{E}_u ) +  C ( \mathcal{E}_u + \mathrm{E}_u + \mathcal{E}_h+ C_{\mathrm{in}})^3 . 
    \end{split}
\end{align} 
We now prove the precise description of \(T\). Motivated by the assumption on \(\lambda \), we take 
\[  \dot{\lambda}(t)  = 32 C C_{\mathrm{in}}  \implies \lambda(t) = 32 C C_{\mathrm{in}} t \implies  T' := \frac{\delta _0}{64 C C_{\mathrm{in}}}. 
\] 
We denote 
\[  T^\# := \sup \, \left\{t \in [0, T']  :  \sup_{\tau \in [0,t]}\,\left(  \mathcal{E}_u +\mathrm{E}_u +   \mathcal{E}_h + \int_{0} ^\tau (\mathcal{D}_u  + \mathcal{CK}_h)\right)\leq \frac{1}{32 C^2} \right\}. 
\] 
As the initial data is \(0\) and the energies are continuous in time so \( T^\#\) is positive. 
Using \eqref{eq:Full Bootstrap for LW}, we obtain for \(t \leq T^\# : \) 
\[   \frac{\mathrm{d}}{\mathrm{d}t}(\mathcal{E}_u +\mathrm{E}_u + \mathcal{E}_h) + \mathcal{D}_u + \mathcal{CK}_h \leq 4C C_{\mathrm{in}}^2 (\mathcal{E}_u + \mathrm{E}_u + \mathcal{E}_h + C_{\mathrm{in}}). 
\] 
Integrating in time gives us for \(t \leq T^\#\), 
\[  \mathcal{E}_u + \mathrm{E}_u + \mathcal{E}_h + \int_{0}^t (\mathcal{D}_u  + \mathcal{CK}_h) \leq 8 C C_{\mathrm{in}}^3 e^{4 C C_{\mathrm{in}}^2 t } . 
\] 
Let us take \(T_0 \) to be defined as 
\[  T_0 := \min \,  \left\{\frac{\delta _0}{64 C C_{\mathrm{in}}} , \frac{1}{4 C C_{\mathrm{in}}^2} \ln  \frac{c_0}{16 C C_{\mathrm{in}}^3}\right\} \implies 8 C C_{\mathrm{in}}^3 e^{4 C C_{\mathrm{in}}^2  T_0} \leq \frac{c_0}{2}. 
\] 
Then we deduce that for all \(t \leq \min \, \{T^\# , T_0\} = T_0\), we have 
\begin{equation}\label{eq:Final Conclusion for LW}
    \mathcal{E}_u + \mathrm{E}_u + \mathcal{E}_h + \int_{0}^t (\mathcal{D}_u  + \mathcal{CK}_h) \leq \frac{c_0}{2}. 
\end{equation} 
Setting \( T := T_0\), the estimate \eqref{eq:Estimate on u for LW} and \eqref{eq:Estimate on h for LW} follows from \eqref{eq:Final Conclusion for LW}. 
\end{proof}

\section{Justification of the Approximation in Gevrey-2 class}  
In this section, we prove the solution to the Navier-Stokes Maxwell \eqref{eq:System for (u^e,b^e)} converges to the solution to the Hydrostatic Navier-Stokes Maxwell \eqref{eq:System for (u^p,b^p)} as \( \varepsilon \to 0_+\) and study the convergence rate. It is natural to study the equation of \( (u^R, v^R, h^R, e^R, f^R, p^R) = ( u^\varepsilon - u^P , v^\varepsilon - v^P , h^\varepsilon - h^P , e^\varepsilon - e^P , f^\varepsilon - f^P , p^\varepsilon - p^P )\) which satisfy 
\begin{equation}\label{eq:Equations for u^R}
    \begin{cases} \partial_t u^R  - \varepsilon ^2 \partial_{xx} u^R  - \partial_{yy} u^R   + \partial_x p^R  + N_u + S_u  = \alpha (f^R  - u^R) + F_u, \\ 
       S_v + \partial_y p^P  = - \alpha (e^R + v^R ) + F_v, \\ 
        \partial_x u^R + \partial_y v^R = \partial_x e^R  +\partial_y f^R = 0, \\ 
        \partial_t h^R + \partial_x f^R - \partial_y e^R = 0, \\
        \beta \partial_{tt} h^R + \partial_t h^R  - \gamma(\varepsilon ^2 \partial_{xx} + \partial_{yy}   ) h^R   + S_h   + N_h= 0 , \\
  {  (u^R, v^R, h^R, \partial_ y e^R, f^R)|_{y = 0,1} = 0 , } \\
{ (u^R, v^R, h^R, \partial_y e^R , f^R)|_{t = 0} = 0 .}  \end{cases}  
\end{equation} 
where 
\begin{align*}
    N_u &:= (u^\varepsilon, v^\varepsilon) \cdot \nabla u^\varepsilon - (u^P, v^P) \cdot \nabla u ^P     =  (u^\varepsilon, v^\varepsilon) \cdot \nabla u^R + (u^R, v^R) \cdot \nabla u^P , \\
    S_u &:= \varepsilon ^2 \partial_{xx}  u^P, \\  
    F_u &:= \alpha(f^\varepsilon h^\varepsilon - 2 u^\varepsilon h^\varepsilon - u^\varepsilon (h^\varepsilon)^2 - f^P h^P   +  2 u^P h^P + u^P (h^P)^2), \\
    S_v &:= \varepsilon ^2 ( \partial_t v^\varepsilon  + u^\varepsilon \partial_x v^\varepsilon  + v^\varepsilon \partial_y v^\varepsilon - \varepsilon ^2 \partial_{xx} v^\varepsilon - \partial_{yy} v^\varepsilon  ), \\
    F_v &:= -\alpha (e^\varepsilon h^\varepsilon + 2 v^\varepsilon h^\varepsilon  +v^\varepsilon  (h^\varepsilon)^2  - e^P h^P - 2 v^P h^P - v^P (h^P)^2 ), \\ 
    S_h &:= \gamma \varepsilon ^2 \partial_{xx} h^P , \\ 
     N_h &:=    (u^\varepsilon, v^\varepsilon) \cdot \nabla h^R + (u^R, v^R) \cdot \nabla h^P.
\end{align*} 
We introduce the following energy functionals: for \(\sigma = s - \frac{3}{4}\) and \( s \geq 10 : \)  
\begin{align}\label{eq:Convergence:Definition of energies}
    \begin{split} 
        \mathcal{E}_u &:= \lVert  u^R  \rVert_{\mathcal{G}^{ s - 4}_2}^2,  \qquad\qquad\quad \mathcal{D}_u:= \lVert  ( \varepsilon \partial_x  , \partial_y  )u^R \rVert_{\mathcal{G}^{s - 4}_2}^2  + \alpha \lVert (u^R, v^R) \rVert_{\mathcal{G}^{s - 4}_2}^2 , \\  \mathrm{E}^{u} &:= \lVert  \partial_y u^R \rVert_{\mathcal{G}^{ s - 6}_{2}}^2 ,\qquad\qquad \mathrm{D}_u  := \lVert  \partial_t u^R \rVert_{\mathcal{G}^{s - 6}_2}^2,  \\
        \mathcal{E}_h &:= \lVert  (\beta \partial_t, \gamma \varepsilon \partial_x, \gamma \partial_y) h^R \rVert_{\mathcal{G}^{ \sigma- 4}_2}^2 + \beta\dot{\lambda} ^2 \lVert  h^R \rVert_{\mathcal{G}^{ \sigma   - 4+ \frac{1}{2 }}_{2}}^2. 
    \end{split}
\end{align} 
As in the proof of Theorem \ref{thm:Local well-posedness}, we take \(\lambda\) to be defined as in the result. Then thanks to Remark \ref{rmk:Estimates extend to system for u^p}, we get that the system \eqref{eq:System for (u^p,h^p)} has a unique solution \((u^P, e^P, h^P)\) on \( [0,T_0]\) with \( T_0\) being determined by the previous proof so that for \( t \leq T_0: \)  
\begin{equation}\label{eq:Smallness on Eu for convergence}
    \lVert  (u^P, u^\varepsilon) \rVert_{\mathcal{G}^{s}_2}^2   + \lVert  \partial_y (u^P, u^\varepsilon) \rVert_{\mathcal{G}^{s - 2}_2}^2  + \int_{0}^t \lVert  \partial_y (u^P, u^\varepsilon)   \rVert_{\mathcal{G}^{s}_2}^2 + \int_{0}^t  \lVert  \partial_t (u^P, u^\varepsilon)  \rVert_{\mathcal{G}^{s - 2}_2}^2 \leq C_{\mathrm{in}},
\end{equation}  
and 
\begin{equation}\label{eq:Smallness on Eh for convergence}
    \lVert  (\partial_t, \partial_y) (h^P, h^\varepsilon) \rVert_{\mathcal{G}^{\sigma}_2}^2  \leq c,  
\end{equation}   
We now establish some estimates pertaining to the nonlinearities appearing in the system. 
\subsection{A Priori Estimates.}

\begin{lemma}[Estimates for the small terms]\label{lem:Estimates for the small terms}
    For any \(t \leq T_0  \) and \(\varepsilon \in [0,1]\), we have  
    \[  \lVert  S_u \rVert_{\mathcal{G}^{s - 4}_2}  \leq \sqrt{C_{\mathrm{in}}} \varepsilon ^2,\quad \lVert  S_h \rVert_{\mathcal{G}^{s - 4}_2}\leq \sqrt{c} \varepsilon ^2 , \quad {\lVert  S_v \rVert_{\mathcal{G}^{ s - 4}_2}  \leq   \varepsilon ^2 \left(C_{\mathrm{in}}+ \lVert  \partial_t u^\varepsilon \rVert_{\mathcal{G}^{ s - 3}_2}\right)} 
    \] 
\end{lemma}
\begin{proof} 
    For \(S_u\) and \(S_h\) the result is a direct consequence of \eqref{eq:Smallness on Eu for convergence} and \eqref{eq:Smallness on Eh for convergence}. Indeed, 
    \begin{align*}
          \lVert S_u \rVert_{\mathcal{G}^{s - 4}_2}  &\leq \varepsilon ^2 \lVert  \partial_{xx} u^P \rVert_{\mathcal{G}^{s - 4}_2}  \leq \varepsilon ^2 \lVert  u^P \rVert_{\mathcal{G}^{ s - 2}_2} \leq \varepsilon ^2 \sqrt{C_{\mathrm{in}}},   \\
          \lVert  S_h \rVert_{\mathcal{G}^{ s - 4}_2} &\leq \varepsilon ^2 \lVert  \partial_{xx} h^P \rVert_{\mathcal{G}^{ s - 4}_2} \leq \varepsilon ^2 \lVert  h^P \rVert_{\mathcal{G}^{ s - 2}_2} \leq  \varepsilon ^2\sqrt{c}.    \end{align*}
        For \(S_v\), we need an extra application of Proposition \ref{prop:Gevrey-Holder}:
          \begin{align*}
          \lVert  S_v \rVert_{\mathcal{G}^{s - 4}_2}  &\leq  \varepsilon ^2 \lVert  \partial_t v^\varepsilon \rVert_{\mathcal{G}^{ s- 4}_2}  + \varepsilon ^2 \lVert  (u^\varepsilon , v^\varepsilon) \cdot \nabla v^\varepsilon \rVert_{\mathcal{G}^{s - 4}_2}  +  \varepsilon ^2 \lVert  (\varepsilon ^2 \partial_{xx}, \partial_{yy}) v^\varepsilon \rVert_{\mathcal{G}^{s - 4}_2} \\
          &\leq \varepsilon ^2 \lVert  \partial_t u^\varepsilon \rVert_{\mathcal{G}^{s - 3}_2}  + \varepsilon ^2 \left( \lVert  u^\varepsilon \rVert_{\mathcal{G}^{s - 4}_\infty} \lVert  v^\varepsilon \rVert_{\mathcal{G}^{ s - 3}_2}  + \lVert  v^\varepsilon \rVert_{\mathcal{G}^{ s - 4}_\infty} \lVert  \partial_y v^\varepsilon \rVert_{\mathcal{G}^{ s - 4}_2}\right)  + \varepsilon ^2 \lVert  (\varepsilon ^2, \partial_y) u^\varepsilon \rVert_{\mathcal{G}^{ s - 1}_2} \\
          &\leq \varepsilon ^2 \lVert  \partial_t u^\varepsilon \rVert_{\mathcal{G}^{ s - 3}_2}  + C_{\mathrm{in}}\varepsilon ^2  + \sqrt{C_{\mathrm{in}}}\varepsilon ^2 . \qedhere
    \end{align*}   
\end{proof}  

For the forcing terms we have 
\begin{lemma}[Estimates for the forcing terms]\label{lem:Estimates for the forcing terms}
    For any \(t \leq T_0,  \) and \( \varepsilon \in [0,1]\), we have:
    \[  \lVert  F_u \rVert_{\mathcal{G}^{ s -4}_2} \lesssim \sqrt{c} \lVert  u^R \rVert_{\mathcal{G}^{s -4 }_2} + \sqrt{C_{\mathrm{in}}} \lVert  h^R \rVert_{\mathcal{G}^{s - 4}_2} + \sqrt{c \mathcal{E}_h}
    \]  
    and 
    \[   \lVert  F_v \rVert_{\mathcal{G}^{ s - 4}_2}\lesssim \sqrt{c} \lVert  v^R \rVert_{\mathcal{G}^{s -4 }_2} + \sqrt{C_{\mathrm{in}}} \lVert  h^R \rVert_{\mathcal{G}^{s - 4}_2} + \sqrt{c \mathcal{E}_h}
    \] 
\end{lemma}
\begin{proof} 
    We will split the forcing terms in terms that are linear in \( (u^R, v^R, h^R)\) and those that are non-linear: 
    \begin{align*}
        F_{u,L} &:= \alpha \left(f^R h^P  + f^P h^R  - 2 u ^R h^P - 2u^P h^R  - u^R (h^P)^2 - 2 u ^P h^R h^P\right) , \\
        F_{u,N} & := \alpha \left( f^R h^R  - 2 u^R h^R - 2 u^R h^R h^P - u^R ( h^R)^2  - u^P (h^R)^2 \right) , \\
        F_{v, L} & : = - \alpha \left(e^R h^P  + e^P h^R  + 2 v ^R h^P + 2v^P h^R  + v^R (h^P)^2 + 2 v ^P h^R h^P\right) , \\
        F_{v,N} & :=-  \alpha \left( e^R h^R  + 2 v^R h^R + 2v^R h^R h^P + v^R ( h^R)^2  + v^P (h^R)^2 \right) . 
    \end{align*} 
    Via the estimates on \( (u^P , v^P , h^P)\) and Proposition \ref{prop:Elliptic Estimates} we obtain for the linear terms: 
    \begin{align*}
        \lVert  F_{u,L} \rVert_{\mathcal{G}^{s - 4}_2} &\lesssim  \lVert    f^R \rVert_{\mathcal{G}^{s - 4}_2} \lVert \partial_y   h^P \rVert_{\mathcal{G}^{s - 4}_2} + \lVert \partial_y  f^P \rVert_{\mathcal{G}^{s - 4}_2} \lVert   h^R \rVert_{\mathcal{G}^{s - 4}_2} + \lVert   u^R \rVert_{\mathcal{G}^{s - 4}_2}  \lVert \partial_y  h^P \rVert_{\mathcal{G}^{s - 4}_2}  \\
        &\quad+ \lVert u^R \rVert_{\mathcal{G}^{s - 4}_2} \lVert  \partial_y h^P \rVert_{\mathcal{G}^{s- 4}_2}^2 + \lVert \partial_y u^P \rVert_{\mathcal{G}^{s - 4}_2} \lVert   h^P \rVert_{\mathcal{G}^{s - 4}_2} \lVert h^R \rVert_{\mathcal{G}^{s - 4}_2} \\
        &\lesssim \sqrt{c} \lVert  \partial _t h^R \rVert_{\mathcal{G}^{\sigma - 4}_2} + \sqrt{c} \lVert   h^R \rVert_{\mathcal{G}^{s - 4}_2} + \sqrt{c} \lVert  u^R \rVert_{\mathcal{G}^{s - 4}_2} + c \lVert  u^R \rVert_{\mathcal{G}^{s - 4}_2} + \sqrt{C_{\mathrm{in}}c} \lVert  h^R \rVert_{\mathcal{G}^{s - 4}_2}, \\
        &\lesssim  \sqrt{c} \lVert  u^R \rVert_{\mathcal{G}^{s - 4}_2} + \sqrt{C_{\mathrm{in}} c} \lVert  h^R \rVert_{\mathcal{G}^{s -4}_2} + \sqrt{c \mathcal{E}_h}.  
    \end{align*} 
    and similarly,  
   \begin{align*}
       \lVert  F_{v,L} \rVert_{\mathcal{G}^{s - 4}_2} &\lesssim  \lVert  e^R \rVert_{\mathcal{G}^{s - 4}_2} \lVert \partial_y   h^P \rVert_{\mathcal{G}^{s - 4}_2} + \lVert \partial_y  e^P \rVert_{\mathcal{G}^{s - 4}_2} \lVert h^R \rVert_{\mathcal{G}^{s - 4}_2} + \lVert  v^R \rVert_{\mathcal{G}^{s - 4}_2} \lVert  \partial_y h^P \rVert_{\mathcal{G}^{s - 4}_2}  \\
       & \quad+ \lVert  \partial_y v^P \rVert_{\mathcal{G}^{s - 4}_2} \lVert  h^R \rVert_{\mathcal{G}^{s - 4}_2} + \lVert  v^R \rVert_{\mathcal{G}^{s - 4}_2} \lVert  \partial_y h^P \rVert_{\mathcal{G}^{s - 4}_2}^2  + \lVert u^P \rVert_{\mathcal{G}^{s - 3}_2} \lVert  h^P \rVert_{\mathcal{G}^{s - 4}_2} \lVert  h^R \rVert_{\mathcal{G}^{s - 4}_2}, \\
       &\lesssim  \sqrt{c} \lVert v^R\rVert_{\mathcal{G}^{s - 4}_2} + \sqrt{C_{\mathrm{in}} c} \lVert  h^R \rVert_{\mathcal{G}^{s -4}_2} + \sqrt{c \mathcal{E}_h}.  
   \end{align*}  
    For the non-linear terms, the idea is to split the extra \( \lvert \Psi^R \rvert\) term into \( \lvert \Psi^\varepsilon \rvert +  \lvert \Psi \rvert\) for \( \Psi \in \{u,v,f,h,e\}\) which is atmost \(\sqrt{C_{\mathrm{in}}}\) in size as per \eqref{eq:Smallness on Eu for convergence} and \eqref{eq:Smallness on Eh for convergence}. Indeed, using Proposition \ref{prop:Gevrey-Holder}, we get 
    \begin{align*}
        \lVert  F_{u,N} \rVert_{\mathcal{G}^{ s- 4}_2}& \lesssim  \alpha \lVert  h^R \rVert_{\mathcal{G}^{ s - 4}_2} \left( \lVert  f^R \rVert_{\mathcal{G}^{2}_\infty} + \lVert  u^R \rVert_{\mathcal{G}^{2}_\infty} + \lVert  u^R \rVert_{\mathcal{G}^{2}_\infty} \lVert  h^R \rVert_{\mathcal{G}^{2}_\infty} + \lVert  h^R \rVert_{\mathcal{G}^{2}_\infty}\right) \\
        &\qquad\qquad + \alpha \lVert  h^R \rVert_{\mathcal{G}^{2}_\infty}\left( \lVert  f^R \rVert_{\mathcal{G}^{s-4}_2} + \lVert  u^R \rVert_{\mathcal{G}^{s-4}_2} + \lVert u^R \rVert_{\mathcal{G}^{s-4}_2} \lVert  h^R \rVert_{\mathcal{G}^{2}_\infty} + \lVert  h^R \rVert_{\mathcal{G}^{s-4}_2} \right) \\
        &\lesssim \alpha \lVert  h^R \rVert_{\mathcal{G}^{s - 4}_2} \left( \lVert  \partial_t h^R \rVert_{\mathcal{G}^{2}_2} + \lVert  \partial_y u^R \rVert_{\mathcal{G}^{2}_2} + \lVert  \partial_y u^R \rVert_{\mathcal{G}^{2}_2} \lVert  \partial_y h^R \rVert_{\mathcal{G}^{2}_2} + \lVert  \partial_y h^R \rVert_{\mathcal{G}^{2}_2}\right)\\
        &\qquad\qquad + \alpha \lVert  \partial_y h^R \rVert_{\mathcal{G}^{2}_2} \left( \lVert  \partial_t h^R \rVert_{\mathcal{G}^{ s - 5}_2} + \lVert  u^R \rVert_{\mathcal{G}^{ s- 4}_2} + \lVert  u^R \rVert_{\mathcal{G}^{ s - 4}_2} \lVert  \partial_y h^R \rVert_{\mathcal{G}^{2}_2} + \lVert  h^R \rVert_{\mathcal{G}^{ s - 4}_2} \right)  \\
        &\lesssim \sqrt{C_{\mathrm{in}}} \lVert  h^R \rVert_{\mathcal{G}^{s - 4}_2} + \sqrt{c} \lVert  \partial_t h^R \rVert_{\mathcal{G}^{\sigma - 4}_2} + \sqrt{c}\lVert  u^R \rVert_{\mathcal{G}^{s - 4}_2} \\ 
        &\lesssim  \sqrt{c} \lVert  u^R \rVert_{\mathcal{G}^{s - 4}_2} + \sqrt{C_{\mathrm{in}}} \lVert  h^R \rVert_{\mathcal{G}^{s - 4}_2} + \sqrt{c \mathcal{E}_h}. \\
        \shortintertext{Similarly, the same strategy gives the estimate for \(  \lvert F_{v,N} \rvert:\)}
        \lVert  F_{v, N} \rVert_{\mathcal{G}^{s-4}_2} &\lesssim  \alpha \lVert  h^R \rVert_{\mathcal{G}^{ s - 4}_2} \left( \lVert  e^R \rVert_{\mathcal{G}^{2}_\infty} + \lVert  v^R \rVert_{\mathcal{G}^{2}_\infty} + \lVert  v^R \rVert_{\mathcal{G}^{2}_\infty} \lVert  h^R \rVert_{\mathcal{G}^{2}_\infty} + \lVert  h^R \rVert_{\mathcal{G}^{2}_\infty}\right) \\ 
        &\qquad\qquad + \alpha \lVert  h^R \rVert_{\mathcal{G}^{2}_\infty}\left( \lVert  e^R \rVert_{\mathcal{G}^{s-4}_2} + \lVert  v^R \rVert_{\mathcal{G}^{s-4}_2} + \lVert v^R \rVert_{\mathcal{G}^{s-4}_2} \lVert  h^R \rVert_{\mathcal{G}^{2}_\infty} + \lVert  h^R \rVert_{\mathcal{G}^{s-4}_2} \right) \\
        &\lesssim \alpha \lVert  h^R \rVert_{\mathcal{G}^{s - 4}_2} \left( \lVert  \partial_t h^R \rVert_{\mathcal{G}^{2}_2} + \lVert  u^R \rVert_{\mathcal{G}^{3}_2} + \lVert    u^R \rVert_{\mathcal{G}^{3}_2} \lVert  \partial_y h^R \rVert_{\mathcal{G}^{2}_2} + \lVert  \partial_y h^R \rVert_{\mathcal{G}^{2}_2}\right)\\
        &\qquad\qquad + \alpha \lVert  \partial_y h^R \rVert_{\mathcal{G}^{2}_2} \left( \lVert  \partial_t h^R \rVert_{\mathcal{G}^{ s - 5}_2} + \lVert  v^R \rVert_{\mathcal{G}^{ s- 4}_2} + \lVert  v^R \rVert_{\mathcal{G}^{ s - 4}_2} \lVert  \partial_y h^R \rVert_{\mathcal{G}^{2}_2} + \lVert  h^R \rVert_{\mathcal{G}^{ s - 4}_2} \right)  \\
        &\lesssim \sqrt{C_{\mathrm{in}}} \lVert  h^R \rVert_{\mathcal{G}^{s - 4}_2} + \sqrt{c} \lVert  \partial_t h^R \rVert_{\mathcal{G}^{\sigma - 4}_2} + \sqrt{c}\lVert  v^R \rVert_{\mathcal{G}^{s - 4}_2} \\ 
        &\lesssim  \sqrt{c} \lVert  v^R \rVert_{\mathcal{G}^{s - 4}_2} + \sqrt{C_{\mathrm{in}}} \lVert  h^R \rVert_{\mathcal{G}^{s - 4}_2} + \sqrt{c \mathcal{E}_h}. \qedhere
    \end{align*}  
\end{proof}  
\subsection{The convergence process}
Hitting the first two equations of the system \eqref{eq:Equations for u^R} by \( (\mathrm{A}_{s - 4}^2 u^R, \mathrm{A}_{s - 4}^2 v^R)\) and summing, we obtain 
\begin{align}\label{eq:EnergyEquation for u^R}
    \begin{split} 
    &\frac{1}{2}\frac{\mathrm{d}}{\mathrm{d}t} \lVert  u^R \rVert_{\mathcal{G}^{ s - 4}_2} ^2 + \dot{\lambda} \lVert  u^R \rVert_{\mathcal{G}^{s - 4 + \frac{1}{4}}_2}^2   + \lVert  (\varepsilon \partial_x, \partial_y) u^R \rVert_{\mathcal{G}^{ s - 4}_2}^2 + \alpha \lVert (u^R, v^R) \rVert_{\mathcal{G}^{ s - 4}_2}^2\\
    &\qquad\qquad = \langle  \mathrm{A}_{s - 4} (F_u - N_u - S_u ), \mathrm{A}_{s - 4}  u^R  \rangle + \langle  \mathrm{A}_{s - 4} ( F_v - S_v), \mathrm{A}_{s - 4}  v^R \rangle, 
    \end{split}
\end{align} 
 
\noindent where we have used the divergence-free condition of \((u^R, v^R)\) to get rid of \( (p^R, f^R, e^R)\). Also hitting the equation for \(h^R\) by \( \mathrm{A}_{\sigma - 4}^2 \partial_t h^R\), we get  
\begin{equation}\label{eq:EnergyEquation for h^R-Int1}
    \frac{1}{2} \frac{\mathrm{d}}{\mathrm{d}t} \lVert  (\beta \partial_t, \gamma \varepsilon \partial_x, \gamma \partial_y) h^R \rVert_{\mathcal{G}^{ \sigma- 4}_2}^2 +\beta  \dot{\lambda} \lVert \partial_t  h^R \rVert_{\mathcal{G}^{ \sigma - 4 + \frac{1}{4}}_2}^2  + \lVert  \partial_t h^R \rVert_{\mathcal{G}^{\sigma - 4}_2}^2   = - \langle  \mathrm{A}_{ \sigma- 4}(S_h + N_h) , \mathrm{A}_{ \sigma- 4} \partial_t h^R \rangle
\end{equation} 
where we further note that, per Lemma \ref{lem:|∂ₜh|ₛ expansion}, 
\[ \dot{\lambda} \lVert  \partial_t h \rVert_{\mathcal{G}^{\sigma - 4 + \frac{1}{4}}_2 }^2  = \dot{\lambda} \lVert  \partial_t [ \mathrm{A}_{\sigma - 4 + \frac{1}{4}} h] \rVert_{L^2}^2 + \dot{\lambda}^3 \lVert  h \rVert_{\mathcal{G}^{\sigma -4  + \frac{3}{4}}_2}^2  + \frac{\mathrm{d}}{\mathrm{d}t} \left(\dot{\lambda}^2 \lVert  h \rVert_{\mathcal{G}^{\sigma - 4 + \frac{1}{2}}_2}^2\right) - 2 \ddot{\lambda} \dot{\lambda} \lVert  h \rVert_{\mathcal{G}^{\sigma - 4  + \frac{1}{2}}_2}^2, 
\] 
to write \eqref{eq:EnergyEquation for h^R-Int1} as  
\[  \frac{\mathrm{d}}{\mathrm{d}t} \left( \lVert (\sqrt{\beta} \partial_t , \sqrt{\gamma} \varepsilon \partial_x, \sqrt{\gamma}\partial_y )h^R \rVert_{\mathcal{G}^{\sigma - 4}_2}^2 +\beta \dot{\lambda}^2 \lVert  h ^R\rVert_{\mathcal{G}^{\sigma - 4 + \frac{1}{2}}_2}^2\right) +  \dot{\lambda} \lVert  (\sqrt{\beta} \partial_t , \sqrt{2\gamma} \varepsilon \partial_x,\sqrt{2\gamma} \partial_y )h^R \rVert_{\mathcal{G}^{\sigma - 4 + \frac{1}{4}}_2}^2 
\] 
\begin{equation}\label{eq:EnergyEquation for h^R}
    \quad\qquad\qquad+ \beta \dot{\lambda} \lVert  \partial_t [ \mathrm{A}_{\sigma - 4 + \frac{1}{4}} h] \rVert_{L^2}^2 + \beta \dot{\lambda} ^3 \lVert  h \rVert_{\mathcal{G}^{s - 4 }_2}^2  + 2\lVert  \partial_t h^R \rVert_{\mathcal{G}^{\sigma - 4}_2}^2 \qquad\qquad\qquad\qquad
\end{equation} 
\[ \qquad\qquad = 2 \ddot{\lambda} \dot{\lambda} \lVert  h \rVert_{\mathcal{G}^{\sigma - 4 + \frac{1}{2}}_2}^2  - 2\left\langle  \mathrm{A}_{\sigma - 4}( S_h + N_h) , \mathrm{A}_{\sigma - 4} (\partial_t h^R) \right\rangle. 
\]    
{Hitting the equation for \( (u,v)\) by \( (\mathrm{A}_{ s - 6} ^2 (\partial_t u^R), \mathrm{A}_{ s - 6} ^2 (\partial_t v^R))\) and summing, we get }   
\begin{align}\label{eq:EnergyEquation for ∂ₜu^R}
    \begin{split} 
        \lVert  \partial_t u^R \rVert_{\mathcal{G}^{ s - 6}_2}^2 &+ \frac{1}{2} \frac{\mathrm{d}}{\mathrm{d}t} \left(\lVert  (\alpha, \varepsilon \partial_x, \partial_y ) u^R \rVert_{\mathcal{G}^{ s - 6}_2}^2 + \lVert \alpha v^R \rVert_{\mathcal{G}^{s - 6}_2}^2 \right) + \dot{\lambda} \lVert  (\alpha ,\varepsilon \partial_x, \partial_y) u^R \rVert_{\mathcal{G}^{\sigma - 5}_2}  + \dot{\lambda} \lVert  \alpha v^R \rVert_{\mathcal{G}^{\sigma - 5}_2}^2  \\
        & = \langle  \mathrm{A}_{ s- 6} ( F_u- N_u - S_u) , \mathrm{A}_{ s- 6} [ \partial_t u^R] \rangle + \langle \mathrm{A}_{s - 6} (F_v - S_v), \mathrm{A}_{s - 6}(\partial_t v^R )\rangle. 
    \end{split}
\end{align} 
 
The quantities that we need to control are the nonlinear \(N_u\) and \(N_h\) terms. To that end, using Proposition \ref{prop:Gevrey-Holder}, Proposition \ref{prop:Estimate on〈Aₛ(g∇ϕ),Aₛϕ〉}, we get  
\begin{align}\label{eq:Estimate for Nu,u}
    \begin{split} 
        &\lvert  \langle \mathrm{A}_{s - 4} N_u , \mathrm{A}_{s - 4} u^R \rangle \rvert   \leq   \lvert  \langle  \mathrm{A}_{s - 4}(u^\varepsilon \partial_x u^R) , \mathrm{A}_{s - 4} u^R \rangle \rvert +  \lvert  \langle \mathrm{A}_{s - 4} (v^\varepsilon \partial_y u^R) , \mathrm{A}_{s - 4} u^R \rangle \rvert  \\
        &\qquad \qquad\qquad\qquad\qquad\qquad +  \lvert  \langle  \mathrm{A}_{s - 4} (u^R \partial_x u^P) , \mathrm{A}_{s - 4} u^R \rangle \rvert +  \lvert  \langle  \mathrm{A}_{s- 4} ( v^R \partial_y u^P) , \mathrm{A}_{s - 4} u^R \rangle \rvert \\ 
     &\lesssim  \lVert  \partial_y u^\varepsilon \rVert_{\mathcal{G}^{2}_2} \lVert   u^R \rVert_{\mathcal{G}^{s -4 + \frac{1}{4}}_2}^2 + \lVert  u^\varepsilon \rVert_{\mathcal{G}^{s - 4}_2} \lVert  u^R \rVert_{\mathcal{G}^{s}_2} \lVert  \partial_y u^R \rVert_{\mathcal{G}^{2}_2} \\
     &\qquad + \lVert  u^\varepsilon \rVert_{\mathcal{G}^{3}_2} \lVert  \partial_y u^R \rVert_{\mathcal{G}^{s - 4}_2} \lVert  u^R \rVert_{\mathcal{G}^{s - 4}_2} + \lVert  u^\varepsilon \rVert_{\mathcal{G}^{s}_2} \lVert  \partial_y u^R \rVert_{\mathcal{G}^{2}_2} \lVert  \partial_y u^R \rVert_{\mathcal{G}^{s - 4}_2} + \lVert  \partial_y u^\varepsilon \rVert_{\mathcal{G}^{s - 2}_2} \lVert  u^R \rVert_{\mathcal{G}^{s - 4}_2}^2 \\
      & \,\qquad + \lVert  u^R \rVert_{\mathcal{G}^{s - 4}_2}^2 \lVert \partial_y u^P \rVert_{\mathcal{G}^{2}_2} + \lVert  u^R \rVert_{\mathcal{G}^{s - 4}_2} \lVert  u^P \rVert_{\mathcal{G}^{s - 4}_2} \lVert  \partial_y u^R \rVert_{\mathcal{G}^{2}_2} \\
      &\, \qquad + \lVert  u^R \rVert_{\mathcal{G}^{s - 4}_2}^{\frac{1}{2}} \lVert  \partial_y u^R \rVert_{\mathcal{G}^{s - 4}_2}^{\frac{1}{2}} \lVert  v^R \rVert_{\mathcal{G}^{s - 4}_2} \lVert  \partial_y u^P \rVert_{\mathcal{G}^{2}_2} + \lVert  u^R \rVert_{\mathcal{G}^{s - 4}_2} \lVert  u^R \rVert_{\mathcal{G}^{3}_2} \lVert  \partial_y u^P \rVert_{\mathcal{G}^{s - 4}_2}.  \\
     &\lesssim  \sqrt{C_{\mathrm{in}}}\left( \lVert   u^R\rVert_{\mathcal{G}^{s -4+ \frac{1}{4}}_2}^2  +    \lVert  u^R \rVert_{\mathcal{G}^{s-4}_2} \lVert  \partial_y u^R \rVert_{\mathcal{G}^{s-4}_2}  \right) \\
     & \qquad + \sqrt{C_{\mathrm{in}}}\left( \lVert  \partial_y u^R \rVert_{\mathcal{G}^{s - 4}_2} \lVert  \partial_y u^R \rVert_{\mathcal{G}^{2}_2} + \lVert  u^R \rVert_{\mathcal{G}^{s - 4}_2}^{\frac{1}{2}} \lVert  \partial_y u^R \rVert_{\mathcal{G}^{s - 4}_2}^{\frac{1}{2}}  \lVert  v^R \rVert_{\mathcal{G}^{s - 4}_2}   \right) \\
     &\lesssim  \sqrt{C_{\mathrm{in}}} \lVert  u^R \rVert_{\mathcal{G}^{s - 4 + \frac{1}{4}}_2}^2  + \sqrt{C_{\mathrm{in}}}\lVert  \partial_y u^R \rVert_{\mathcal{G}^{s - 4}_2} \left(\sqrt{  \mathcal{E}_u} + \sqrt{ \mathrm{E}_u}   \right) +\sqrt{C_{\mathrm{in}}} \lVert  u^R \rVert_{\mathcal{G}^{s - 4}_2}^{\frac{1}{2}} \lVert  \partial_y u^R \rVert_{\mathcal{G}^{s - 4}_2}^{\frac{1}{2}}  \lVert  v^R \rVert_{\mathcal{G}^{s - 4}_2}. 
    \end{split}
\end{align} 
Here note that we couldn't get rid of the nonlinear behavior using \( \lvert \Psi^R \rvert \leq  \lvert\Psi^\varepsilon \rvert +  \lvert \Psi^L \rvert \lesssim \sqrt{C_{\mathrm{in}}}\) like we did in Lemma \ref{lem:Estimates for the forcing terms}. This is because, unlike for the forcing terms \( \lvert F_{u,N} \rvert\) and \( F_{v,N}\), the velocities do not come coupled with terms we expect to be small like \(h^R\). Here \(u^R\) is coupled with terms like \( \partial_x u^R\) so such terms are neither directly controllable by energy nor are they small. Regardless, precise estimates akin to those in the proof of Theorem \ref{thm:Local well-posedness} suffices here as shown above. 
\begin{align}\label{eq:Estimate for Nh,h}
    \begin{split}  
      &\lvert  \langle  \mathrm{A}_{s - 4} N_h , \mathrm{A}_{s - 4} \partial_t h^R \rangle \rvert \\
      &\leq  \lvert  \langle \mathrm{A}_{\sigma - 4} (u^\varepsilon \partial_x h^R) , \mathrm{A}_{\sigma - 4} \partial_t h^R \rangle \rvert + \lvert  \langle \mathrm{A}_{\sigma - 4} (v^\varepsilon \partial_y h^R) , \mathrm{A}_{\sigma - 4} \partial_t h^R \rangle \rvert  \\
      &\qquad +  \lvert  \langle  \mathrm{A}_{\sigma - 4}( u^R \partial_x h^P) , \mathrm{A}_{\sigma - 4} \partial_t h^R  \rangle \rvert +  \lvert  \langle  \mathrm{A}_{\sigma - 4}( v^R \partial_y h^P ), \mathrm{A}_{\sigma - 4} \partial_t h^R \rangle \rvert \\
      &\lesssim  \lVert  \partial_t h^R \rVert_{\mathcal{G}^{\sigma  - 4 + \frac{1}{4}}_2} \lVert  \partial_y u^\varepsilon \rVert_{\mathcal{G}^{s - 2}_2} \lVert  h^R  \rVert_{\mathcal{G}^{s - 4}_2}   + \lVert  \partial_t h^R \rVert_{\mathcal{G}^{\sigma - 4}_2} \lVert  u^\varepsilon \rVert_{\mathcal{G}^{s -2}_2} \lVert  \partial_y h^R \rVert_{\mathcal{G}^{\sigma - 4}_2}  \\
      & \qquad\qquad + \lVert  \partial_t h^R \rVert_{\mathcal{G}^{\sigma - 4}_2} \lVert  u^R \rVert_{\mathcal{G}^{\sigma - 4}_2} \lVert  \partial_y h^P \rVert_{\mathcal{G}^{\sigma - 2}_2} + \lVert  \partial_th^R \rVert_{\mathcal{G}^{\sigma - 4}_2}  \lVert  u^R \rVert_{\mathcal{G}^{s - 4 + \frac{1}{4}}_2} \lVert  \partial_y h^P \rVert_{\mathcal{G}^{\sigma - 4}_2}\\
      &\lesssim \sqrt{C_{\mathrm{in}}} \lVert  \partial_t h^R \rVert_{\mathcal{G}^{\sigma - 4 + \frac{1}{4}}_2}  \lVert  h^R \rVert_{\mathcal{G}^{s - 4}_2} + \sqrt{C_{\mathrm{in}}\mathcal{E}_h} \left( \lVert  \partial_y h^R \rVert_{\mathcal{G}^{\sigma - 4}_2} + \sqrt{\mathcal{E}_u}  + \lVert  u^R \rVert_{\mathcal{G}^{s - 4 + \frac{1}{4}}_2 } \right) 
    \end{split}
\end{align}  
\begin{equation}\label{eq:Estimate for Nu,par_t u^R}
    \lvert \langle  \mathrm{A}_{s - 6} N_u , \mathrm{A}_{s - 6} \partial_t u^R \rangle \rvert\leq \lVert  \partial_t u^R \rVert_{\mathcal{G}^{s - 6}_2}  \sqrt{C_{\mathrm{in}}} \lVert  u^R \rVert_{\mathcal{G}^{s - 5}_2}  \leq \lVert  \partial_t u^R \rVert_{\mathcal{G}^{s - 6}_2} \sqrt{C_{\mathrm{in}} \mathcal{E}_u}.      \qquad\qquad\qquad\quad   \, 
\end{equation} 
Then using Lemma \ref{lem:Estimates for the small terms}, Lemma \ref{lem:Estimates for the forcing terms} and \eqref{eq:Estimate for Nu,u}, we can control the left-hand side of \eqref{eq:EnergyEquation for u^R} as 
\begin{align*}
    \frac{1}{2}\frac{\mathrm{d}}{\mathrm{d}t} &\lVert  u^R \rVert_{\mathcal{G}^{s - 4}_2}^2 + \dot{\lambda} \lVert  u^R \rVert_{\mathcal{G}^{s - 4 + \frac{1}{4}}_2}^2 +  \lVert  (\varepsilon \partial_x, \partial_y ) u^R \rVert_{\mathcal{G}^{s - 4}_2}^2  + \alpha\lVert  (u^R, v^R) \rVert_{\mathcal{G}^{s- 4}_2}^2  \\
    & \lesssim  \lVert  u^R \rVert_{\mathcal{G}^{s - 4}_2}  \left( \sqrt{c}\lVert u^R \rVert_{\mathcal{G}^{s - 4}_2} +\sqrt{c \mathcal{E}_h} + \sqrt{C_{\mathrm{in}}} \lVert  h^R \rVert_{\mathcal{G}^{s - 4}_2}  \right) + \sqrt{C_{\mathrm{in}}} \lVert  u^R \rVert_{\mathcal{G}^{s - 4 + \frac{1}{4}}_2}^2  \\
    & \qquad + \sqrt{C_{\mathrm{in}}}\lVert  \partial_y u^R \rVert_{\mathcal{G}^{s - 4}_2} \left(\sqrt{  \mathcal{E}_u} + \sqrt{ \mathrm{E}_u}   \right) + \sqrt{C_{\mathrm{in}}} \lVert  u^R \rVert_{\mathcal{G}^{s - 4}_2}^{\frac{1}{2}} \lVert  \partial_y u^R \rVert_{\mathcal{G}^{s - 4}_2}^{\frac{1}{2}}  \lVert  v^R \rVert_{\mathcal{G}^{s - 4}_2}\\
    & \qquad + \lVert v^R \rVert_{\mathcal{G}^{s - 4}_2}  \left( \sqrt{c}\lVert  v^R\rVert_{\mathcal{G}^{s - 4}_2} + \sqrt{C_{\mathrm{in}}} \lVert  h^R \rVert_{\mathcal{G}^{s - 4}_2} + \sqrt{c\mathcal{E}_h} + \varepsilon ^2 \sqrt{C_{\mathrm{in}}} + \varepsilon ^2 \lVert  \partial_t u^\varepsilon \rVert_{\mathcal{G}^{s - 3}_2}\right).
\end{align*} 
Then using the smallness of \(\varepsilon ^2\) and \(\sqrt{c}\) and \(\dot{\lambda} \geq 32 C C_{\mathrm{in}}\) with Young's inequality, we get 
\begin{align}\label{eq:EnergyInequality for u^R}
    \begin{split} 
        \frac{\mathrm{d}}{\mathrm{d}t} \mathcal{E}_u +\lVert  \partial_y u^R \rVert_{\mathcal{G}^{s - 4}_2}^2 +  &\dot{\lambda} \lVert  u^R \rVert_{\mathcal{G}^{s - 4 + \frac{1}{4}}_2}^2  + \mathcal{D}_u  - \frac{1}{32} \lVert  \partial_y u^R \rVert_{\mathcal{G}^{s - 4}_2}^2  \\
        & \lesssim  C_{\mathrm{in}} \left(\mathcal{E}_h  + \mathcal{E}_u + \mathrm{E}_u  +   \lVert  h^R \rVert_{\mathcal{G}^{s - 4}_2}^2 \right) + \varepsilon ^4 \lVert  \partial_t u^\varepsilon \rVert_{\mathcal{G}^{s - 3}_2}^2. 
    \end{split}
\end{align}  
The equation \eqref{eq:EnergyEquation for h^R} for \(h^R\) can be controlled similarly using Lemma \ref{lem:Estimates for the small terms}, Lemma \ref{lem:Estimates for the forcing terms} and \eqref{eq:Estimate for Nh,h}
\begin{align*}
    &\frac{\mathrm{d}}{\mathrm{d}t} \left( \lVert ( \sqrt{\beta}\partial_t , \sqrt{\gamma} \varepsilon \partial_x,\sqrt{\gamma} \partial_y )h^R \rVert_{\mathcal{G}^{\sigma - 4}_2}^2 +\beta \dot{\lambda}^2 \lVert  h \rVert_{\mathcal{G}^{\sigma - 4 + \frac{1}{2}}_2}^2\right) +  \dot{\lambda} \lVert  (\sqrt{\beta} \partial_t ,\sqrt{2 \gamma} \varepsilon \partial_x,\sqrt{2 \gamma} \partial_y )h^R \rVert_{\mathcal{G}^{\sigma - 4 + \frac{1}{4}}_2}^2  \\
     &\quad\qquad\qquad+ \beta \dot{\lambda} \lVert  \partial_t [ \mathrm{A}_{\sigma - 4 + \frac{1}{4}} h^R] \rVert_{L^2}^2 + \beta \dot{\lambda} ^3 \lVert  h^R \rVert_{\mathcal{G}^{s - 4 }_2}^2  + 2\lVert  \partial_t h^R \rVert_{\mathcal{G}^{\sigma - 4}_2}^2  \\
     & \qquad  \leq 0  - \left\langle  \mathrm{A}_{\sigma - 4}( S_h + N_h) , \mathrm{A}_{\sigma - 4} (\partial_t h^R) \right\rangle \\
     &\qquad   \lesssim \sqrt{C_{\mathrm{in}}} \lVert  \partial_t h^R \rVert_{\mathcal{G}^{\sigma - 4 + \frac{1}{4}}_2}  \lVert  h^R \rVert_{\mathcal{G}^{s - 4}_2} + \sqrt{C_{\mathrm{in}}\mathcal{E}_h} \left( \lVert  \partial_y h^R \rVert_{\mathcal{G}^{\sigma - 4}_2} + \sqrt{\mathcal{E}_u}  + \lVert  u^R \rVert_{\mathcal{G}^{s - 4 + \frac{1}{4}}_2 } + \sqrt{c}\varepsilon ^2 \right)
\end{align*} 
where we have used the property \(\ddot{\lambda} \leq 0 \leq \dot{\lambda}\) to bound \(2 \ddot{\lambda} \dot{\lambda} \lVert h \rVert_{\mathcal{G}^{\sigma - 4 + \frac{1}{2}}_2}^2 \leq 0\). Thus using \(\dot{\lambda} \geq 32 C C_{\mathrm{in}}:\) 
\begin{align}\label{eq:EnergyInequality for h^R}
    \begin{split} 
         \frac{\mathrm{d}}{\mathrm{d}t} \mathcal{E}_h& + \dot{\lambda} \lVert  (  \partial_t ,  \partial_y) h^R \rVert_{\mathcal{G}^{\sigma - 4  
        + \frac{1}{4}}_2}^2  +  \dot{\lambda} ^3 \lVert  h^R \rVert_{\mathcal{G}^{s - 4 }_2}^2  - \frac{1}{2} \lVert  u^R \rVert_{\mathcal{G}^{s - 4 + \frac{1}{4}}_2}^2  \lesssim  C_{\mathrm{in}}(\mathcal{E}_h + \mathcal{E}_u ) + \varepsilon ^2 \sqrt{C_{\mathrm{in}} \mathcal{E}_h}. 
    \end{split}
\end{align} 
Finally for \(\partial_t u^R\), we use Lemma \ref{lem:Estimates for the small terms}, Lemma \ref{lem:Estimates for the forcing terms} and \eqref{eq:Estimate for Nu,par_t u^R}, we obtain 
\begin{align*}
    \lVert  \partial_t u^R &\rVert_{\mathcal{G}^{s - 6}_2}^2 + \frac{1}{2} \frac{\mathrm{d}}{\mathrm{d}t} \lVert  (\alpha u^R, \partial_y u^R, \alpha v^R )\rVert_{\mathcal{G}^{s - 6}_2}^2  + \dot{\lambda} \lVert (\alpha u^R, \partial_y u^R, \alpha v^R ) \rVert_{\mathcal{G}^{\sigma - 5}_2}^2\\
    &\lesssim  \lVert  \partial_t u^R \rVert_{\mathcal{G}^{s - 6}_2} \left(  \sqrt{C_{\mathrm{in}}} \lVert (u^R, h^R)\rVert_{\mathcal{G}^{s - 6}_2} + \sqrt{c \mathcal{E}_h} + \sqrt{C_{\mathrm{in}}}\varepsilon ^2\right), \\
    & \qquad + \lVert  \partial_t v^R \rVert_{\mathcal{G}^{s - 7}_2} \left( \sqrt{c} \lVert  v^R \rVert_{\mathcal{G}^{s - 5}_2} + \sqrt{C_{\mathrm{in}}}\lVert  h^R \rVert_{\mathcal{G}^{s - 5}_2} + \sqrt{c \mathcal{E}_h} + C_{\mathrm{in}} \varepsilon ^2 + \varepsilon ^2 \lVert  \partial_t u^\varepsilon \rVert_{\mathcal{G}^{s - 3}_2}\right) \\
    &\lesssim  \lVert  \partial_t u^R \rVert_{\mathcal{G}^{s - 6}_2} \left( \sqrt{c} \lVert u^R \rVert_{\mathcal{G}^{s - 4}_2} + \sqrt{C_{\mathrm{in}}} \lVert  (u^R, h^R) \rVert_{\mathcal{G}^{s - 4}_2} + \sqrt{c \mathcal{E}_h} + C_{\mathrm{in}} \varepsilon ^2 + \varepsilon ^2 \lVert  \partial_t u^\varepsilon \rVert_{\mathcal{G}^{s - 3}_2}\right). 
\end{align*} 
Using Young's inequality and smallness of \(c\) and \(\varepsilon ^2\) to absorb terms on the left-side, we obtain 
\begin{equation}\label{eq:EnergyInequality for par_tu^R}
    \mathrm{D}_u + \frac{\mathrm{d}}{\mathrm{d}t} \lVert  (\alpha , \varepsilon \partial_x, \partial_y)  u^R \rVert_{\mathcal{G}^{s - 6}_2}^2 \lesssim  C_{\mathrm{in}}  \lVert  h^R \rVert_{\mathcal{G}^{s - 4}_2}^2 +C_{\mathrm{in}} \mathcal{E}_u +  C_{\mathrm{in}} \mathcal{E}_h  + \varepsilon ^4 \lVert  \partial_t u^\varepsilon \rVert_{\mathcal{G}^{s - 3}_2}^2. 
\end{equation} 
Summing up \eqref{eq:EnergyInequality for u^R}, \eqref{eq:EnergyInequality for h^R} and \eqref{eq:EnergyInequality for par_tu^R} and dropping positive terms, we obtain  
\begin{align}\label{eq:EnergyInequality for Convergence}
    \begin{split} 
      &  \frac{\mathrm{d}}{\mathrm{d}t}\left(\mathcal{E}_u + \mathcal{E}_h +\mathrm{E}_u +  \lVert  (\alpha u^R, \alpha v^R, \partial_y u^R) \rVert_{\mathcal{G}^{s - 6}_2}^2\right) +   \lVert  \partial_y u^R \rVert_{\mathcal{G}^{s - 4}_2}^2 +  \dot{\lambda} \lVert  u^R \rVert_{\mathcal{G}^{s - 4 + \frac{1}{4}}_2}^2  + \mathcal{D}_u  \\
      &\qquad + \dot{\lambda} \lVert  (  \partial_t ,  \partial_y) h^R \rVert_{\mathcal{G}^{\sigma - 4  + \frac{1}{4}}_2}^2  +  \dot{\lambda} ^3 \lVert  h^R \rVert_{\mathcal{G}^{s - 4 }_2}^2  - \frac{1}{2} \lVert  u^R \rVert_{\mathcal{G}^{s - 4 + \frac{1}{4}}_2}^2   - \frac{1}{32} \lVert  \partial_y u^R \rVert_{\mathcal{G}^{s - 4}_2}^2 \\
      &\qquad \lesssim  C_{\mathrm{in}} \left( \mathcal{E}_h + \mathcal{E}_u + \mathrm{E}_u + \lVert  h^R \rVert_{\mathcal{G}^{s - 4}_2}^2\right) + \varepsilon ^4 \lVert  \partial_t u^\varepsilon  \rVert_{\mathcal{G}^{s  -3}_2}^2  +  \varepsilon ^2 \sqrt{C_{\mathrm{in}} \mathcal{E}_h}. 
    \end{split} 
\end{align} 
Again using \(\dot{\lambda} \geq 32 C C_{\mathrm{in}}\) to move \(\lVert  h \rVert_{\mathcal{G}^{s - 4}_2}^2\) to left, we conclude 
\begin{equation}\label{eq:Energy Inequality Reduced for Convergence}
    \frac{\mathrm{d}}{\mathrm{d}t}\left(\mathcal{E}_u + \mathcal{E}_h +\mathrm{E}_u +  \lVert  ( u^R, v^R, \partial_y u^R) \rVert_{\mathcal{G}^{s - 6}_2}^2\right)  \lesssim  C_{\mathrm{in}} (\mathcal{E}_h + \mathcal{E}_u + \mathrm{E}_u ) +   \varepsilon ^4 \left(\lVert  \partial_t u^\varepsilon \rVert_{\mathcal{G}^{s - 3}_2}^2 + C_{\mathrm{in}}\right).  
\end{equation}  

Thanks to \eqref{eq:Smallness on Eu for convergence}, we have for \(T_0\) coming from Theorem \ref{thm:Local well-posedness}  
\begin{equation}\label{eq:Convergence:Integral control on dtu^e}
    \int_{0}^{T_0} \lVert  \partial_t u^\varepsilon \rVert_{\mathcal{G}^{s - 3}_2}^2 \leq 2 C_{\mathrm{in}} .  
\end{equation} 
Using Gronwall's inequality, we conclude
\begin{equation}\label{eq:Convergence:Final Energy bound}
    \mathcal{E}_u + \mathcal{E}_h + \mathrm{E}_u \leq   \varepsilon ^4 ( 2 C_{\mathrm{in}} + C_{\mathrm{in}} T_0) e^{C C_{\mathrm{in}} T_0 } \leq C_{{\mathrm{in}}, c_0, \delta _0 } \varepsilon ^4, \quad t \leq T_0.
\end{equation} 
Thus \eqref{eq:Convergence:Final Energy bound} combined with \eqref{eq:Convergence:Definition of energies} implies Theorem \ref{thm:Convergence rate}. \qed

\section{Global well-posedness in Gevrey-\texorpdfstring{\(2\)}{2} class}
In this section, we study the global well-posedness with small data and prove Theorem \ref{thm:Global well-posedness}. 

Let us first recall the system: For \((t,x,y) \in \mathbb{R}_+ \times \mathbb{R} \times [0,1]: \)
\begin{equation}\label{eq:Anisotropic 2D NSM without e}  
    \begin{cases} \partial_t u + u \partial_x u + v \partial_y u  - \varepsilon ^2 \partial_{xx} u - \partial_{yy} u  + \partial_x p  = (f + fh - uh^2  - 2uh - u ) , \\ 
    \varepsilon ^2 (\partial_t v + u \partial_x v  + v \partial_y v - \varepsilon ^2 \partial_{xx} v - \partial_{yy} v)  + \partial_y p = - (e  + eh + vh^2 + 2vh  + v ), \\
    \partial_x u + \partial_y v = \partial_x e + \partial_y f = 0, \\ \partial_t h + \partial_x f - \partial_y e = 0, \\ 
    \partial_{tt} h  + \partial_t h - (\varepsilon ^2 \partial_{xx} h + \partial_{yy} h   )  + u \partial_x h + v \partial_y h = 0, \\ 
    (u,v,e, f, h)|_{t = 0} = (u_{\mathrm{in}}, v_{\mathrm{in}}, e_{\mathrm{in}},f_{\mathrm{in}}, h_{\mathrm{in}}),\\
    (u,v,\partial_y e, f, h)|_{y = 0,1} = 0, \\ 
   \end{cases}  
\end{equation} 

Here we have normalised \( \alpha = \beta = \gamma = 1\) and dropped the \(\varepsilon\) superscripts. Define  \( \tilde{\phi} := e^{\theta t } \phi \) where \(0 <  \theta < \min \{\frac{\delta _0}{2} , 1 \} \) is small enough and satisfies 
\begin{equation}\label{eq:Definition of theta}
    \lVert  \partial_y \phi \rVert_{L^2} \geq 10 \theta \lVert  \phi \rVert_{L^2} , \quad   \forall  \phi \in H^1_0(0,1). 
\end{equation} 

Then for \(\sigma = s - \frac{3}{4}\) and \(s \geq 10\), define 
\begin{align*}
    \mathcal{E}_{u,v} &:= \lVert  (\tilde{u}, \varepsilon \tilde{v}) \rVert_{\mathcal{G}^{s}_2}^2, \quad \mathcal{E}_{h} := \lVert  (\partial _t, \varepsilon \partial_x, \partial_y ) \tilde{h}  \rVert_{\mathcal{G}^{\sigma}_2}^2 - ( \theta - \theta ^2) \lVert  \tilde{h} \rVert_{\mathcal{G}^{s}_2}^2    ,  \\
    \mathcal{CK}_{h} & := \dot{\lambda} \lVert  ( \partial _t, \varepsilon \partial_x, \partial_y) \tilde{h} \rVert_{\mathcal{G}^{\sigma + \frac{1}{4}}_2}^2 + 2\dot{\lambda} \lVert  \partial_t [ \mathrm{A}_{\sigma + \frac{1}{4}} \tilde{h}] \rVert_{L^2}^2 + \dot{\lambda}^3 \lVert  \tilde{h} \rVert_{\mathcal{G}^{s}_2}^2, \\
    \mathrm{E}_{u} &:= \lVert (u,v,\partial_yu ,\sqrt{2}\varepsilon \partial_x u,  \varepsilon ^2 \partial_x v) \rVert_{\mathcal{G}^{s-2}_2}^2 , \quad\mathrm{E}_{h}  := \lVert  (\partial_t, \varepsilon \partial_x, \partial_y) h \rVert_{\mathcal{G}^{s - 1}_2}
\end{align*}   
We retain the positivity of the energy \(\mathcal{E}_h\) as \eqref{eq:Definition of theta} implies 
\[ \frac{1}{2} \lVert  (\partial_t, \varepsilon \partial_x, \partial_y ) \tilde{h} \rVert_{\mathcal{G}^{\sigma}_2}^2 \leq  \mathcal{E}_h. 
\] 
We now proceed with the standard bootstrap argument. To that end, we obtain a priori estimate on the energies: 
\subsection{Bootstrap on \texorpdfstring{\(\mathcal{E}_h\)}{Eh}}
By acting \( \mathrm{A}_{\sigma} e^{\theta t} := e^{(\delta _0 - \lambda(t)) \langle  \partial_x \rangle^{\frac{1}{2}}} \langle  \partial _x \rangle^s e^{\theta t }\) on the equation for \(h\), and using 
\begin{align*}
    \partial_t \tilde{h}     &= \theta \tilde{h} + e^{\theta t } \partial_t h , \\
    \partial_{tt} \tilde{h} &= \theta ^2 \tilde{h} + 2\theta e^{\theta t } \partial_t h  + e^{\theta t} \partial_{tt} h  =  - \theta  ^2 \tilde{h} + 2 \theta \partial_t \tilde{h} + e^{\theta t} \partial_{tt} h, 
\end{align*} 
we obtain 
\begin{align*}
    0 &= \mathrm{A}_{\sigma} [ \partial_{tt} \tilde{h}] - 2 \theta \mathrm{A}_{\sigma} [ \partial_t \tilde{h}] + \theta ^2 \mathrm{A}_{\sigma} \tilde{h}  + \mathrm{A}_{\sigma} \partial_t \tilde{h}  - \theta \mathrm{A}_{\sigma} \tilde{h} - \mathrm{A}_{\sigma} ( \varepsilon ^2 \partial_{xx} \tilde{h} + \partial_{yy} \tilde{h})  +\mathrm{A}_{\sigma} ( (u,v) \cdot \nabla \tilde{h}) \\
    &= \partial_t ( \mathrm{A}_{\sigma} [\partial_t \tilde{h}]) + \dot{\lambda} \langle \partial_x \rangle^{\frac{1}{2}} \mathrm{A}_{\sigma} [ \partial_t \tilde{h}] + (1 - 2 \theta) \mathrm{A}_{\sigma} [ \partial_t \tilde{h}]  - (\theta - \theta ^2) \mathrm{A}_{\sigma} \tilde{h} \\
    &\qquad\qquad\quad - \mathrm{A}_{\sigma} ( \varepsilon ^2 \partial_{xx} \tilde{h} + \partial_{yy} \tilde{h}) + \mathrm{A}_{\sigma}((u,v) \cdot \nabla \tilde{h}). 
\end{align*} 
Inner-producting with \(\mathrm{A}_{\sigma} [ \partial_t \tilde{h}]\), we obtain 
\begin{align*}
    &\frac{1}{2}\frac{\mathrm{d}}{\mathrm{d}t} \lVert  \partial_t \tilde{h} \rVert_{\mathcal{G}^{\sigma}_2}^2 + \dot{\lambda} \lVert  \partial_t \tilde{h} \rVert_{\mathcal{G}^{\sigma +  \frac{1}{4}}_2}^2 + (1 - 2 \theta) \lVert \partial_t \tilde{h} \rVert_{\mathcal{G}^{\sigma}_2}^2 =  \\
    &\qquad\qquad ( \theta  - \theta ^2)  \langle  \mathrm{A}_{\sigma} \tilde{h} , \mathrm{A}_{\sigma} [\partial_t \tilde{h}] \rangle  + \langle  \mathrm{A}_{\sigma} (\varepsilon ^2 \partial_{xx} + \partial_{yy}) \tilde{h}, \mathrm{A}_{\sigma} [ \partial_t\tilde{h}] \rangle  - \langle  \mathrm{A}_{\sigma} ((u,v) \cdot \nabla \tilde{h}), \mathrm{A}_{\sigma} [\partial_t \tilde{h}] \rangle. 
\end{align*} 
Using \( \partial_t[ \mathrm{A}_{\sigma} f] = \mathrm{A}_{\sigma} [\partial_t f ]- \dot{\lambda} \mathrm{A}_{s  + \frac{1}{2}} f \), we get 
\begin{align*}
    \langle  \mathrm{A}_{\sigma} \tilde{h}, \mathrm{A}_{\sigma} [ \partial_t \tilde{h}] \rangle &= \frac{1}{2}\frac{\mathrm{d}}{\mathrm{d}t} \lVert  \tilde{h} \rVert_{\mathcal{G}^{\sigma}_2}^2 + \dot{\lambda} \lVert  \tilde{h} \rVert_{\mathcal{G}^{\sigma +  \frac{1}{4}}_2}^2, \\
    \langle  \mathrm{A}_{\sigma}(\varepsilon ^2 \partial_{xx} + \partial_{yy} \tilde{h}) , \mathrm{A}_{\sigma} [ \partial_t \tilde{h}] \rangle &= -\frac{1}{2} \frac{\mathrm{d}}{\mathrm{d}t} \lVert  (\varepsilon \partial_x, \partial_y) \tilde{h} \rVert_{\mathcal{G}^{\sigma}_2}^2 - \dot{\lambda} \lVert  (\varepsilon \partial_x, \partial_y) \tilde{h} \rVert_{\mathcal{G}^{\sigma +  \frac{1}{4}}_2}^2. 
\end{align*} 
We thus have 
\begin{align*}
    &\frac{1}{2} \frac{\mathrm{d}}{\mathrm{d}t} \lVert  (\partial_t, \varepsilon \partial_x, \partial_y ) \tilde{h}  \rVert_{\mathcal{G}^{\sigma}_2}^2 - (\theta - \theta ^2) \lVert \tilde{h} \rVert_{\mathcal{G}^{\sigma}_2}^2 + \dot{\lambda} \lVert  (\partial_t, \varepsilon \partial_x, \partial_y, - (\theta - \theta ^2)) \tilde{h} \rVert_{\mathcal{G}^{\sigma +  \frac{1}{4}}_2}^2  + (1 - 2 \theta) \lVert  \partial_t \tilde{h} \rVert_{\mathcal{G}^{\sigma}_2}^2 \\
    &\qquad\qquad\qquad\qquad = - \langle  \mathrm{A}_{\sigma} ((u,v) \cdot \nabla \tilde{h})  , \mathrm{A}_{\sigma} [ \partial_t \tilde{h}] \rangle. 
\end{align*} 
Since \( \theta\) is small, and we have \( \lVert  \tilde{h} \rVert_{L^2} \leq (10 \theta) ^{-1} \lVert  \partial_y \tilde{h} \rVert_{L^2}\), we get 
\[  \frac{\mathrm{d}}{\mathrm{d}t} \lVert (\partial_t, \varepsilon \partial_x, \partial_y) \tilde{h} \rVert_{\mathcal{G}^{\sigma}_2}^2- (\theta - \theta ^2) \lVert \tilde{h} \rVert_{\mathcal{G}^{\sigma}_2}^2   + \dot{\lambda} \lVert  (\partial_t, \varepsilon \partial_x, \partial_y) \tilde{h} \rVert_{\mathcal{G}^{\sigma +  \frac{1}{4}}_2}^2 + \lVert  \partial_t \tilde{h}  \rVert_{\mathcal{G}^{\sigma}_2} ^2 \lesssim   \lvert \langle  \mathrm{A}_{\sigma} ((u,v) \cdot \nabla \tilde{h}) , \mathrm{A}_{\sigma} [ \partial_t \tilde{h}]\rangle \rvert . 
\] 
We now deal with the \(\dot{\lambda} \lVert  \partial_t \tilde{h} \rVert_{\mathcal{G}^{\sigma + \frac{1}{4}}_2}^2\) term. We have 
\begin{align*}
 \dot{\lambda}    \lVert  \partial_t \tilde{h} \rVert_{\mathcal{G}^{\sigma + \frac{1}{4}}_2}^2 &=  \dot{\lambda} \langle  \mathrm{A}_{\sigma + \frac{1}{4}} \partial_t \tilde{h} , \mathrm{A}_{\sigma + \frac{1}{4}} \partial_t \tilde{h} \rangle = \dot{\lambda} \lVert  \partial_t [ \mathrm{A}_{\sigma + \frac{1}{4}} \tilde{h}] \rVert_{L^2}^2 + 2 \dot{\lambda} \langle  \partial_t [ \mathrm{A}_{\sigma + \frac{1}{4}} \tilde{h}] , \dot{\lambda} \mathrm{A}_{\sigma + \frac{3}{4}}\tilde{h} \rangle  + \dot{\lambda} \lVert  \mathrm{A}_{\sigma + \frac{3}{4}} \tilde{h}  \rVert_{L^2}^2 \\
 &= \dot{\lambda} \lVert  \partial_t [ \mathrm{A}_{\sigma + \frac{1}{4}} \tilde{h}] \rVert_{L^2}^2  + \frac{\mathrm{d}}{\mathrm{d}t} \left(\dot{\lambda}^2 \lVert \tilde{h} \rVert_{\mathcal{G}^{\sigma + \frac{1}{2}}_2}^2\right)  - 2 \ddot{\lambda}\dot{\lambda} \lVert  \tilde{h} \rVert_{\mathcal{G}^{\sigma + \frac{1}{2}}_2}^2 + \dot{\lambda} ^3 \lVert  \tilde{h} \rVert_{\mathcal{G}^{\sigma + \frac{3}{4}}_2}^2. 
\end{align*} 
Splitting \(\dot{\lambda} \lVert  \partial_t \tilde{h} \rVert_{\mathcal{G}^{\sigma + \frac{1}{4}}}\) into two halves: 
\begin{align*}
    \frac{\mathrm{d}}{\mathrm{d}t} \mathcal{E}_h +  \mathcal{CK}_h + \lVert  \partial_t \tilde{h} \rVert_{\mathcal{G}^{\sigma}_2}^2\lesssim   \lvert \langle  \mathrm{A}_{\sigma} ((u,v) \cdot \nabla \tilde{h}) , \mathrm{A}_{\sigma} [ \partial_t \tilde{h}] \rangle \rvert +  \ddot{\lambda} \dot{\lambda} \lVert  \tilde{h} \rVert_{\mathcal{G}^{\sigma + \frac{1}{2}}_2}^2
\end{align*} 
 where we have fixed \(\ddot{\lambda} \leq 0 \leq \dot{\lambda} \) so the last term is dropped. 
The controls for non-linear terms are:
\begin{align*}
     \lvert  \langle  \mathrm{A}_{\sigma} (u \partial_x \tilde{h}), \mathrm{A}_{\sigma} [ \partial_t \tilde{h}] \rangle \rvert &\leq \left( \lVert  u \rVert_{\mathcal{G}^{2}_\infty} \lVert  \tilde{h} \rVert_{\mathcal{G}^{\sigma + \frac{3}{4}}_2} \lVert  \partial_t \tilde{h} \rVert_{\mathcal{G}^{\sigma + \frac{1}{4}}_2}+ \lVert  u \rVert_{\mathcal{G}^{\sigma - \frac{1}{4}}_\infty} \lVert  \partial_x \tilde{h} \rVert_{\mathcal{G}^{2}_2} \lVert  \partial_t \tilde{h} \rVert_{\mathcal{G}^{\sigma + \frac{1}{4}}_2} \right) \\
     &\lesssim  \lVert  u \rVert_{\mathcal{G}^{2}_2}^{\frac{1}{2}} \lVert  \partial_y u  \rVert_{\mathcal{G}^{2}_2}^{\frac{1}{2}} \lVert  \tilde{h} \rVert_{\mathcal{G}^{s}_2} \lVert  \partial_t \tilde{h} \rVert_{\mathcal{G}^{\sigma + \frac{1}{4}}_2}  + \lVert  u \rVert_{\mathcal{G}^{s - 1}_2}^{\frac{1}{2}} \lVert  \partial_ y u   \rVert_{\mathcal{G}^{s - 1}_2}^{\frac{1}{2}} \lVert  \tilde{h} \rVert_{\mathcal{G}^{3}_2} \lVert  \partial_t \tilde{h} \rVert_{\mathcal{G}^{\sigma}_2} \\ 
     &\lesssim \kappa  e^{ - \theta t / 2}  \lVert  \tilde{h} \rVert_{\mathcal{G}^{s}_2} \lVert  \partial_t \tilde{h} \rVert_{\mathcal{G}^{\sigma + \frac{1}{4}}_2}  + \kappa ^3  e^{-\theta t /2 }  , \\
      \lvert  \langle  \mathrm{A}_{\sigma} ( v \partial_y \tilde{h}) , \mathrm{A}_{\sigma} [ \partial_t \tilde{h}]  \rangle \rvert  &\lesssim \lVert  \partial_t \tilde{h} \rVert_{\mathcal{G}^{\sigma}_2}\left( \lVert  v \rVert_{\mathcal{G}^{\sigma}_\infty} \lVert  \partial_y \tilde{h} \rVert_{\mathcal{G}^{2}_2} + \lVert  v \rVert_{\mathcal{G}^{2}_\infty} \lVert  \partial_y \tilde{h} \rVert_{\mathcal{G}^{\sigma} _2}\right) \\
      &\lesssim  \lVert  \partial_t \tilde{h} \rVert_{\mathcal{G}^{\sigma}_2} \left( \lVert  \partial_y v \rVert_{\mathcal{G}^{\sigma}_2} \lVert  \partial_y \tilde{h} \rVert_{\mathcal{G}^{2}_2} + \lVert  \partial_y v  \rVert_{\mathcal{G}^{2}_2} \lVert  \partial_y \tilde{h} \rVert_{\mathcal{G}^{\sigma}_2}\right) \\
      &\lesssim  \lVert  \partial_t \tilde{h} \rVert_{\mathcal{G}^{\sigma}_2} \left( \lVert  \tilde{u} \rVert_{\mathcal{G}^{s + \frac{1}{4}}_2} \lVert  \partial_y h \rVert_{\mathcal{G}^{2}_2} + \lVert  u \rVert_{\mathcal{G}^{3}_2} \lVert  \partial_y \tilde{h} \rVert_{\mathcal{G}^{\sigma}_2}\right) \\
      &\lesssim  \kappa \left( \lVert  \tilde{u} \rVert_{\mathcal{G}^{s + \frac{1}{4}}_2} \kappa e^{-\theta t} + \kappa ^2 e^{-\theta t }\right) \lesssim  \kappa  e^{-\theta t } \lVert  \tilde{u} \rVert_{\mathcal{G}^{s + \frac{1}{4}}_2}^2 + \kappa ^3 e^{-\theta t }. 
\end{align*}  
Note that the reason to invoke \( \lVert  \phi \rVert_{L^\infty} \leq \lVert  \phi \rVert_{L^2}^{\frac{1}{2}} \lVert  \partial_y \phi \rVert_{L^2}^{\frac{1}{2}}\) is to obtain the term \(\lVert  \phi \rVert_{L^2}^{\frac{1}{2}}\) which can be bounded by \( \leq \kappa e^{-\theta t / 2 }\) which gives a weaker growth requirement for \(\dot{\lambda}\) letting us have integrability for \(\dot{\lambda}\) (necessary to invoke Gronwall later on the full inequality). 

Combining the above we obtain 
\begin{align*}
    \frac{\mathrm{d}}{\mathrm{d}t} \mathcal{E}_h +  \mathcal{CK}_h +  \lVert  \partial_t \tilde{h} \rVert_{\mathcal{G}^{\sigma}_2}^2  &\lesssim  \kappa ^3 e^{-\theta t /2}     + \kappa  e^{-\theta t } \lVert  \tilde{u}\rVert_{\mathcal{G}^{s + \frac{1}{4}}_2}^2 +  \kappa e^{-\theta t / 2} \lVert  \tilde{h} \rVert_{\mathcal{G}^{s}_2} \lVert  \partial_t \tilde{h} \rVert_{\mathcal{G}^{\sigma + \frac{1}{4}}_2} \\
      &\lesssim  \kappa ^3 e^{- \theta t / 2} +  \kappa e^{-\theta t / 2} \left( \lVert  \tilde{u} \rVert_{\mathcal{G}^{s + \frac{1}{4}}_2}^2  + \lVert  \partial_t \tilde{h} \rVert_{\mathcal{G}^{\sigma + \frac{1}{4}}_2}^2 + \lVert  \tilde{h} \rVert_{\mathcal{G}^{s}_2}^2 \right).  
\end{align*} 
Letting \(\dot{\lambda} \geq 16 C \kappa  e^{-\theta t /2 } \land (16 C \kappa e ^{- \theta t / 2} )^{\frac{1}{3}}\) (first to absorb \(\partial_t h\) and second to absorb \(h\)), we get    
       \begin{equation}\label{eq:Bootstrap on Eh for GWP}
        \frac{\mathrm{d}}{\mathrm{d}t} \mathcal{E}_h +  \mathcal{CK}_h + \lVert  \partial_t \tilde{h} \rVert_{\mathcal{G}^{\sigma}_2}^2 \lesssim  \kappa ^3 e^{-\theta t /2} + \kappa e^{-\theta t / 2}  \lVert \tilde{u} \rVert_{\mathcal{G}^{s + \frac{1}{4}}_2}^2. 
       \end{equation}

\subsection{Bootstrap on \texorpdfstring{\(\mathcal{E}_u\)}{Eu}}
Recall that 
\[  \partial_t [ \mathrm{A}_{s} \tilde{u}] = \mathrm{A}_{s} e^{\theta t } [ \partial_t u]  - \dot{\lambda} \langle  \partial_x \rangle^{\frac{1}{2}}  \mathrm{A}_{s} \tilde{u} + \theta \mathrm{A}_{s} \tilde{u}. 
\] 
By acting \( \mathrm{A}_{s} e^{\theta t }\) on the equation of \((u,v)\) and using above, we have 
\begin{align*}
    &\partial_t [ \mathrm{A}_{s} \tilde{u}] + \dot{\lambda} \langle  \partial_x \rangle^{\frac{1}{2}} \mathrm{A}_{s} \tilde{u}    + (1 - \theta) \mathrm{A}_{s} \tilde{u}  + \mathrm{A}_{s} ((u,v) \cdot \nabla \tilde{u}) - \mathrm{A}_{s} (\varepsilon ^2 \partial_{xx} \tilde{u} + \partial_{yy} \tilde{u}) + \partial_x \mathrm{A}_{s} \tilde{p} \\
    &\qquad\qquad\quad  = \mathrm{A}_{s} ( \tilde{f} + f \tilde{h} - u h \tilde{h} - 2 u \tilde{h}), \\ 
    & \varepsilon ^2 \left( \partial_t [ \mathrm{A}_{s} \tilde{v}] + \dot{\lambda} \langle  \partial_x \rangle^{\frac{1}{2}} \mathrm{A}_{s} \tilde{v}  + \mathrm{A}_{s} ( (u,v) \cdot \nabla \tilde{v}) - \mathrm{A}_{s} (\varepsilon ^2 \partial_{xx} \tilde{v} + \partial_{yy} \tilde{v})\right) + (1 - \varepsilon ^2 \theta) \mathrm{A}_{s} \tilde{v} + \partial_y \mathrm{A}_{s} \tilde{p} \\
    &\qquad\qquad\quad = - \mathrm{A}_{s} ( \tilde{e} + e \tilde{h} + v h \tilde{h} + 2 v \tilde{h}). 
\end{align*} 
Inner-producting with \( (\mathrm{A}_{s} \tilde{u}, \mathrm{A}_{s} \tilde{v})\), we get 
\begin{align*}
    &\frac{1}{2} \frac{\mathrm{d}}{\mathrm{d}t} \lVert  (\tilde{u}, \varepsilon \tilde{v}) \rVert_{\mathcal{G}^{s}_2} ^2  + \dot{\lambda} \lVert  (\tilde{u}, \varepsilon \tilde{v}) \rVert_{\mathcal{G}^{s + \frac{1}{4}}_2}^2  + \lVert  ((1- \theta) \tilde{u} , (1 - \varepsilon ^2 \theta) \tilde{v}) \rVert_{\mathcal{G}^{s}_2}^2  + \lVert  ( \partial_y \tilde{u} , \sqrt{2}\varepsilon \partial_x \tilde{u}, \varepsilon ^2 \partial_x \tilde{v}) \rVert_{\mathcal{G}^{s}_2}^2 \\
    &\qquad\qquad\quad = - \langle  \mathrm{A}_{s} ((u,v) \cdot \nabla \tilde{u}) , \mathrm{A}_{s} \tilde{u} \rangle  - \varepsilon ^2 \langle  \mathrm{A}_{s} ((u,v) \cdot \nabla \tilde{v}) , \mathrm{A}_{s} \tilde{v} \rangle \\
    &\qquad\qquad\qquad \qquad +  \langle  \mathrm{A}_{s} (\tilde{f}  + f \tilde{h} - uh \tilde{h} - 2 u \tilde{h}) , \mathrm{A}_{s} \tilde{u} \rangle -\langle   \mathrm{A}_{s} ( \tilde{e} + e \tilde{h} + v h \tilde{h}  + 2 v \tilde{h}) , \mathrm{A}_{s} \tilde{v} \rangle. 
\end{align*}  
Here come the estimates: 
\begin{align*}
     \lvert \langle \mathrm{A}_{ s} (u \partial_x \tilde{u}) , \mathrm{A}_{s} \tilde{u} \rangle \rvert  &\lesssim   \lvert  \langle  u \partial _x \mathrm{A}_{s} \tilde{u}, \mathrm{A}_{s} \tilde{u} \rangle \rvert +  \lvert  \langle  [ \mathrm{A}_{s} , u \partial_x] \tilde{u}, \mathrm{A}_{s} \tilde{u} \rangle \rvert  \\
     &\lesssim  \lVert  \partial_x u \rVert_{L^\infty_{x,y}} \lVert  \tilde{u} \rVert_{\mathcal{G}^{s}_2}^2 + \lVert  u \rVert_{\mathcal{G}^{2}_{\infty}} \lVert  \tilde{u} \rVert_{\mathcal{G}^{s + \frac{1}{4}}_2}^2  + \lVert  u \rVert_{\mathcal{G}^{s}_2} \lVert  \tilde{u} \rVert_{\mathcal{G}^{s}_2} \lVert  \tilde{u} \rVert_{\mathcal{G}^{2}_{\infty}} \\
     &\lesssim  \kappa e^{-\theta t / 2} \lVert  \tilde{u} \rVert_{\mathcal{G}^{s + \frac{1}{4}}_2}^2  + \lVert  u \rVert_{\mathcal{G}^{2}_\infty} \lVert  \tilde{u} \rVert_{\mathcal{G}^{s}_2}^2 \lesssim  \kappa e^{-\theta t / 2} \lVert  \tilde{u} \rVert_{\mathcal{G}^{s + \frac{1}{4}}_2}^2 + \kappa ^3 e^{-\theta t / 2},  \\ 
      \lvert  \langle  \mathrm{A}_{s} (v \partial_y \tilde{u}) , \mathrm{A}_{s} \tilde{u} \rangle \rvert &\lesssim \lVert  \tilde{u} \rVert_{\mathcal{G}^{s}_\infty} \left( \lVert  v \rVert_{\mathcal{G}^{s}_2} \lVert  \partial_y \tilde{u} \rVert_{\mathcal{G}^{2}_2} + \lVert  v \rVert_{\mathcal{G}^{2}_2} \lVert  \partial_y \tilde{u} \rVert_{\mathcal{G}^{s}_2}\right) \\
      &\lesssim  \lVert  \partial_y \tilde{u} \rVert_{\mathcal{G}^{s}_2}\kappa e^{-\theta t } \left(\lVert  \tilde{v} \rVert_{\mathcal{G}^{s}_2} +  \lVert  \partial_y \tilde{u} \rVert_{\mathcal{G}^{s}_2} \right) \lesssim  \kappa e^{-\theta t } \left( \lVert  \tilde{v} \rVert_{\mathcal{G}^{s}_2}^2 + \lVert  \partial_y \tilde{u} \rVert_{\mathcal{G}^{s}_2}^2\right) , \\
      \varepsilon ^2  \lvert \langle \mathrm{A}_{s} ( u \partial_x \tilde{v}), \mathrm{A}_{s} \tilde{v}  \rangle\rvert &\lesssim \lVert  u \rVert_{\mathcal{G}^{2}_\infty} \lVert \varepsilon \tilde{v} \rVert_{\mathcal{G}^{s + \frac{1}{4}}_2}^2  + \lVert  u \rVert_{\mathcal{G}^{s}_2} \lVert \varepsilon \tilde{v} \rVert_{\mathcal{G}^{s}_2} \lVert \varepsilon  \tilde{v} \rVert_{\mathcal{G}^{2}_\infty} \\
      &\lesssim  \kappa e^{-\theta t / 2} \lVert  \varepsilon\tilde{v} \rVert_{\mathcal{G}^{s + \frac{1}{4}}_2}^2  + \lVert  u \rVert_{\mathcal{G}^{s}_2} \lVert \varepsilon \tilde{v} \rVert_{\mathcal{G}^{s}_2}  \lVert  \varepsilon \partial_y \tilde{v} \rVert_{\mathcal{G}^{2}_2} \\ 
      &\lesssim \kappa e^{-\theta t / 2} \lVert  \varepsilon \tilde{v} \rVert_{\mathcal{G}^{s + \frac{1}{4}}_2}^2 +  \kappa e^{-\theta t}\lVert  (1, \partial_y) \varepsilon \tilde{v} \rVert_{\mathcal{G}^{s}_2}^2 , \\
     \varepsilon ^2 \lvert  \langle  \mathrm{A}_{s} (v \partial_y \tilde{v}), \mathrm{A}_{s} \tilde{v} \rangle \rvert &\lesssim  \lVert  \varepsilon\tilde{v} \rVert_{\mathcal{G}^{s}_\infty}  \left( \lVert  v \rVert_{\mathcal{G}^{2}_2} \lVert  \varepsilon \partial_y \tilde{v} \rVert_{\mathcal{G}^{s}_2} + \lVert  \varepsilon \tilde{v} \rVert_{\mathcal{G}^{s}_2} \lVert  \partial_y v \rVert_{\mathcal{G}^{2}_2} \right) \\
     &\lesssim  \lVert  \varepsilon \partial_y v \rVert_{\mathcal{G}^{s}_2} \left( \lVert  u \rVert_{\mathcal{G}^{3}_2} \lVert  \varepsilon \partial_y \tilde{v} \rVert_{\mathcal{G}^{s}_2} + \lVert  \varepsilon \tilde{v} \rVert_{\mathcal{G}^{s}_2} \lVert  u \rVert_{\mathcal{G}^{3}_2}\right) \lesssim  \kappa \lVert  (1, \partial_y) \varepsilon \tilde{v} \rVert_{\mathcal{G}^{s}_2}^2. 
\end{align*} 
We also have 
\[  \langle  \mathrm{A}_{s} \hat{f} , \mathrm{A}_{s} \hat{u} \rangle  + \langle  \mathrm{A}_{s} \hat{e} , \mathrm{A}_{s} \hat{v} \rangle  = \langle  \mathrm{A}_{s} \hat{f}, \mathrm{A}_{s} (- \partial_y \hat{\phi}) \rangle + \langle  \mathrm{A}_{s} \hat{e} , \mathrm{A}_{s} (\partial_x \hat{\phi} ) \rangle = \langle  \mathrm{A}_{s} (\nabla \cdot (e,f)) , \mathrm{A}_{s} \hat{\phi} \rangle = 0, 
\] 
where we have used the divergence-free condition of \((e,f)\) and the existence of stream-function for \((u,v)\) proven in section 3. For the remaining, 
\begin{align*}
    \langle  \langle  \mathrm{A}_{s} ( f \tilde{h}  -u h \tilde{h} - 2 u \tilde{h}) , \mathrm{A}_{s} \tilde{u} \rangle \rangle 
    &\lesssim   \lVert  \tilde{u} \rVert_{\mathcal{G}^{s}_2} \lVert  \tilde{h} \rVert_{\mathcal{G}^{s}_2}  \left( \lVert  f \rVert_{\mathcal{G}^{2}_\infty} + \lVert  h \rVert_{\mathcal{G}^{2}_\infty} \lVert  u \rVert_{\mathcal{G}^{2}_\infty}  + \lVert  u \rVert_{\mathcal{G}^{2}_\infty}  \right) \\
    & \qquad\qquad + \lVert  \tilde{u} \rVert_{\mathcal{G}^{s}_2} \lVert  \tilde{h} \rVert_{\mathcal{G}^{2}_\infty} \left( \lVert  f \rVert_{\mathcal{G}^{s}_2} + \lVert  h \rVert_{\mathcal{G}^{2}_\infty} \lVert  u \rVert_{\mathcal{G}^{s}_2} + \lVert  u \rVert_{\mathcal{G}^{s}_2}\right)  \\
    &\lesssim  \lVert  \tilde{u} \rVert_{\mathcal{G}^{s}_2} \lVert  \tilde{h} \rVert_{\mathcal{G}^{s}_2}  \left( \lVert  \partial_t h \rVert_{\mathcal{G}^{2}_2}  + \lVert  \partial_y h \rVert_{\mathcal{G}^{2}_2} \lVert  u \rVert_{\mathcal{G}^{2}_2}^{\frac{1}{2}} \lVert  \partial_yu  \rVert_{\mathcal{G}^{2}_2}^{\frac{1}{2}} + \lVert  u \rVert_{\mathcal{G}^{2}_2}^{\frac{1}{2}} \lVert  \partial_y u \rVert_{\mathcal{G}^{2}_2}^{\frac{1}{2}}\right) \\
    & \qquad\qquad + \lVert  \tilde{u} \rVert_{\mathcal{G}^{s}_2} \lVert  \partial_y h \rVert_{\mathcal{G}^{2}_2}\left( \lVert  \partial_t \tilde{h} \rVert_{\mathcal{G}^{\sigma}_2} + \lVert  \partial_y \tilde{h} \rVert_{\mathcal{G}^{2}_2} \lVert  u \rVert_{\mathcal{G}^{s}_2} + \lVert  \tilde{u} \rVert_{\mathcal{G}^{s}_2}\right) \\
    &\lesssim  \kappa \lVert  \tilde{h} \rVert_{\mathcal{G}^{s}_2} \left(\kappa e^{-\theta t  }  + \kappa ^2 e^{-\theta t / 2}   + \kappa e^{-\theta t / 2 }\right)  + \kappa ^2 e^{-\theta t} ( \kappa + \kappa ^2 e^{-\theta t } + \kappa ) \\
    &\lesssim  \kappa ^2 e^{-\theta t / 2} \lVert  \tilde{h} \rVert_{\mathcal{G}^{s}_2}  + \kappa ^3 e^{-\theta t } \lesssim  \kappa e^{-\theta t/2}(\kappa ^2  + \lVert  \tilde{h} \rVert_{\mathcal{G}^{s}_2}^2) + \kappa ^3 e^{-\theta t } \\
    &\lesssim  \kappa e^{-\theta t / 2}\lVert  \tilde{h} \rVert_{\mathcal{G}^{s}_2} ^2  + \kappa ^3 e^{-\theta t /2}, \\
     \lvert  \langle  \mathrm{A}_{s} ( e \tilde{h}  + v h \tilde{h} + 2 v \tilde{h}) , \mathrm{A}_{s} \tilde{v} \rangle \rvert &\lesssim  \lVert \tilde{v} \rVert_{\mathcal{G}^{s}_2} \lVert  \tilde{h} \rVert_{\mathcal{G}^{s}_2} \left( \lVert  e \rVert_{\mathcal{G}^{2}_\infty} + \lVert  h \rVert_{\mathcal{G}^{2}_\infty} \lVert  v \rVert_{\mathcal{G}^{2}_\infty} + \lVert  v \rVert_{\mathcal{G}^{2}_\infty}\right) \\
     &  \qquad\qquad + \lVert  \tilde{v} \rVert_{\mathcal{G}^{s}_2} \lVert  h \rVert_{\mathcal{G}^{2}_\infty} \left( \lVert \tilde{e} \rVert_{\mathcal{G}^{s}_2} + \lVert  h \rVert_{\mathcal{G}^{2}_\infty} \lVert  \tilde{v} \rVert_{\mathcal{G}^{s}_2} + \lVert  \tilde{v} \rVert_{\mathcal{G}^{s}_2}\right) \\
     &\lesssim  (\lVert  \tilde{v} \rVert_{\mathcal{G}^{s}_2}^2 + \lVert  \tilde{h} \rVert_{\mathcal{G}^{s}_2}^2) \left( \lVert  \partial_t h \rVert_{\mathcal{G}^{2}_2} + \lVert  \partial_y h \rVert_{\mathcal{G}^{2}_2} \lVert  u \rVert_{\mathcal{G}^{3}_2} + \lVert u \rVert_{\mathcal{G}^{3}_2}\right)  \\
     &\qquad\qquad + \lVert  \tilde{v} \rVert_{\mathcal{G}^{s}_2} \kappa e^{-\theta t } \left( \kappa + \kappa e^{-\theta t } \lVert  \tilde{v} \rVert_{\mathcal{G}^{s}_2} + \lVert  \tilde{v} \rVert_{\mathcal{G}^{s}_2}\right) \\
     &\lesssim  \kappa e^{-\theta t} \left( \lVert  \tilde{v} \rVert_{\mathcal{G}^{s}_2}^2 + \lVert  \tilde{h} \rVert_{\mathcal{G}^{s}_2}^2\right) +  \kappa ^3 e^{-\theta t }. 
\end{align*}  
Combining the results, we have 
\begin{align*}
    &\frac{1}{2}\frac{\mathrm{d}}{\mathrm{d}t} \lVert  (\tilde{u}, \varepsilon \tilde{v}) \rVert_{\mathcal{G}^{s}_2}^2 + \dot{\lambda} \lVert  (\tilde{u}, \varepsilon \tilde{v}) \rVert_{\mathcal{G}^{s + \frac{1}{4}}_2}^2 + \lVert  (( 1- \theta) \tilde{u} ,(1 - \varepsilon ^2 \theta) \tilde{v} ) \rVert_{\mathcal{G}^{s}_2}^2 + \lVert  (\partial_y \tilde{u}, \sqrt{ 2} \varepsilon \partial_x \tilde{u}, \varepsilon ^2 \partial_x \tilde{v}) \rVert_{\mathcal{G}^{s}_2}^2 \\
    &\qquad\qquad \lesssim  \kappa ^3 e^{-\theta t / 2} + \kappa e^{-\theta t / 2}\lVert  (\tilde{u}, \varepsilon \tilde{v}) \rVert_{\mathcal{G}^{s + \frac{1}{4}}_2}^2 + \kappa  \lVert  (\partial_y, \varepsilon \partial_x) \tilde{u} \rVert_{\mathcal{G}^{s}_2}^2 + \kappa e^{-\theta t / 2} \lVert  \tilde{h} \rVert_{\mathcal{G}^{s}_2}^2. 
\end{align*} 
As \(\dot{\lambda} \geq 16 C \kappa e^{- \theta  t / 2}\) and \( 1 \geq 16 C \kappa  \) so we have  
       \begin{equation}\label{eq:Bootstrap on Eu for GWP}
        \frac{\mathrm{d}}{\mathrm{d}t} \mathcal{E}_{u,v} + \dot{\lambda} \lVert  (\tilde{u}, \varepsilon  \tilde{v}) \rVert_{\mathcal{G}^{s + \frac{1}{4}}_2}^2 +  \lVert (\tilde{u}, \tilde{v}) \rVert_{\mathcal{G}^{s}_2}^2 +  \lVert  (\partial_y \tilde{u}, \sqrt{ 2} \varepsilon \partial_x \tilde{u}, \varepsilon ^2 \partial_x \tilde{v}) \rVert_{\mathcal{G}^{s}_2}^2 \lesssim  \kappa ^3 e^{-\theta t / 2} + \kappa e^{-\theta t / 2} \lVert  \tilde{h} \rVert_{\mathcal{G}^{s}_2}^2. 
       \end{equation}  
\subsection{Bootstrap on \texorpdfstring{\(\mathrm{E}_u\)}{Elow}.}
Acting \( \mathrm{A}_{s-2}\) on the equation for \(u,v\), we obtain 
\begin{align*}
    &\mathrm{A}_{s-2} ( \partial _t u ) + \mathrm{A}_{s-2} u  + \mathrm{A}_{s-2} ((u,v) \cdot \nabla u) - \mathrm{A}_{s-2} (\varepsilon ^2 \partial_{xx} u  + \partial_{yy} u ) + \partial_x \mathrm{A}_{s-2} p 
    \\ & \qquad\qquad\qquad\qquad \qquad\qquad\qquad\qquad  = \mathrm{A}_{s-2} ( f + fh - uhh - 2uh) , \\ 
   & \varepsilon ^2 \left( \mathrm{A}_{s-2} ( \partial _t v ) + \mathrm{A}_{s-2} ((u,v) \cdot \nabla v) - \mathrm{A}_{s-2} (\varepsilon ^2 \partial_{xx} v + \partial_{yy} v)\right) + \mathrm{A}_{s-2} v  + \partial_y \mathrm{A}_{s-2} p  \\
  &\qquad\qquad\qquad\qquad \qquad\qquad\qquad\qquad  = - \mathrm{A}_{s-2} (e + eh + v hh + 2vh). 
\end{align*} 
Hitting the equations by \( \mathrm{A}_{s-2} (\partial_t u, \partial_t v)\), we obtain  
\begin{align*}
    &\lVert  \partial_t (u,\varepsilon v ) \rVert_{\mathcal{G}^{s-2}_2}^2 + \frac{1}{2}\frac{\mathrm{d}}{\mathrm{d}t}\left( \lVert  (u, v) \rVert_{\mathcal{G}^{s-2}_2}^2 + \lVert  (\partial_y u , 2\varepsilon \partial_x u , \varepsilon ^2 \partial_x v ) \rVert_{\mathcal{G}^{s-2}_2}^2  \right)\\
    &\qquad\qquad\quad\quad+ \dot{\lambda} \left( \lVert  (u, v) \rVert_{\mathcal{G}^{\sigma -1}_2}^2 + \lVert  (\partial_y u , 2\varepsilon \partial_x u , \varepsilon ^2 \partial_x v ) \rVert_{\mathcal{G}^{\sigma -1}_2}^2  \right) \\
    & \qquad\qquad\quad  = \langle  \mathrm{A}_{s-2} (  fh - uhh - 2uh) , \mathrm{A}_{s-2}( \partial_t u ) \rangle  - \langle   \mathrm{A}_{s-2} ( eh + vhh + 2vh) , \mathrm{A}_{s-2} ( \partial_t v) \rangle,
\end{align*} 
where we have used the divergence-free conditions on velocity and magnetic fields to remove the pressure and magnetic field terms. 
A direct calculation shows that 
\begin{align*}
    &\lvert  \langle  \mathrm{A}_{s-2} ( fh - uhh - 2uh ), \mathrm{A}_{s-2} (\partial _t u) \rangle \rvert \\
    &\lesssim  \lVert  \partial_t u  \rVert_{\mathcal{G}^{s-2}_2}  \left( \lVert  f \rVert_{\mathcal{G}^{s-2}_2} \lVert  h \rVert_{\mathcal{G}^{2}_\infty} + \lVert  f \rVert_{\mathcal{G}^{2}_\infty} \lVert  h \rVert_{\mathcal{G}^{s-2}_2}\right) \\
    & \qquad\qquad\quad + \lVert  \partial_t u  \rVert_{\mathcal{G}^{s-2}_2} \left(\lVert  u \rVert_{\mathcal{G}^{s-2}_2} \lVert  h \rVert_{\mathcal{G}^{2}_\infty}^2 + \lVert  u \rVert_{\mathcal{G}^{2}_\infty} \lVert  h \rVert_{\mathcal{G}^{2}_\infty} \lVert  h \rVert_{\mathcal{G}^{s-2}_2}\right)  \\
    & \qquad\qquad\quad + \lVert  \partial_t u \rVert_{\mathcal{G}^{s-2}_2} \left( \lVert  u \rVert_{\mathcal{G}^{s-2}_2} \lVert  h \rVert_{\mathcal{G}^{2}_{\infty}} + \lVert  u \rVert_{\mathcal{G}^{2}_\infty} \lVert  h \rVert_{\mathcal{G}^{s-2}_2}\right) \\
    &\lesssim  \lVert  \partial_t u \rVert_{\mathcal{G}^{s-2}_2} \left( \lVert  \partial_t h \rVert_{\mathcal{G}^{s-2}_2}\lVert  \partial_y h \rVert_{\mathcal{G}^{2}_2} + \lVert  \partial_t h \rVert_{\mathcal{G}^{2}_2} \lVert  h \rVert_{\mathcal{G}^{s-2}_2} + \lVert  u \rVert_{\mathcal{G}^{s-2}_2} \lVert  \partial_y h  \rVert_{\mathcal{G}^{2}_2}^2 + \lVert  u \rVert_{\mathcal{G}^{3}_2} \lVert  \partial_y h \rVert_{\mathcal{G}^{2}_2} \lVert  h \rVert_{\mathcal{G}^{s-2}_2}  \right) \\
    &\qquad\qquad + \lVert  \partial_t u \rVert_{\mathcal{G}^{s-2}_2} \left( \lVert  u \rVert_{\mathcal{G}^{s-2}_2} \lVert  \partial_y h \rVert_{\mathcal{G}^{2}_2} + \lVert  u \rVert_{\mathcal{G}^{2}_2}^{\frac{1}{2}} \lVert  \partial_y u \rVert_{\mathcal{G}^{2}_2} \lVert h \rVert_{\mathcal{G}^{s-2}_2}\right) \\
    &\leq \frac{1}{2} \lVert  \partial_t u \rVert_{\mathcal{G}^{s-2}_2}^2  +C \left( \kappa ^2 e^{-\theta t } + \kappa ^3 e^{- \theta t}  + \kappa ^2 e^{-\theta t } + \kappa ^2 e^{-\theta t }\right) ^2\leq \frac{1}{2} \lVert  \partial_t u \rVert_{\mathcal{G}^{s-2}_2}^{2}  + \kappa ^4 e^{-2\theta t }. 
 \end{align*} 
Similarly, 
\begin{align*}
     &\lvert  \langle \mathrm{A}_{s - 1} ( eh + vhh + 2vh), \mathrm{A}_{s-3} (\partial_t v) \rangle \rvert \\
     &\lesssim   \lVert  \partial_t v  \rVert_{\mathcal{G}^{s-3}_\infty} \left( \lVert  f \rVert_{\mathcal{G}^{s - 1}_2} \lVert  h \rVert_{\mathcal{G}^{2}_2} + \lVert  f \rVert_{\mathcal{G}^{2}_2} \lVert  h \rVert_{\mathcal{G}^{s - 1}_2}\right) \\
     & \qquad\qquad\quad  + \lVert  \partial_t v  \rVert_{\mathcal{G}^{s - 3}_{\infty}}\left( \lVert v \rVert_{\mathcal{G}^{s-1}_2} \lVert h \rVert_{\mathcal{G}^{2}_2} \lVert  h \rVert_{\mathcal{G}^{2}_\infty} + \lVert  h \rVert_{\mathcal{G}^{s-1}_2} \lVert  h \rVert_{\mathcal{G}^{2}_\infty} \lVert  v \rVert_{\mathcal{G}^{2}_2}\right) \\
     &\qquad\qquad\quad + \lVert  \partial_t v \rVert_{\mathcal{G}^{s - 2}_2} \left( \lVert v \rVert_{\mathcal{G}^{s-1}_2} \lVert  h \rVert_{\mathcal{G}^{2}_2} +\lVert  v \rVert_{\mathcal{G}^{2}_2} \lVert  h \rVert_{\mathcal{G}^{s-1}_2}\right) \\
     &\lesssim  \lVert  \partial_t u \rVert_{\mathcal{G}^{s-2}_2} \left( \lVert \partial_t h \rVert_{\mathcal{G}^{s-2}_2} \lVert  h \rVert_{\mathcal{G}^{2}_2} + \lVert  \partial_t h \rVert_{\mathcal{G}^{2}_2} \lVert  h \rVert_{\mathcal{G}^{s}_2} + \lVert  v \rVert_{\mathcal{G}^{s}_2} \lVert  h  \rVert_{\mathcal{G}^{2}_2} \lVert  \partial_y h \rVert_{\mathcal{G}^{2}_2} + \lVert  u \rVert_{\mathcal{G}^{3}_2} \lVert  \partial_y h \rVert_{\mathcal{G}^{2}_2} \lVert  h \rVert_{\mathcal{G}^{s}_2} \right)  \\
     &\qquad\qquad\quad + \lVert \partial_t u  \rVert_{\mathcal{G}^{s-2}_2}\left(   \lVert  v \rVert_{\mathcal{G}^{s}_2} \lVert  h \rVert_{\mathcal{G}^{2}_2} + \lVert u \rVert_{\mathcal{G}^{3}_2} \lVert h \rVert_{\mathcal{G}^{s}_2}\right) \\
     &\lesssim \lVert  \partial_t u \rVert_{\mathcal{G}^{s-2}_2} \left(\kappa ^2e^{-\theta t } + \kappa e^{-\theta t} \lVert  \tilde{h} \rVert_{\mathcal{G}^{s}_2} + \kappa ^2 \lVert  \tilde{v} \rVert_{\mathcal{G}^{s}_2} + \kappa ^2 e^{-\theta t} \lVert  \tilde{h} \rVert_{\mathcal{G}^{s}_2} + \kappa \lVert  \tilde{v} \rVert_{\mathcal{G}^{s}_2} + \kappa e^{-\theta t } \lVert  \tilde{h} \rVert_{\mathcal{G}^{s}_2}   \right) \\
     &\leq \frac{1}{2} \lVert  \partial_tu \rVert_{\mathcal{G}^{s-2}_2}^2  +C \kappa ^3 e^{-\theta t } + C \kappa ^2 e^{-2 \theta t} \lVert  \tilde{h} \rVert_{\mathcal{G}^{s}_2}^2  + C \kappa ^2 \lVert \tilde{v} \rVert_{\mathcal{G}^{s}_2}^2 
\end{align*} 
We have thus obtained  
    \begin{align}\label{eq:Bootstrap for Elow for GWP}
        \begin{split} 
            \lVert  \partial_t (u,\varepsilon v) \rVert_{\mathcal{G}^{s-2}_2}^2 + \frac{\mathrm{d}}{\mathrm{d}t}  \mathrm{E}_{u,v}  +& \dot{\lambda} \lVert  (u,v, \partial_y u , \varepsilon \partial_x u , \varepsilon ^2 \partial_x v) \rVert_{\mathcal{G}^{\sigma -1}_2}^2 \\
            &\qquad  \lesssim  \kappa ^3 e^{-\theta t} + \kappa ^2  e^{-2 \theta t } \lVert  \tilde{h} \rVert_{\mathcal{G}^{s}_2}^2 + \kappa ^2 \lVert \tilde{v} \rVert_{\mathcal{G}^{s}_2}^2 . 
        \end{split}
    \end{align}    
\subsection{Bootstrap on \texorpdfstring{\(\mathrm{E}_h\)}{Elowh}.}
Acting \(\mathrm{A}_{s-1}\) on the equation of \(h\), we have 
\begin{align*}
     \partial_t [ \mathrm{A}_{ s-1} \partial_t  h] + \dot{\lambda}^{\frac{1}{2}} \langle  \partial_x  \rangle^{\frac{1}{2}} \partial_t h + \mathrm{A}_{s-1} \partial_t h  - \mathrm{A}_{s-1} (\varepsilon ^2 \partial_{xx} h + \partial_{yy} h ) + \mathrm{A}_{s-1} ( (u,v) \cdot \nabla h) = 0 . 
\end{align*} 
Taking the inner-product with \(\mathrm{A}_{s-1} \partial_t h\), we have 
\[ \frac{1}{2}\frac{\mathrm{d}}{\mathrm{d}t} \lVert  (\partial_t ,\varepsilon \partial_x, \partial_y ) h \rVert_{\mathcal{G}^{s-1}_2}^2 + \dot{\lambda}  \lVert (\partial_t, \varepsilon \partial_x, \partial_y ) h \rVert_{\mathcal{G}^{\sigma}_2}^2    + \lVert  \partial_t h \rVert_{\mathcal{G}^{s-1}_2}^2  = - \langle  \mathrm{A}_{s-1} ((u,v) \cdot \nabla h) , \mathrm{A}_{s-1} [\partial_t h] \rangle. 
\] 
We directly obtain 
\begin{align*}
    & \lvert \langle  \mathrm{A}_{s-1} ((u,v) \cdot \nabla h) , \mathrm{A}_{s-1} [ \partial_t h ] \rangle \rvert \lesssim  \lVert  \partial_t h \rVert_{\mathcal{G}^{s-1}_2} \left( \lVert  u \rVert_{\mathcal{G}^{s-1}_2} \lVert  \partial_y h \rVert_{\mathcal{G}^{3}_2} + \lVert  u \rVert_{\mathcal{G}^{2}_\infty} \lVert    h  \rVert_{\mathcal{G}^{s }_2}   \right)   \\
    & \qquad\qquad\quad + \lVert  \partial_t h \rVert_{\mathcal{G}^{s-1}_2} \left( \lVert  v \rVert_{\mathcal{G}^{s-1}_\infty} \lVert  \partial_y h \rVert_{\mathcal{G}^{2}_2} +  \lVert  v \rVert_{\mathcal{G}^{2}_\infty} \lVert\partial_y  h \rVert_{\mathcal{G}^{s-1}_2}\right)  \\
     &\lesssim  \lVert  \partial_t h \rVert_{\mathcal{G}^{s-1}_2} \left(\kappa ^2 e^{- \theta t} + \kappa e^{-\theta t} \lVert  \tilde{h} \rVert_{\mathcal{G}^{s}_2} + \kappa ^2 e^{-\theta t}\right)  \leq \frac{1}{2} \lVert  \partial_t h \rVert_{\mathcal{G}^{s-1}_2}^2 + C\kappa ^3 e^{-\theta t} + C \kappa ^2 e^{ - \theta t } \lVert  \tilde{h} \rVert_{\mathcal{G}^{s}_2}^2
\end{align*} 
We conclude     
        \begin{equation}\label{eq:Bootstrap on Elowh for GWP}
            \frac{1}{2}\frac{\mathrm{d}}{\mathrm{d}t} \mathrm{E}_h + \dot{\lambda} \lVert  (\partial_t, \varepsilon \partial_x, \partial_y) h \rVert_{\mathcal{G}^{\sigma}_2}^2 + \lVert  \partial_t h \rVert_{\mathcal{G}^{ s-1}_2}^2 \lesssim  \kappa ^3 e^{-\theta t } + \kappa e^{-\theta t } \lVert  \tilde{h} \rVert_{\mathcal{G}^{s}_2}^2.  
        \end{equation}  
\subsection{Conclusion}
Combining \eqref{eq:Bootstrap on Eh for GWP}, \eqref{eq:Bootstrap on Eu for GWP}, \eqref{eq:Bootstrap for Elow for GWP} and \eqref{eq:Bootstrap on Elowh for GWP}, we obtain 
\begin{align*}
    \frac{\mathrm{d}}{\mathrm{d}t} \left(  \mathrm{E}_{u,v} + \mathrm{E}_h  + \mathcal{E}_h + \mathcal{E}_{u,v} \right)& + \mathcal{CK}_h +  \dot{\lambda} \lVert  (\tilde{u}, \varepsilon \tilde{v}) \rVert_{\mathcal{G}^{s + \frac{1}{4}}_2}^2 + \lVert  (\tilde{u}, \tilde{v}) \rVert_{\mathcal{G}^{s}_2}^2 + \lVert \partial_y \tilde{u} \rVert_{\mathcal{G}^{s}_2}^2  \\
    & \lesssim  \kappa ^3 e^{-\theta t /2} +\kappa e^{- \theta t/2} \lVert  \tilde{u} \rVert_{\mathcal{G}^{s + \frac{1}{4}}_2}^2  +\kappa e^{-\theta t / 2} \lVert  \tilde{h} \rVert_{\mathcal{G}^{s}_2}^2 + \kappa ^2 \lVert  \tilde{v} \rVert_{\mathcal{G}^{s}_2}^2
\end{align*} 
As we are requiring \(\dot{\lambda}^3 \land \dot{\lambda}  \geq 16 C \kappa e^{-\theta t /2 }\) and \(\kappa \ll  1\) so we obtain 
\begin{equation}\label{eq:Bootstrap Inequality for GWP}
    \frac{\mathrm{d}}{\mathrm{d}t} \left(  \mathrm{E}_{u,v} +\mathrm{E}_h + \mathcal{E}_h + \mathcal{E}_{u,v}\right)  + \mathcal{CK}_h + \frac{\dot{\lambda}}{2} \lVert  (\tilde{u}, \varepsilon \tilde{v}) \rVert_{\mathcal{G}^{s + \frac{1}{4}}_2}^2 + \frac{1}{2} \lVert  (\tilde{u} , \tilde{v}) \rVert_{\mathcal{G}^{s}_2}^2  \lesssim  \kappa ^3 e^{-\theta t / 2} . 
\end{equation}  
For the construction of \(\lambda\), recall that we required 
\[  \ddot{\lambda} \leq 0 \leq \lambda \leq \frac{\delta _0}{2} , \quad \dot{\lambda}^3 \land \dot{\lambda} \geq 16 C \kappa e^{-\theta t / 2} , \quad \dot{\lambda} \in L^1.  
\] 
As \(C \kappa \ll  1\), it is sufficient to relax the second condition to \( \dot{\lambda} \geq 16 e^{-\theta t / 6}\). We choose 
\begin{equation}\label{eq:Definition of λ̇ for GWP}
    \dot{\lambda} (t)= 16 e^{- \frac{\theta t}{3\delta _0}}    \quad \text{for }t \geq 0 . 
\end{equation}   

To conclude this bootstrap argument, let 
\begin{equation}\label{eq:Definition of Thash for GWP}
    T^\# :=\sup \,   \{t \geq 0 : \mathcal{E}_{u,v} + \mathcal{E}_h + \mathrm{E}_{u,v} + \mathcal{E}_h \leq 3C \kappa ^2 \} 
\end{equation}  
Integrating \eqref{eq:Bootstrap Inequality for GWP} implies that for any \(t \in [0,T^\star] : \)  
\begin{align*}
    \lVert (u,v) \rVert_{\mathcal{G}^{s - 1}_2}^2 + \mathrm{E}_{u,v} + \mathrm{E}_h + \mathcal{E}_h + \mathcal{E}_{u,v} &\lesssim \lVert (\partial_y u_{\mathrm{in}}, v_{\mathrm{in}} , \partial_t h_{\mathrm{in}} , \partial_y h_{\mathrm{in}} ) \rVert_{\mathcal{G}^{s}_2}^2  + \frac{2\kappa ^3}{\theta}(1 - e^{ - T^\# \theta / 2}) \\
    &\lesssim \lVert (\partial_y u_{\mathrm{in}}, \partial_x u_{\mathrm{in}}, \partial_y e_{\mathrm{in}} , \partial_y f_{\mathrm{in}} , \partial_y h_{\mathrm{in}} ) \rVert_{\mathcal{G}^{s}_2}^2  + \frac{2\kappa ^3}{\theta} \\
    &\leq  C \kappa ^2 + \frac{2C}{\theta}\kappa ^3,
\end{align*} 
where we have used the equation for \( h\) and Proposition \ref{prop:Elliptic Estimates} to control \( \lVert  \partial_t h_{\mathrm{in}} \rVert_{\mathcal{G}^{s}_2}^2\).   
Letting \(\kappa _0 \leq \frac{\theta}{2}\), we obtain for \(t \in [0,T^\#] : \) 
\begin{equation}\label{eq:Bootstrapped step for GWP}
    \mathcal{E}_{u,v} + \mathcal{E}_h + \mathrm{E}_{u,v} + \mathcal{E}_h \leq 2C \kappa ^2 . 
\end{equation} 
Thus, using a continuous argument, \eqref{eq:Definition of Thash for GWP} and \eqref{eq:Bootstrapped step for GWP} imply that \(T^\#  = \infty\) giving us global well-posedness of the system.

\section{Ill-posedness in Gevrey-\texorpdfstring{\(2+\)}{2+} class}
 In this section, we will show some evidence for the ill-posedness of the system \eqref{eq:System for (u,h)} in \( \mathcal{G}^{2+}\) by showing the existence of solutions that grows like \( e^{\sqrt{k}t}\) for frequency \(k\) in \(x\) for the linearised system for the perturbed magnetic field \(h\). To that end, recall the linearised system for \(h\)
 \begin{equation}
     \begin{cases}  
        \partial_{tt} h + \partial_t h - \partial_{yy} h +y(1 - y) \partial_x h  = 0 ,   \\
        h|_{y = 0,1}  = \partial_t h |_{y = 0,1}= 0 , \\ h|_{t = 0 } = \zeta, \quad \partial_t h|_{ t =0} = \zeta ^1  
     \end{cases}  
 \end{equation}  
 where \( \zeta, \zeta ^1\) are small initial data to be prescribed. For some fixed parameter \(k < 0\) which corresponds to the frequency in \(x\), we split \( h := h^1_k + h^2_k\) where \(h^{1}_k\) solves the nonhomogeneous equation for \(h\) with initial data and \(h^2_k\) solves the inhomogeneous equation for \(h\) with zero initial data where the forcing comes from the boundary correction. To that end, we will first obtain solutions for the homogeneous equation without any initial data: 
 \[  \partial_{tt} \mathcal{Q} + \partial_t \mathcal{Q} - \partial_{yy} \mathcal{Q} + y (1 - y) \partial_x \mathcal{Q} = 0 
 \] 
Taking the Laplace transform in \(t\) and Fourier transform in \(x\), we have  
 \begin{equation}\label{eq:ill-posedness:Equation for Q}
    (c^2 + c + ik y (1 - y) - \partial_{yy}) \hat{\mathcal{Q}} = 0 . 
 \end{equation}  
 Then this is an eigenvalue problem for the complex harmonic oscillator \cite{helffer2013spectral} under a linear change of coordinate (as \(k < 0\)) so we have the precise formulas:  
 \begin{equation}\label{eq:ill-posedness:Equation for c}
    c = - \frac{1}{2}  + \sqrt{1  + 4 i  \lvert k \rvert - 4 \sqrt{ \lvert k \rvert} e^{i \frac{\pi}{4}}} ,\quad \mathcal{\hat{Q}}  = f_k(y) , \quad f_k (y) := m_0 e^{ - \frac{\sqrt{i  \lvert k \rvert}}{2}(y - \frac{1}{2})^2},
 \end{equation} 
 where \(m_0\) is some normalisation constant and we have chosen the \(c\) with positive real part (provided \( \lvert k \rvert\) is large). To account for the boundary data, noting the symmetry at \(y = \frac{1}{2}\), we will set 
 \[  h^1_k := \mathbb{R}{\rm e\,}  \lvert k  \rvert ^{-\frac{5}{8}} e^{i(kx + c_i t ) + c_r t } (f_k(y) - f_k(0)) , \quad \zeta := h^1_k(0,x,y)  , \quad \zeta ^1 := \partial_t h^1_k(0,x,y) . 
 \] 
 where \(c_r  := \mathbb{R}{\rm e\,} c\) and \( c_i : = \mathrm{Im\,} c\). 

 Thus initial data is of constant size:   
 \begin{align*}
    \lVert (\zeta, \zeta ^1 )\rVert_{L^2 H^1_y}^2& \lesssim 1 +    \lvert k \rvert^{ - \frac{5}{4}}   \lVert c e ^{ - \frac{\sqrt{ \lvert k \rvert}}{2 \sqrt{2}}(y  - \frac{1}{2})^2} \rVert_{H^1_y}^2 \lesssim 1 +  \lvert k \rvert^{\frac{3}{4}} \int_{0}^{1} (y - \tfrac{1}{2})^2 e^{-\frac{\sqrt{k}}{\sqrt{2}} ( y -     \frac{1}{2})^2} \mathrm{\,d} y   = O(1). 
 \end{align*} 
 Here recall that \( a_k , b_k\) in Theorem \ref{thm:Results:ill-posedness for the Linear system for h} is given simply by the expansion of \(h^1_k(0,x,y)\) in \(\cos kx\) and \(\sin kx\).  
 We can now obtain the equation for \(h^2_k = h - h^1_k : \)
 \begin{equation}\label{eq:ill-posedness:System for h2}
     \begin{cases} \partial_{tt} h_k^2 + \partial_t h_k^2 - \partial_{yy} h_k^2 + y(1-y) \partial_x h_k^2 = \mathscr{F}, \\
    h_k^2|_{ t = 0 } =  h_k^2|_{y = 0,1}  = \partial _t h ^2_k|_{t = 0,1} = 0.  \end{cases}  
 \end{equation} 
 where \( \mathscr{F} = \mathbb{R}{\rm e\,}  \lvert k \rvert^{-\frac{5}{8}} (c^2 + c + ik y(1 - y)) e^{ct + ikx} f_k(0).\)
 As the forcing is small in Gevrey-2 for small time:
 \[   \lvert \mathscr{F} \rvert \lesssim \lvert k \rvert^{\frac{3}{8}} e^{c_r t  - \frac{\sqrt{ \lvert k \rvert}}{8 \sqrt{ 2}}} \lesssim  e^{c_r t - \theta _1 \sqrt{  \lvert k \rvert}}
 \]   
we expect that \(h^2_k\) remains Gevrey-2 locally for \(\theta _1  < \frac{1}{16 \sqrt{ 2}}. \) This is precised by the following proposition: 
 \begin{proposition}\label{prop:ill-posedness:Global bound on h2}
    { There exists a universal constant \(\delta\) such that for any \(\kappa > 0\) and any \(k \in \mathbb{Z}   \)}, there exists a \(T_0\) independent of \(k\) such that  
    \[    \lVert  e^{\frac{\delta }{2}  \sqrt{ \lvert k \rvert}} h^2_k \rVert_{L^2}  \lesssim  \kappa, \quad \forall  t \in [0,T_0).
    \]  
 \end{proposition} 
 \begin{proof}[Proof of Proposition \ref{prop:ill-posedness:Global bound on h2}] 
    Fix \(k \in \mathbb{Z}\) and let \[   \tilde{h} :=  h^2_k, \quad \mathrm{A}_{s} := e^{(\delta - \lambda(t))\sqrt{ \lvert k \rvert} }  \lvert k \rvert^s. \] Then we have   
    \begin{align*}
        \mathrm{A}_0 (\tilde{g}) &=   \mathrm{A}_{0} [\partial_{tt} \tilde{h}] + \mathrm{A}_0 [\partial_t \tilde{h}]   - \partial_{yy} \mathrm{A}_0 \tilde{h} + y(1 - y) \mathrm{A}_0 [\partial_x \tilde{h}], \\
        &= \partial_t (\mathrm{A}_0 [\partial_t \tilde{h}]) + \dot{\lambda}  \lvert k \rvert^{\frac{1}{2}} \mathrm{A}_0 [ \partial_t \tilde{h}] +   \mathrm{A}_0 [ \partial_t \tilde{h}] + \mathrm{A}_0  \tilde{h} - \partial_{yy} \mathrm{A}_0 \tilde{h}+ y(1 - y) \mathrm{A}_0 [ \partial_x \tilde{h}]. 
    \end{align*}
    Inner-producting with \( \mathrm{A}_2 [ \partial_t \tilde{h}]\), we have 
    \begin{align*}
        \frac{\mathrm{d}}{\mathrm{d}t} \lVert  (\partial_t,   \partial_y  ) \tilde{h}  \rVert_{\mathcal{G}^{0}_2}^2 + \dot{\lambda} \lVert  \partial_t \tilde{h} \rVert_{\mathcal{G}^{ \frac{1}{4}}_2}^2& +   \lVert  \partial_t \tilde{h} \rVert_{\mathcal{G}^{0}_2}^2 = -    \langle  y(1 - y) \mathrm{A}_0 [ \partial_x \tilde{h}] , \mathrm{A}_0 [ \partial_t \tilde{h}] \rangle  +    \langle  \mathrm{A}_0 \mathscr{F}, \mathrm{A}_0 [\partial_t \tilde{h}] \rangle. 
    \end{align*} 
    Here, we write 
    \[  \lVert  \phi \rVert_{\mathcal{G}^{s}_2} := \lVert  e^{(\delta  - \lambda(t))  \sqrt{ \lvert k \rvert}}  \lvert k \rvert^s \phi(t,x,y) \rVert_{L^2_y}.  
    \] 
    Using the usual expansion, we have 
    \[  \dot{\lambda} \lVert  \partial_t \tilde{h} \rVert_{\mathcal{G}^{  \frac{1}{4}}}^2 = \dot{\lambda} \lVert  \partial_t \mathrm{A}_{2 + \frac{1}{4}} \tilde{h} \rVert_{L^2}^2 + \frac{\mathrm{d}}{\mathrm{d}t} \left(\dot{\lambda}^2 \lVert  \tilde{h} \rVert_{\mathcal{G}^{  \frac{1}{2}}_2}^2\right)  - 2 \ddot{\lambda} \dot{\lambda} \lVert  \tilde{h} \rVert_{\mathcal{G}^{ \frac{1}{2}}_2}^2 + \dot{\lambda}^3 \lVert  \tilde{h} \rVert_{\mathcal{G}^{ \frac{3}{4}}_2}^2.  
    \] 
    This gives 
    \begin{align*}
        \frac{\mathrm{d}}{\mathrm{d}t} \lVert  (\partial_t , \partial_y ) \tilde{h}  \rVert_{\mathcal{G}^0_2}^2 +\frac{1}{2} \frac{\mathrm{d}}{\mathrm{d}t} \lVert  \dot{\lambda} \tilde{h} \rVert_{\mathcal{G}^{\frac{1}{2}}_2}^2  &+ \frac{\dot{\lambda} }{2}\lVert  \partial_t \tilde{h} \rVert_{\mathcal{G}^{ \frac{1}{4}}_2}^2 +  \lVert  \partial_t \tilde{h} \rVert_{\mathcal{G}^0_2}^2 + \frac{\dot{\lambda}^3}{2} \lVert \tilde{h} \rVert_{\mathcal{G}^{ \frac{3}{4}}_2}^2 + \frac{\dot{\lambda}}{2} \lVert  \partial_t [\mathrm{A}_{ \frac{1}{4}} \tilde{h}] \rVert_{L^2}^2   \\
        &= 2 \ddot{\lambda}\dot{\lambda} \lVert  \tilde{h} \rVert_{\mathcal{G}^{ \frac{1}{2}}_2}^2 -    \langle  y(1 - y) \mathrm{A}_0 [ \partial_x \tilde{h}] , \mathrm{A}_0 [ \partial_t \tilde{h}] \rangle  +    \langle  \mathrm{A}_0 \tilde{g} , \mathrm{A}_0 [\partial_t \tilde{h}] \rangle. 
    \end{align*} 
    Requiring \(\ddot{\lambda} \leq 0 \leq \dot{\lambda}\), it suffices to just bound the latter two terms on right. For them, we have 
   \begin{align*}
         \lvert  \langle  y(1 - y) \mathrm{A}_0 [ \partial_x \tilde{h}] , \mathrm{A}_0 [ \partial_t \tilde{h}]  \rangle \rvert &\leq  \lvert  \langle  \mathrm{A}_{-\frac{1}{4}} [ \partial_x \tilde{h}] , \mathrm{A}_{\frac{1}{4}} [ \partial_t \tilde{h}] \rangle \rvert  \lesssim { \lVert   \tilde{h} \rVert_{\mathcal{G}^{\frac{3}{4}}_2}^2 }  +   \lVert  \partial_t \tilde{h} \rVert_{\mathcal{G}^{\frac{1}{4}}_2}^2, 
    \end{align*}   
    and for the other term \(g\), we have 
    \begin{align*}
         \lvert  \langle \mathrm{A}_0\mathscr{F}, \mathrm{A}_0 [\partial_t \tilde{h}]  \rangle \rvert &\leq  \lvert k  \rvert\lVert  \mathscr{F}\rVert_{\mathcal{G}^{0}_\infty} \lVert  \partial_t \tilde{h} \rVert_{\mathcal{G}^{0}_2} \lesssim   e^{ c_r  t  + (\delta- \lambda(t) - \frac{\theta_1}{2}) \sqrt{k }}    \lVert  \partial_t \tilde{h} \rVert_{\mathcal{G}^{0}_2}. 
    \end{align*} 
    For \(\dot{\lambda} \geq 32 C  \), we get 
    \[  \frac{\mathrm{d}}{\mathrm{d}t} \lVert  \partial_t \tilde{h} \rVert_{\mathcal{G}^{0}_2}^2 + \frac{\dot{\lambda}}{32}\lVert  \partial_t \tilde{h} \rVert_{\mathcal{G}^{\frac{1}{4}}_2} ^2 + \dot{\lambda}^3 \lVert  \tilde{h} \rVert_{\mathcal{G}^{\frac{3}{4}}_2}^2 \lesssim   e^{2 c_r  t  + (\delta- \lambda(t) - \theta_1) \sqrt{k }}    \lVert  \partial_t \tilde{h} \rVert_{\mathcal{G}^{0}_2}. 
    \]  
    {Fixing 
        \[  \delta = \frac{\theta_1}{4}, \quad \dot{\lambda} = 32 C , \quad  T_\# \text{ small enough such that }  \frac{2 c_r }{\sqrt{k}} T_\#  - \dot{\lambda} T _\#  \leq (2 \sqrt{2}  - 32 C ) T_\#  =   \frac{\theta_1}{8}. 
        \] 
        where we have used 
        \[  c_r \leq \sqrt{2 \sqrt{ 2k}} \, \mathbb{R}{\rm e\,} \sqrt{i \sqrt{ 2  \lvert k \rvert }}    = \sqrt{2  \lvert k \rvert}. 
        \] } 
    Then for any \(t \leq  T_\#: \) 
    \[  \frac{\mathrm{d}}{\mathrm{d}t} \lVert  ( \partial_t, \partial_y) h^2_k  \rVert_{\mathcal{G}^{0}_2}^2 + \frac{\dot{\lambda}}{32} \lVert  \partial_t h^2_k \rVert_{\mathcal{G}^{\frac{1}{4}}_2}^2 + \dot{\lambda}^3 \lVert  h^2_k \rVert_{\mathcal{G}^{\frac{3}{4}}_2}^2  \lesssim  e^{\sqrt{k} ( \frac{\theta_1}{8} + \frac{\theta_1}{4} -  \frac{\theta_1}{2}) } \lVert  \partial _t h^2_k \rVert_{\mathcal{G}^{0}_2}^2 \lesssim e^{-\frac{\theta_1}{8} \sqrt{k}} \lVert  \partial_t h^2_k \rVert_{\mathcal{G}^{0}_2}^2. 
    \]  
    So via a continuity argument, for any \(\kappa > 0\) there exists \( T_0 \leq T_\# \) such that for all \( t \leq T_0 : \)
    \[ \lVert  (\partial_t , 1) h^2_k \rVert_{\mathcal{G}^{0}_2}^2 \leq    \lVert  (\partial_t, \partial _y) h^2_k \rVert_{\mathcal{G}^{0}_2} \lesssim  \kappa ^2. \qedhere
    \] 
 \end{proof}  
 Now let us move onto proving the growth bounds for \(h^1_k:\)
\begin{proposition}\label{prop:ill-posedness:Growth of h1k}
    There exists \(M > 0\) such that for any \( \lvert k \rvert \geq M \) there exists some constant \( m_0 > 0 \)
    \[ \lVert  h^1_k \rVert_{H^1_y} \geq m_0  \lvert k \rvert^{-\frac{1}{2}} e^{  \sqrt{k} t}. 
    \] 
\end{proposition}
    
     \begin{proof} 
        To obtain the lower bound, we need to make sure the real part doesn't cancel out due to multiplication of different factor, we will make a very precise calculation; to that end, 
        \begin{align*}
            \mathbb{R}{\rm e\,} e^{c_r t + i(kx+ c_i t)} f_k(y)&=m_0 \mathbb{R}{\rm e\,} e^{c_r t + i (kx + c_i t) - \frac{\sqrt{ik}}{2}(y - \frac{1}{2})^2} \\
            &= m_0 e^{c_r t  - \frac{\sqrt{k}}{2 \sqrt{ 2}}(y - \frac{1}{2})^2} \cos  \left( kx + c_i t  - \tfrac{\sqrt{k}}{2 \sqrt{ 2}}(y - \tfrac{1}{2})^2\right) \\
            \mathbb{R}{\rm e\,} e^{ct + ikx} f_k(0 ) & = m_0 e^{c_r t  - \frac{\sqrt{k}}{8 \sqrt{ 2}} } \cos  \left( kx  + c_i t    - \tfrac{\sqrt{k}}{8 \sqrt{ 2}}\right). 
        \end{align*} 
        This implies for \(\Theta :=c_i t - \tfrac{\sqrt{k}}{8 \sqrt{ 2}} \) and \(\Theta _y := c_i t - \tfrac{\sqrt{k}}{2 \sqrt{2}}(y - \tfrac{1}{2})^2\), we have 
        \begin{align*}
           \lvert k \rvert^{\frac{5}{4}}  \lVert  h^1_k \rVert_{L^2_{x,y}}^2 &= 2m_0^2 e^{2 c_r t} \int_{0}^{\frac{1}{2}}\int_{\mathbb{T}} e^{- \sqrt{\frac{k}{2}}(y - \frac{1}{2})^2} \cos ^2(kx + \Theta _y)  + e^{ - \frac{\sqrt{k}}{4 \sqrt{ 2}}}  \cos ^2 (kx + \Theta)  \mathrm{\,d}x \mathrm{\,d} y \\
             & \qquad - 4 m_0^2  e^{2 c_r t } \int_{0}^{\frac{1}{2}}\int_{\mathbb{T}}  e^{- \frac{\sqrt{k}}{8 \sqrt{ 2}} - \frac{\sqrt{k}}{2 \sqrt{2}} ( y - \frac{1}{2})^2 } \cos (kx + \Theta _y)\cos (kx + \Theta ) \mathrm{\,d} x \mathrm{\,d} y , \\
             &\geq 2 m_0^2 e^{2 c_r t} \pi \int_{0}^{\frac{1}{ \lvert k \rvert^{\frac{1}{4}}}} \left( e^{ -\frac{\sqrt{k}}{2 \sqrt{ 2}} \xi^2} - e^{ - \frac{\sqrt{k}}{8 \sqrt{ 2}}} \right)^2 \mathrm{\,d} \xi  \geq m_0^2  \lvert k \rvert^{-\frac{1}{4}} e^{2c_r t }. 
        \end{align*}  
        Similarly, we can also compute the norm for the derivative: We have 
        \begin{align*}
              &\partial_y (\mathbb{R}{\rm e\,} e^{ct + ikx} f_k(y) - f_k(0))  = m_0 \sqrt{ \tfrac{k}{2}}( \tfrac{1}{2} - y) \left( \mathbb{R}{\rm e\,} e^{ct + ikx + \frac{\sqrt{ik}}{2}(y - \frac{1}{2})^2} -  \mathrm{Im\,} e^{ct + ikx + \frac{\sqrt{ik}}{2}(y - \frac{1}{2})^2} \right) \\
              &= m_0 \sqrt{\tfrac{k}{2}}( \tfrac{1}{2} - y)  e^{2c_r t - \frac{\sqrt{k}}{2 \sqrt{ 2}}(y - \frac{1}{2})^2} \left( \cos  (kx  + \Theta _y) -\sin  (kx  + \Theta _y) \right). 
        \end{align*} 
        which implies 
        \begin{align*}
            \lvert k \rvert^{\frac{5}{4}}  \lVert  \partial_y h^1_k \rVert_{L^2_{x,y}}^2 &= m_0^2  \lvert k \rvert e^{2c_r t}\int_{0}^{\frac{1}{2}} (\tfrac{1}{2} - y)^2 e^{- \sqrt{\frac{k}{2}}(y - \frac{1}{2})^2}   \int_{\mathbb{T}} \left( \cos  (kx  + \Theta _y) -\sin  (kx  + \Theta _y) \right)^2 \mathrm{\,d} x  \mathrm{\,d} y \\
             &=2 \pi  m_0^2  \lvert k \rvert e^{2 c_r t } \int_{0}^{\frac{1}{2  \lvert k \rvert^{\frac{1}{4}}}} (y  - \tfrac{1}{2})^2  e^{ - \frac{\sqrt{k}}{2 \sqrt{ 2}}(y - \frac{1}{2})^2} \mathrm{\,d} y \geq C^2  \lvert k \rvert e^{2 c_r t}  \lvert k \rvert^{ - \frac{3}{4}} = C^2  \lvert k \rvert^{\frac{1}{4}} e^{2 c_r t}. 
        \end{align*}  
        Combining the above two estimates, we obtain a growth result for \(h^1_k: \)
        \[  \lVert  h^1_k \rVert_{L^2_x H^1_y}\geq  m_0  \lvert k \rvert^{-\frac{1}{2}}  e^{c_r  t} \geq  m_0  \lvert k \rvert^{-\frac{1}{2}}  e^{\sqrt{k}  t}
        \] 
        where we have used  \begin{equation}\label{eq:ill-posedness:Lower bound for cr}
            c_r \geq -\sqrt{2 \sqrt{  \lvert k \rvert}}\,  \mathrm{Im\,} \sqrt{1 + i (1 - \sqrt{ 2  \lvert k \rvert })} \geq -\sqrt{ 2 \sqrt{  \lvert k \rvert}}\,  \mathrm{Im\,} \sqrt{ -i \sqrt{ 2  \lvert k \rvert}}   \geq \sqrt{  \lvert k \rvert }. 
        \end{equation} 
        as \(c\) is given by the formula \eqref{eq:ill-posedness:Equation for c}
      \end{proof}   
      Then Proposition \ref{prop:ill-posedness:Global bound on h2} and Proposition \ref{prop:ill-posedness:Growth of h1k} immediately implies Theorem \ref{thm:Results:ill-posedness for the Linear system for h}. 
      \begin{proof}[Proof of Theorem \ref{thm:Results:ill-posedness for the Linear system for h}]
        Fix \( s < \frac{1}{2}\) and \( C_0 > 0\). Then using \eqref{eq:ill-posedness:Lower bound for cr}, we have for \(t \geq T_k := 2  C_0  \lvert k \rvert^{s - \frac{1}{2}} ,  \) 
        \[  \lVert  e^{ -2 C_0  \lvert k \rvert^s} h^1_k \rVert_{L^2_x H^1_y}(t)  \geq m_0   \lvert k \rvert^{-\frac{1}{2}} e^{\sqrt{ \lvert k \rvert}T_k -2 C_0  \lvert k \rvert^s} \geq m_0  \lvert k \rvert^{-\frac{1}{2}} .  
        \]  
        As \( h_k = h^1_k + h^2_k\), where \( \lVert  e^{\frac{\delta}{2} \sqrt{k}} h^2_k \rVert_{L^2_x H^1_y} \lesssim  \kappa \) due to Proposition \mbox{\ref{prop:ill-posedness:Global bound on h2}}, we have for \( C_0 > 0 , \)  
        \[  \lVert  h_k \rVert_{L^2_x H^1_y} \geq \lVert  h^1_k  \rVert_{L^2_x H^1_y} - \lVert  h^2_k  \rVert_{L^2_x H^1_y} \geq m_0 \lvert k \rvert^{ - \frac{1}{2}}  e^{2 C_0  \lvert k \rvert^s}   - C \kappa e^{ - \frac{\delta}{2} \sqrt{  \lvert k \rvert}} \geq \frac{ 1}{2} C_0^{\frac{1}{2} - s} m_0  e^{C_0  \lvert k \rvert^s} 
        \] 
        provided \(\kappa \ll \frac{m_0}{C}\) and \( t \in [T_k, T_0)\).
      \end{proof}

\section{Appendix: Nonlinear estimates and Gevrey spaces}



We present some of the inequalities which we have routinely used in the above calculations. 

We fix \( X := \mathbb{R} \times  [0,1]\) and \( \mathrm{A}_s := e^{\delta \sqrt{ \langle \partial_x \rangle}} \langle \partial_x \rangle^s\) for some fixed \(\delta > 0 \) and \(s\) to be specified later. In the proofs, we will routinely omit absolute values in the integrands for ease of readability. 

\begin{proposition}\label{prop:Elliptic Estimates}
    Let \(\phi , \psi \in H^1(X)\) satisfying
    \[  \begin{cases} \partial_x \phi + \partial_y \psi  = 0 , \qquad  \partial_y \phi - \partial_x \psi = \omega, \\ (\partial_y\phi,  \psi)|_{y = 0,1} = 0, \end{cases}   
    \] 
    where \(\omega \in \mathcal{G}^{s}_2(X) \). Then we have 
    \[   \lVert  (\phi , \psi) \rVert_{\mathcal{G}^{s}_2} \lesssim  \lVert  \omega \rVert_{\mathcal{G}^{s-1}_2}, \quad \text{and} \quad \lVert  \partial_y (\phi, \psi) \rVert_{\mathcal{G}^{s}_2} \lesssim  \lVert  \omega \rVert_{\mathcal{G}^{s}_2}. 
    \] 
\end{proposition} 
\begin{proof} 
    Define a potential function 
    \[  F(x,y) := \int_{0}^y \phi (x, y') \mathrm{\,d} y. 
    \] 
    Then \(F\) satisfies the following properties: 
    \[  \nabla^\perp F = (\phi, \psi ) , \quad \Delta F  =\omega , \quad F(x,0) =\partial_{yy} F|_{y = 0,1} = \partial_x F|_{y = 0,1}   = 0.  
    \] 
    We now prove \(\lVert  ( \lvert \partial_x \rvert, \partial_y) \nabla^\perp F \rVert_{L^2} \lesssim  \lVert  \Delta F \rVert_{L^2} \) as the same strategy will follow when \(L^2\) is replaced by \( \mathcal{G}^{s}_2\). To that end, we have 
    \begin{align*}
        \lVert  ( \lvert \xi \rvert , y) \widehat{\nabla^\perp F} \rVert_{L^2} &= \int  \lvert \xi ^2 \hat{F} \rvert^2 + 2  \lvert  \xi \partial_y \hat{F} \rvert^2 +  \lvert \partial_{yy} \hat{F} \rvert^2 \lesssim   \lvert \xi \rvert^2 \langle  \partial_y \hat{F} , \partial_y \hat{F} \rangle + \lVert  \Delta F \rVert_{L^2}^2 \\
        &\lesssim   \lvert \xi \rvert^2 \langle  \partial_{yy} \hat{F} , \hat{F} \rangle + \lVert  \Delta F \rVert_{L^2}^2 \lesssim  \lVert  \Delta F \rVert_{L^2}^2. 
    \end{align*} 
    Finally, we conclude 
    \begin{align*}
        \lVert  (\phi , \psi) \rVert_{\mathcal{G}^{s}_2} &= \lVert  \langle  \partial_x \rangle (\phi , \psi) \rVert_{\mathcal{G}^{s-1}_2} \lesssim  \lVert  (\phi, \psi) \rVert_{\mathcal{G}^{s-1}_2}  + \lVert   \lvert  \partial_x  \rvert ( \phi , \psi) \rVert_{\mathcal{G}^{s-1}_2} \lesssim  \lVert ( \lvert \partial_x \rvert , \partial_y) (\phi, \psi) \rVert_{\mathcal{G}^{s-1}_2} \\
        &\lesssim  \lVert  ( \lvert \partial_x \rvert, \partial_y) \nabla^\perp  F \rVert_{\mathcal{G}^{s-1}_2} \lesssim  \lVert  \Delta F \rVert_{\mathcal{G}^{s-1}_2}  = \lVert  \omega \rVert_{\mathcal{G}^{s-1}_2} . \qedhere
    \end{align*} 
\end{proof}   
\begin{proposition}\label{prop:Gevrey-Holder}
    It holds that for any \(2 \leq p,q \leq \infty,  1 \leq \mu \leq s , 0 \leq \delta :  \)
    \[  \lVert fg \rVert_{\mathcal{G}_2^{\delta;s }}  \lesssim  \lVert  f \rVert_{\mathcal{G}_p^{ \delta;s  }} \lVert  g \rVert_{\mathcal{G}_{p'}^{  \delta ; \mu}} + \lVert g\rVert_{\mathcal{G}_q^{   \delta;s }} \lVert  f \rVert_{\mathcal{G}_{q'}^{    \delta ; \mu }}
    \]  
    where \(f \in \mathcal{G}^{ \delta;s }_{p} (X) \cap \mathcal{G}^{  \delta ; \mu}_{q'}(X), g \in \mathcal{G}^{\delta ;\mu  }_{p'} (X) \cap \mathcal{G}^{  \delta ; s}_{q}(X)\) and \(\frac{1}{p} + \frac{1}{p'} = \frac{1}{2} = \frac{1}{q} + \frac{1}{q'}\).
\end{proposition}
\begin{proof} 
    The lemma follows directly from the fact that 
    \[  e^{\delta \langle \xi \rangle^{\frac{1}{2}}} \langle  \xi \rangle^s \lesssim e^{\delta \langle \eta \rangle^{\frac{1}{2}}}   e^{\delta\langle \eta - \xi \rangle^{\frac{1}{2}}} (\langle \eta \rangle^s + \langle  \xi - \eta \rangle^s)
    \] 
    and Holder's inequality. Indeed 
    \begin{align*}
       \lVert  fg \rVert_{\mathcal{G}_2^{  \delta ;s}} & = \lVert  e^{\delta\langle \xi \rangle^\frac{1}{2}} \langle \xi \rangle^s \widehat{fg} \rVert_{L^2 L^2} \lesssim \lVert  \int_{\mathbb{R}} \hat{f}(\eta,y) \hat{g}(\xi - \eta, y) e^{\delta \langle \xi \rangle^\frac{1}{2}} \langle \xi \rangle^s  \rVert_{L^2 L^2} \\
        &\lesssim  \lVert  e^{\delta \langle\,\cdot\, \rangle^\frac{1}{2}} \langle\,\cdot\, \rangle^s \hat{f} \star_\xi e^{\delta \langle\,\cdot\,\rangle^\frac{1}{2}}   \hat{g}   \rVert_{L^2 L^2} +
         \lVert  e^{\delta\langle\,\cdot\, \rangle^\frac{1}{2}} \hat{f} \star_\xi e^{\delta \langle\,\cdot\,\rangle^\frac{1}{2}} \langle\,\cdot\, \rangle^s \hat{g}   \rVert_{L^2 L^2}    \\
         &\lesssim \left\lVert  \lVert  e^{\delta \langle\,\cdot\, \rangle^\frac{1}{2}} \langle\,\cdot\, \rangle^s \hat{f} \rVert_{L^2} \lVert  e^{\delta \langle\,\cdot\, \rangle^\frac{1}{2}} \langle\,\cdot\, \rangle^1 \hat{g} \rVert_{L^2} \right\rVert_{L^2}  +  \left\lVert  \lVert  e^{\delta \langle\,\cdot\, \rangle^\frac{1}{2}} \langle\,\cdot\, \rangle^1 \hat{f} \rVert_{L^2} \lVert  e^{\delta \langle\,\cdot\, \rangle^\frac{1}{2}} \langle\,\cdot\, \rangle^s \hat{g} \rVert_{L^2} \right\rVert_{L^2}  \\
         &\lesssim \lVert  f \rVert_{\mathcal{G}_p^{  \delta ;  s  }} \lVert  g \rVert_{\mathcal{G}_{p'}^{  \delta ; 1  }} + \lVert g\rVert_{\mathcal{G}_q^{  \delta ; s   }} \lVert  f \rVert_{\mathcal{G}_{q'}^{  \delta ;  1 }},
    \end{align*}  
    where \( (f \star _\xi g)(\xi, y) := \int_{\mathbb{R}} f(\xi - \eta, y) g(\eta, y ) \mathrm{\,d} \eta.  \)
\end{proof} 
\begin{corollary}\label{cor:Gevrey-Holder with δ}
    For \( 2  \leq p,q \leq \infty \) and \(0 \leq \mu \leq 1 \leq s\), we have 
    \[   \lvert \langle  \mathrm{A}_s(f \partial_i g) , \mathrm{A}_s h \rangle \rvert  \lesssim \lVert  h \rVert_{\mathcal{G}^{s+ \mu}_2} \left( \lVert  f \rVert_{\mathcal{G}^{s - \mu}_p} \lVert \partial_i g   \rVert_{\mathcal{G}^{2}_{p'}} + \lVert  f \rVert_{\mathcal{G}^{2}_{q}} \lVert  \partial_i g \rVert_{\mathcal{G}^{s - \mu}_{q'}}\right) 
    \] 
    where \(i \in \{x,y\}\) and \( \frac{1}{p} + \frac{1}{p'} = \frac{1}{q} + \frac{1}{q'} = \frac{1}{2}\). 
\end{corollary}

\begin{proposition}[Estimate on \(\langle \mathrm{A}_s(g \nabla \phi) , \mathrm{A}_s \phi \rangle\)]\label{prop:Estimate on〈Aₛ(g∇ϕ),Aₛϕ〉}
    Fix \(s \geq 1\) and let \( g, \phi \in C^\infty(X)\) such that \(g, \phi|_{y = 0,1} = 0\). Then we have 
    
    \[   \lvert  \langle  \mathrm{A}_s(g \partial_x \phi) , \mathrm{A}_s \phi \rangle \rvert  \lesssim  \lVert  \partial_y g \rVert_{\mathcal{G}^{2}_2} \lVert  \phi \rVert_{\mathcal{G}^{s + \frac{1}{4}}_2}^2 + \lVert  g \rVert_{\mathcal{G}^{s}_2}\lVert  \phi \rVert_{\mathcal{G}^{s}_2} \lVert  \partial_y \phi \rVert_{\mathcal{G}^{2}_2},
    \]  
    \[   \lvert  \langle  \mathrm{A}_s(g \partial_y \phi) , \mathrm{A}_s \phi \rangle \rvert  \lesssim  \lVert  \partial_y g \rVert_{\mathcal{G}^{2}_2} \lVert  \partial_y \phi \rVert_{\mathcal{G}^{s}_2} \lVert  \phi \rVert_{\mathcal{G}^{s}_2}  + \lVert  g \rVert_{\mathcal{G}^{s}_2} \lVert  \partial_y \phi \rVert_{\mathcal{G}^{2}_2} \lVert  \partial_y \phi\rVert_{\mathcal{G}^{s}_2} + \lVert  \partial_{yy} g \rVert_{\mathcal{G}^{2}_2} \lVert  \phi \rVert_{\mathcal{G}^{s}_2}^2,
    \]   
\end{proposition} 
\begin{proof}  
    Using integration by parts, we have 
    \begin{align*}
            \lvert  \langle  \mathrm{A}_s (g \partial_x \phi) , \mathrm{A}_s \phi \rangle \rvert &\leq  \lvert \left\langle  [ \mathrm{A}_s, g \partial_x] \phi , \mathrm{A}_s \phi \right\rangle \rvert   +  \lvert  \langle  \partial_x g , (\mathrm{A}_s \phi)^2 \rangle \rvert \leq  \lvert  \langle  [ \mathrm{A}_s, g \partial_x] \phi , \mathrm{A}_s \phi \rangle \rvert + \lVert  \partial_x g \rVert_{L^\infty_{x,y}} \lVert  \phi \rVert_{\mathcal{G}^{s}_2}^2 , \\
            &\lesssim   \lvert \langle  [ \mathrm{A}_s , g \partial_x] \phi , \mathrm{A}_s \phi \rangle \rvert + \lVert \partial_y g \rVert_{\mathcal{G}^{2}_2} \lVert  \phi \rVert_{\mathcal{G}^{s}_2}^2  = : \mathrm{I} + \lVert  \partial_y g  \rVert_{\mathcal{G}^{2}_2} \lVert  \phi \rVert_{\mathcal{G}^{s}_2}^2. 
    \end{align*} 
    We split \(\mathrm{I}\) into high-low and low-high frequencies: 
    \begin{align*}
        \mathrm{I} &\lesssim  \int_{0}^1 \int_{ \lvert \xi - \eta \rvert \leq \frac{ \lvert \eta \rvert}{4}}   \lvert  \mathrm{A}_s (\xi ) - \mathrm{A}_s(\eta) \rvert  \lvert \hat{g}(\xi - \eta) \rvert  \lvert \eta \hat{\phi} \rvert  \lvert  \mathrm{A}_s \hat{\phi}(\xi) \rvert \\
        & \qquad\qquad  +  \int_{0}^1 \int_{ \lvert \xi - \eta \rvert \geq \frac{ \lvert \eta \rvert}{4}}   \lvert  \mathrm{A}_s (\xi ) - \mathrm{A}_s(\eta) \rvert  \lvert \hat{g}(\xi - \eta) \rvert  \lvert \eta \hat{\phi} \rvert  \lvert  \mathrm{A}_s \hat{\phi}(\xi) \rvert  \\
        &=: \mathrm{I}_{1} + \mathrm{I}_2. 
    \end{align*} 
    Then using Lemma \ref{lem:Commutator Estimate}, we have 
    \begin{align*}
        \mathrm{I}_1 &\lesssim   \int_{0}^1 \int_{ \lvert \xi - \eta \rvert \leq \frac{ \lvert \eta \rvert}{4}} \mathrm{A}_1 \hat{g}(\xi - \eta)  \mathrm{A}_{s + \frac{1}{4}}  \hat{\phi}(\eta) \mathrm{A}_{s + \frac{1}{4}} \hat{\phi} (\eta) \lesssim  \lVert  g \rVert_{\mathcal{G}^{2}_\infty} \lVert  \phi  \rVert_{\mathcal{G}^{s+\frac{1}{4}}_2}^2 , \\
        \mathrm{I}_2 &\lesssim \int_{0}^1 \int_{ \lvert \xi - \eta \rvert \geq \frac{ \lvert \eta \rvert}{4}}  \mathrm{A}_s \hat{g}(\xi - \eta) \mathrm{A}_2 \hat{\phi}(\eta) \mathrm{A}_s \hat{\phi}(\eta) \lesssim  \lVert  g \rVert_{\mathcal{G}^{s}_2} \lVert  \phi \rVert_{\mathcal{G}^{2}_\infty} \lVert  \phi \rVert_{\mathcal{G}^{s}_2},
    \end{align*}     
    where for the second line we have used 
    \[   \lvert \mathrm{A}_s(\xi) - \mathrm{A}_s(\eta) \rvert  \lesssim  e^{\delta \sqrt{\langle \xi - \eta \rangle}} e^{\delta \sqrt{ \langle \eta \rangle}}(\langle \xi - \eta \rangle^s + \langle \eta \rangle^s) \lesssim  \mathrm{A}_s(\xi - \eta) \mathrm{A}_2(\eta). 
    \] 
    These estimates along with Gagliardo-Nirenberg imply
    \[  \mathrm{I} \lesssim  \lVert \partial_y g \rVert_{\mathcal{G}^{2}_2} \lVert  \phi \rVert_{\mathcal{G}^{s + \frac{1}{4}}_2}^2 + \lVert  g \rVert_{\mathcal{G}^{s}_2} \lVert  \partial_y \phi \rVert_{\mathcal{G}^{2}_2} \lVert  \phi \rVert_{\mathcal{G}^{s}_2}.  
    \] 
    Similarly, using integration by parts, we have, 
    \begin{align*}
         \lvert  \langle  \mathrm{A}_s ( g \partial_y \phi), \mathrm{A}_s \phi \rangle \rvert &\leq  \lvert  \langle  [ \mathrm{A}_s, g \partial_y ] \phi, \mathrm{A}_s \phi \rangle \rvert +  \lvert  \langle  \partial_y g , (\mathrm{A}_s \phi)^2 \rangle \rvert \leq  \lvert  \langle  [ \mathrm{A}_s , g \partial_y] \phi , \mathrm{A}_s \phi \rangle \rvert + \lVert  \partial_y g \rVert_{L^\infty} \lVert  \phi \rVert_{\mathcal{G}^{s}_2}^2 \\
         &\lesssim   \lvert  \langle [ \mathrm{A}_s, g \partial_y ] \phi , \mathrm{A}_s \phi  \rangle\rvert + \lVert  \partial_{yy} g \rVert_{\mathcal{G}^{2}_2} \lVert  \phi \rVert_{\mathcal{G}^{s}_2}^2 =: \mathrm{J} + \lVert  \partial_{yy} g \rVert_{\mathcal{G}^{2}_2} \lVert  \phi \rVert_{\mathcal{G}^{s}_2}^2. 
    \end{align*} 
    Via the same approach as above, we have \( \mathrm{J} = \mathrm{J}_1 + \mathrm{J} _2\) and 
    \begin{align*}
        \mathrm{J} _1 &\lesssim  \int_{0}^1 \int_{ \lvert \xi - \eta \rvert \leq \frac{ \lvert \eta \rvert}{4}}  \mathrm{A}_1\hat{g}(\xi - \eta) \mathrm{A}_{s - \frac{1}{2}} \widehat{\partial_y \phi} (\eta) \mathrm{A}_s \hat{\phi}(\xi) \lesssim  \lVert  g \rVert_{\mathcal{G}^{2}_\infty} \lVert  \partial_y \phi \rVert_{\mathcal{G}^{s - \frac{1}{2}}_2} \lVert  \phi \rVert_{\mathcal{G}^{s}_2} ,\\
        \mathrm{J}_2 &\lesssim  \int_{0}^1 \int_{ \lvert \xi - \eta \rvert \geq \frac{ \lvert \eta \rvert}{4}}  \mathrm{A}_s \hat{g}(\xi - \eta) \mathrm{A}_2 \widehat{\partial_y \phi}(\eta) \mathrm{A}_s \hat{\phi}(\xi) \lesssim  \lVert  g \rVert_{\mathcal{G}^{s}_2} \lVert  \partial_y \phi \rVert_{\mathcal{G}^{2}_2} \lVert  \phi \rVert_{\mathcal{G}^{s}_\infty} . 
    \end{align*} 
    These estimates along with Gagliardo-Nirenberg imply 
    \[ \mathrm{J} \lesssim  \lVert  \partial_y g \rVert_{\mathcal{G}^{2}_2} \lVert  \partial_y \phi \rVert_{\mathcal{G}^{s}_2} \lVert  \phi \rVert_{\mathcal{G}^{s}_2}  + \lVert  g \rVert_{\mathcal{G}^{s}_2} \lVert  \partial_y \phi \rVert_{\mathcal{G}^{2}_2} \lVert  \partial_y \phi\rVert_{\mathcal{G}^{s}_2}. \qedhere
        \]
\end{proof}

\begin{lemma}[Commutator Estimate]\label{lem:Commutator Estimate}
    For \(s \geq 1\) and \(c \in (0,1)\), we have the following: 
    \begin{itemize}
        \item For \( \lvert \xi - \eta \rvert \leq c \lvert \eta \rvert :\)  
        \begin{equation}\label{eq:As for |x-e|<c|e|}
            \quad\quad\quad\lvert \mathrm{A}_{s}(\xi) - \mathrm{A}(\eta) \rvert  \lesssim  \mathrm{A}_{s - \frac{3}{4}} (\eta) \mathrm{A}_1(\xi - \eta) \langle \xi \rangle^{\frac{1}{4}}, 
        \end{equation} 
        \item For \( \lvert \eta \rvert \leq c\lvert \xi - \eta \rvert : \) 
        \begin{equation}\label{eq:As for |e|<c|x-e|}
            \lvert \mathrm{A}_s(\xi) - \mathrm{A}_s(\eta) \rvert  \lesssim  \mathrm{A}_{0}(\eta) \mathrm{A}_s(\xi - \eta) 
        \end{equation} 
        \item For \(c\lvert \eta \rvert \leq  \lvert \xi - \eta \rvert \leq c^{-1} \lvert \eta \rvert : \)  
        \begin{equation}\label{eq:As for e<e-x<e}
            \lvert \mathrm{A}_s(\xi) - \mathrm{A}_s(\eta) \rvert  \lesssim  \mathrm{A}_{s-1}(\eta) \mathrm{A}_1(\xi - \eta).  
        \end{equation} 
    \end{itemize} 
\end{lemma}
\begin{proof} 
    Recall the following properties of the Japanese bracket \( \langle x \rangle := \sqrt{1 +  \lvert x \rvert^2} : \) 
    \begin{itemize}
        \item For \( 0 < r \leq 1 : \)  
        \begin{equation}\label{eq:SquareBracket for r<1}
            \lvert  \langle x \rangle^r - \langle  y \rangle^r  \rvert \leq \langle  x - y \rangle^r
        \end{equation} 
        \item For general \(s \geq 0 : \)  
        \begin{equation}\label{eq:SquareBracket for s>0}
            \lvert  \langle x \rangle^s - \langle y \rangle^s \rvert  \lesssim _s \frac{1}{\langle x \rangle^{1-s} + \langle y \rangle^{1-s}} \langle  x - y \rangle. 
        \end{equation} 
        \item Combining the above two, we obtain for \(0 < r \leq 1 :  \)
        \begin{equation}\label{eq:ExpSquareBracket for r<1}
            \lvert e^{\langle x \rangle^r - \langle y \rangle^r} - 1 \rvert \leq  \lvert \langle x \rangle^r - \langle y \rangle^r \rvert e^{ \lvert \langle x \rangle^r - \langle y \rangle^r \rvert} \lesssim _r \frac{e^{\langle x - y \rangle^r} \langle  x -y  \rangle}{\langle x \rangle^{1 - r}  + \langle y \rangle^{1 - r}} 
       \end{equation} 
    \end{itemize}  
    For \( \lvert \xi - \eta \rvert \leq c \lvert \eta \rvert, \) triangle inequality gives us:
    \begin{align*}
         \lvert  \mathrm{A}_s(\xi) - \mathrm{A}_s(\eta) \rvert &\leq    \langle \xi \rangle^s  \lvert e^{(\delta _0 - \lambda) \sqrt{\langle \xi \rangle}} - e^{(\delta _0 - \lambda) \sqrt{\langle \eta \rangle}} \rvert + e^{(\delta _0 - \lambda ) \sqrt{\langle \eta \rangle}}  \lvert \langle \xi \rangle^s - \langle \eta \rangle^s \rvert \\
         \shortintertext{\text{Using \(\xi \sim \eta\) and \eqref{eq:SquareBracket for s>0}}}
         &\lesssim  \langle \eta \rangle^s e^{(\delta _0 - \lambda)\sqrt{\langle \eta \rangle}}  \lvert  e^{(\delta _0 - \lambda) (\sqrt{\langle \xi \rangle} - \sqrt{\langle \eta \rangle})}  - 1\rvert + e^{(\delta _0 - \lambda) \sqrt{\langle \eta \rangle}} \langle \eta \rangle^{s - 1} \langle  \xi - \eta \rangle \\
         \shortintertext{Using \eqref{eq:ExpSquareBracket for r<1} and gathering the terms:}
         &\lesssim e^{(\delta _0 - \lambda) \sqrt{\langle \eta \rangle}} \langle \eta \rangle^s \left(e^{(\delta _0 - \lambda)\sqrt{\langle  \xi -\eta \rangle}} \frac{  \langle \xi - \eta \rangle}{\sqrt{\langle \eta \rangle}} +  \frac{ \langle  \xi - \eta \rangle}{\langle \eta \rangle}\right)\\
         &\lesssim  e^{(\delta _0  - \lambda) \sqrt{\langle \eta \rangle}} \langle \eta \rangle^{s - \frac{1}{2}} e^{(\delta _0 - \lambda) \langle \sqrt{\xi - \eta} \rangle} \langle  \xi - \eta \rangle 
         \shortintertext{As \(\eta \sim \xi \) so we can transfer the derivatives from \(\eta \) onto \(\xi \) however many we want. For our purposes, the following will suffice:} 
         &\lesssim e^{(\delta _0 - \lambda) \sqrt{\langle \eta \rangle}} \langle \eta \rangle^{s - \frac{3}{4}} e^{(\delta _0 - \lambda) \sqrt{\langle \xi- \eta \rangle}} \langle \xi - \eta \rangle \langle \xi \rangle^{\frac{1}{4}} \lesssim \mathrm{A}_{s - \frac{3}{4}}(\eta) \mathrm{A}_{1}(\xi - \eta) \langle \xi \rangle^{\frac{1}{4}}.
    \end{align*} 
    
    The proofs for the remaining inequality are somewhat obvious. For \( \lvert \eta \rvert \leq c \lvert \xi - \eta \rvert, \) we have \(\xi \approx  \xi - \eta\) so: 
    \begin{align*}
         \lvert \mathrm{A}_s(\xi) - \mathrm{A}_s(\eta) \rvert &\lesssim  e^{\delta \sqrt{\langle \xi \rangle}} \langle \xi \rangle^s + e^{\delta \sqrt{\langle \eta \rangle}}\langle \eta \rangle^s \lesssim  e^{\delta \sqrt{\langle \xi - \eta \rangle}} e^{\delta \sqrt{\langle \eta \rangle}} \langle \xi -\eta \rangle^s  + e^{\delta \sqrt{\langle \eta \rangle}} \langle \xi - \eta \rangle^s \\
         & \lesssim  e^{\delta \sqrt{ \langle \xi - \eta \rangle}} \langle \xi - \eta \rangle^s e^{\delta \sqrt{\langle \eta \rangle}} \lesssim  \mathrm{A}_s(\xi - \eta) \mathrm{A}_0 (\eta) . 
    \end{align*} 
    And for \(c \lvert \eta \rvert \leq  \lvert \xi - \eta \rvert \leq c^{-1} \lvert  \eta \rvert\), we have 
    \begin{align*}
         \lvert \mathrm{A}_s(\xi) - \mathrm{A}_s(\eta) \rvert &\lesssim  e^{\delta \sqrt{\langle \xi \rangle}} \langle \xi \rangle^s + e^{\delta \sqrt{\langle \eta \rangle}}\langle \eta \rangle^s \lesssim  e^{\delta \sqrt{ \langle \xi - \eta \rangle}} e^{\delta \sqrt{ \langle \eta \rangle}} ( \langle \xi - \eta \rangle^s + \langle \eta \rangle^s) \\
         &  \lesssim  e^{\delta \sqrt{\langle \xi - \eta \rangle}} e^{\delta \sqrt{ \langle \eta \rangle}} \langle \eta \rangle^{s-1} \langle \xi - \eta \rangle.  \qedhere
    \end{align*}  
\end{proof}    

\printbibliography

\end{document}